\documentclass[a4paper]{amsart}

\RequirePackage{amsmath} 
\RequirePackage{amssymb}
\usepackage{amscd,latexsym,amsthm,amsfonts,amssymb,amsmath,amsxtra}
\usepackage[colorlinks=true,urlcolor=blue,citecolor=blue]{hyperref}
\usepackage{color}
\usepackage[all]{xy}
\usepackage{bm}
\usepackage{mathtools}
\usepackage{ mathrsfs }
\usepackage{xcolor}
\usepackage{comment}
\usepackage{marginnote}
\usepackage{enumitem}
\usepackage{thmtools, thm-restate}

\let\Re\undefined
\let\Im\undefined

\DeclareMathOperator{\Re}{Re}
\DeclareMathOperator{\Im}{Im}

\DeclareMathOperator{\Tr}{Tr}

\DeclareMathOperator{\vol}{vol}

	\newcommand{\Res}{\operatorname{Res}}
	
	\newcommand{\Eis}{\operatorname{Eis}}

	\newcommand{\Ad}{\operatorname{Ad}}

	\newcommand{\fin}{\operatorname{fin}}

	\newcommand{\diag}{\operatorname{diag}}
	\newcommand{\cusp}{\operatorname{cusp}}

	\newcommand{\Vol}{\operatorname{Vol}}

	\newcommand{\RNum}[1]{\uppercase\expandafter{\romannumeral #1\relax}}

\begin{document}
\theoremstyle{plain}
\newtheorem{thm}{Theorem}[section]
	
\newtheorem{cor}[thm]{Corollary}

\newtheorem{thmx}{Theorem}
\renewcommand{\thethmx}{\Alph{thmx}} 

\newtheorem{hy}[thm]{Hypothesis}
\newtheorem*{thma}{Theorem A}
\newtheorem*{corb}{Corollary B}
\newtheorem*{thmc}{Theorem C}
\newtheorem{lemma}[thm]{Lemma}  
\newtheorem{prop}[thm]{Proposition}
\newtheorem{conj}[thm]{Conjecture}
\newtheorem{fact}[thm]{Fact}
\newtheorem{claim}[thm]{Claim}
	
\theoremstyle{definition}
\newtheorem{defn}[thm]{Definition}
\newtheorem{example}[thm]{Example}
\theoremstyle{remark}
	
\newtheorem{remark}[thm]{Remark}	
\numberwithin{equation}{section}
	
\title[]{A Non-Split Motohashi Formula for Toric Periods over Number Fields}%
\author{Liyang Yang} 
\hypersetup{pdftitle={A Non-Split Motohashi Formula for Toric Periods over Number Fields},pdfauthor={Liyang Yang},pdfsubject={Automorphic forms and L-functions},pdfkeywords={Motohashi formula, toric periods, spectral reciprocity, class group L-functions, nonvanishing, subconvexity}}

\address{Department of Mathematics, Texas A\&M University, College Station,  TX 77843, USA} 
\email{liyangy@tamu.edu} 

\begin{abstract}
We establish a non-split Motohashi formula for toric Eisenstein periods over an
arbitrary number field $F$. For a quadratic extension $E/F$, we compare toric
Fourier inversion for a product of two Eisenstein series with the
$\mathrm{GL}_2/F$-spectral expansion of its Rankin--Selberg kernel. The kernel
regularizes the product in a cone of absolute convergence. Eisenstein
corrections and contour residues give the geometric and residual terms in
the meromorphically continued identity.
Suitable test data give twisted first- and second-moment formulas for the
completed $L$-functions of everywhere unramified toric characters.  Over
general $F$ these provide an adelic counterpart of the moment formulas of
Duke--Friedlander--Iwaniec that does not use the geometry of Heegner points.  In the CM case, we prove the absolute nonvanishing proportion
$4.9\times10^{-4}$ for every totally real base field and all sufficiently
large CM discriminants, with the threshold depending on the base field
and no restriction on the number of ramified primes.
We incorporate inverse Euler factors at small ramified primes into the
mollifier. Local unit-group tests control the resulting moments and remove
the ramified Euler-factor loss.
\end{abstract}
	
\date{}
\maketitle
\tableofcontents

\section{Introduction}
Moments of $L$-functions play a central role in analytic number theory and its arithmetic applications. A landmark example is Motohashi's formula \cite[Theorem 4.2]{Mot97}, which connects the fourth moment of the Riemann zeta function with the cubic moment of $\mathrm{PGL}_2$ $L$-functions:
\begin{equation}\label{1.1}
\int_{\mathbb{R}}|\zeta(1/2+it)|^4h(t)dt\rightsquigarrow \sum_{\text{$\phi$: $\mathrm{PGL}_2$-form of level $1$}}L(1/2,\phi)^3\check{h}(t_{\phi})+\text{Eisenstein},
\end{equation}
an identity whose use in both directions has led to many important developments. Analogues for Dirichlet $L$-functions were later established in \cite{BHKM20, Kan22}.

The formula \eqref{1.1} may be viewed as a split phenomenon, since $\zeta(s)^2$ is the Dedekind zeta function of the split quadratic algebra. In this paper, we develop a non-split analogue for a general quadratic extension $E/F$ of number fields. More precisely, we prove a non-split Motohashi formula for toric Eisenstein periods over $\mathcal{C}_{E/F}:=E^\times\mathbb A_F^\times\backslash\mathbb A_E^\times$, realized as a spectral reciprocity identity relating moments over characters of $\mathcal{C}_{E/F}$ to dual $\mathrm{GL}_2/F$-spectral moments, with explicit geometric and residual terms.

With suitable test data, the identity specializes to twisted first- and
second-moment formulas for
class group $L$-functions.  In the imaginary quadratic case these recover the
moment structure of Duke--Friedlander--Iwaniec \cite{DFI95}; over a general base
field they arise from an adelic spectral decomposition. The new feature is
the compatibility of the Rankin--Selberg kernel
regularization of \cite{Yan26a} with non-split toric Fourier analysis.
Fourier inversion identifies the toric-character moment, while the
$\mathrm{GL}_2/F$-spectral expansion expresses it as the dual automorphic moment.
Eisenstein corrections and contour residues account for the main terms;
see \textsection\ref{sec2.kernel} for the former.

Blomer--Brumley--Radziwi{\l}{\l}
\cite[\textsection 2.4, Remark~3.6, and \textsection 4]{BBR26}
obtain related spectral identities over $\mathbb Q$ for imaginary quadratic
fields, treating cuspidal forms and incomplete Eisenstein series, and
identify their non-split Motohashi structure. Here the Rankin--Selberg
kernel regularization allows us to work with genuine Eisenstein series at
fixed spectral parameters over a general number field, retaining the
explicit degenerate and residue terms needed for central-value moment
formulas.

The Fourier spectral reciprocity of \cite{Yan26c} compares nonsplit toric
periods of a single Eisenstein series with split Hecke periods and yields
first-moment identities. Here the regularized $\mathrm{GL}_2$-spectral
decomposition of an Eisenstein product yields the second-moment architecture
of a Motohashi formula.

For the arithmetic applications we distinguish two levels of generality.  The
reciprocity and shifted moment identities hold for arbitrary quadratic $E/F$.
For CM extensions of a fixed totally real field, the unramified toric family is
finite and the moment formulas yield a uniform positive proportion.
We modify Blomer's Selberg mollifier \cite{Blo04} to remove its ramified
Euler-factor loss; see \textsection\ref{sec9}. Earlier work over CM fields gave quantitative
nonvanishing by equidistribution or Eisenstein-series methods
\cite{Mas10,Yan20}.

For $F=\mathbb Q$, the Katz--Sarnak philosophy predicts symplectic symmetry
and hence $100$ percent central nonvanishing in the limit \cite{KS99}.
Fouvry and Iwaniec established the corresponding symplectic density of
low-lying zeros under the generalized Riemann hypothesis after averaging over
the discriminant \cite{FI03}.  From this perspective, uniformity in the
discriminant is part of the expected phenomenon, rather than merely a local
refinement.  Theorem~\ref{thmC} gives an unconditional fixed positive
proportion for each sufficiently large CM relative class group family,
including extensions with arbitrarily many ramified primes.  Khayutin
\cite{Kha20} proves quantitative nonvanishing for class group $L$-functions
over general number fields under a small-regulator hypothesis; his lower
bound does not give a fixed positive proportion.

\subsection{The Non-Split Motohashi Formula}
Let $E/F$ be a quadratic extension of number fields, and let $\eta_{E/F}$ be the associated quadratic character. Fix an embedding $\iota: E^{\times}\hookrightarrow  \mathrm{GL}_2(F)$, regarding $E^{\times}$ as a nonsplit torus. For a Bruhat--Schwartz function $\Phi$ on $\mathbb{A}_F^{\oplus 2}$, $\Re(s)\gg 1$, and a character $\xi$ of $E^{\times}\backslash\mathbb{A}_E^{\times}$, we define
\begin{align*}
Z(s,\Phi,\xi)=\int_{\mathbb{A}_E^{\times}}
\Phi((0,1)\iota(x))\xi(x)|x|_{\mathbb{A}_E}^{1/2+s}d^{\times}x.
\end{align*}

Then $Z(s,\Phi,\xi)$ is a Tate integral representing the completed Hecke $L$-function $\Lambda(1/2+s,\xi)$. 

For $1\leq i\leq 2$, let $\xi_i$ be a unitary character of $E^{\times}\backslash\mathbb{A}_E^{\times}$, put $\omega_i=(\xi_i|_{\mathbb{A}_F^{\times}})^{-1}$, and let $\Phi_i$ be a Bruhat--Schwartz function on $\mathbb{A}_F^{\oplus 2}$. We take $\chi_i=\mathbf 1$ in the Eisenstein data associated with $(\Phi_i,\omega_i)$. We obtain the following non-split Motohashi formula for toric periods. 
\begin{thmx}\label{thmA}
Let $\boldsymbol{\Phi}:=(\Phi_1, \Phi_2)$ and let $\mathbf{s}=(s_1,s_2)\in \mathbb{C}^2$ with $|\Re(s_1)|+|\Re(s_2)|<1/2$. Then 
\begin{multline}\label{e1.2}
\sum_{\chi\in \widehat{\mathcal{C}}_{E/F}}\frac{Z(s_1,\Phi_1,\xi_1\chi)Z(s_2,\Phi_2,\xi_2\overline{\chi})}{2L(1,\eta_{E/F})}=Z(\mathbf{s},\boldsymbol{\Phi},\xi_1\xi_2)+ Z^{\Diamond}(\mathbf{s},\boldsymbol{\Phi},\xi_1\xi_2)\\
+\mathcal{P}_{\cusp}(\mathbf{s},\boldsymbol{\Phi},\xi_1\xi_2)
+\mathcal{P}_{\Eis}(\mathbf{s},\boldsymbol{\Phi},\xi_1\xi_2)
-
\sum_{i=1}^{4}\widetilde{\Psi}_{\Res}^{(i)}(\mathbf{s},\boldsymbol{\Phi},\xi_1\xi_2),
\end{multline}	
where $\widehat{\mathcal C}_{E/F}$ is the Pontryagin dual of
$\mathcal C_{E/F}$. Here $Z$ and $Z^\Diamond$ are the geometric terms,
related by intertwining; $\mathcal P_{\cusp}$ and $\mathcal P_{\Eis}$ are
the cuspidal and continuous spectral moments, with Rankin--Selberg
coefficients paired with toric periods; and
$\widetilde\Psi_{\Res}^{(i)}$ are the four contour residues.
\end{thmx}

A precise form of Theorem \ref{thmA} is given in Theorem \ref{thm2.3} in \textsection\ref{sec2.6}. 

\subsection{Twisted Moments of Class Group \texorpdfstring{$L$}{L}-functions}
By specifying the automorphic data $\xi_i$ and $\Phi_i$, $1\leq i\leq 2$, in \eqref{e1.2}, we derive the following twisted moment formulas for class group $L$-functions. 

\subsubsection{The Twisted Second Moment}
Let $\widehat{\mathcal{C}}_{E/F}^{\text{un}}$ be the subgroup of $\widehat{\mathcal{C}}_{E/F}$ consisting of everywhere unramified characters. Let $\mathfrak{n}$ be an integral ideal whose prime divisors split in $E$ and which is coprime to $2\mathfrak{D}_{E/F}\mathfrak{D}_F$. For $\chi\in \widehat{\mathcal{C}}_{E/F}^{\text{un}}$, let $\lambda_{\chi}^*(\mathfrak{n})$ be the renormalized Hecke weight defined precisely in \eqref{f3.4}. 
\begin{thmx}\label{thmB}
Let $|\Re(s_1)|+|\Re(s_2)|<1/2$. Then
\begin{equation}\label{e1.3}
\sum_{\chi\in\widehat{\mathcal C}_{E/F}^{\mathrm{un}}}\frac{\lambda_\chi^*(\mathfrak n)\Lambda(1/2+s_1,\chi)\Lambda(1/2+s_2,\overline\chi)}{L(1,\eta_{E/F})}=\mathscr M(\mathbf s;\mathfrak n)+\mathcal E_{\mathrm{Spec}}(\mathbf s;\mathfrak n),
\end{equation}
as an identity of meromorphic functions. Here $\mathscr M$ is the explicit
main term and $\mathcal E_{\mathrm{Spec}}$ is the cuspidal and continuous
spectral contribution; see \textsection\ref{sec7}. At $s_1=s_2=0$,
\begin{equation}\label{e1.3-central}
\sum_{\chi\in\widehat{\mathcal C}_{E/F}^{\mathrm{un}}}\frac{\lambda_\chi^*(\mathfrak n)|\Lambda(1/2,\chi)|^2}{L(1,\eta_{E/F})}=\mathscr M^\heartsuit(\mathfrak n)+O_{F,\varepsilon}(D_E^{5/12+\varepsilon}N_F(\mathfrak n)^{2/3+\varepsilon}).
\end{equation}
where $\mathscr M^\heartsuit(\mathfrak n)$ denotes the central value of
$\mathscr M$.
\end{thmx}

\begin{remark}
When $F=\mathbb{Q}$ and $E$ is imaginary quadratic, the central formula
\eqref{e1.3-central} specializes to the setting of
Duke--Friedlander--Iwaniec \cite{DFI95}.  In the normalization used here, its
error is $D_E^{1/2-1/12+\varepsilon}N_F(\mathfrak n)^{2/3+\varepsilon}$,
whereas the error underlying \cite[Theorem~4]{DFI95} is
$D_E^{1/2-1/28+\varepsilon}N_F(\mathfrak n)^{6+\varepsilon}$.
\end{remark}

\begin{remark}
Formula \eqref{e1.3} explains the structure of the main terms in the DFI moment formula. Our proof is fundamentally different from that of \cite{DFI95}: the two Eisenstein corrections and the two paired residue contributions in \eqref{eq6.res12}--\eqref{eq6.res34} combine into the symmetric four-term expression \eqref{eq7.main}. The remainder consists entirely of the cuspidal and continuous spectral contributions. This organization explains the ``lucky fact'' observed by Blomer \cite[\textsection 3]{Blo04}, namely the cancellation between certain diagonal and off-diagonal contributions in the second moment.
\end{remark}

\begin{remark}
Under a hypothesis on the supply of small split primes, a DFI-type
amplifier applied to Theorem~\ref{thmB} gives subconvexity over quadratic
extensions of arbitrary number fields; see Theorem~\ref{thm10.subconvexity}
in \textsection\ref{sec10}. This extends the second-moment approach of
\cite[Theorem~6]{DFI95} beyond imaginary quadratic fields, with a conditional
saving approaching $1/88$, compared with $1/1156$ there.
\end{remark}

\subsubsection{The Twisted First Moment}
Our construction of the test functions $\Phi_1$ and $\Phi_2$ also yields the following twisted first moment formula for class group $L$-functions. 
\begin{thm}\label{thmFirstMomentIntro}
Let $m\geq 10$ and $|\Re(s_2)|\leq 1$. We have
\begin{equation}\label{e1.4}
\sum_{\chi\in \widehat{\mathcal{C}}_{E/F}^{\text{un}}}\frac{\lambda_{\chi}^*(\mathfrak{n})\Lambda(1/2+s_2,\chi)}{L(1,\eta_{E/F})}=\mathcal{M}(s_2;\mathfrak{n})+\mathcal{M}^{\Diamond}(s_2;\mathfrak{n})+O(D_E^{-m/2}N_F(\mathfrak{n})^{m+5}),
\end{equation}
as an identity of meromorphic functions in $s_2$. The implied constant depends only on $F$ and $m$. The main terms $\mathcal{M}(s_2;\mathfrak{n})$ and $\mathcal{M}^{\Diamond}(s_2;\mathfrak{n})$ are explicit multiples of $D_E^{1/2}N_F(\mathfrak{n})^{-1/2-s_2}\Lambda_F(1+2s_2)$ and $D_E^{1/2-s_2}N_F(\mathfrak{n})^{-1/2+s_2}\Lambda_F(1-2s_2)$, respectively; see Theorem \ref{thm8.1}.
\end{thm}

\begin{remark}
When $F=\mathbb{Q}$, the formula \eqref{e1.4} is classical: it follows from the
Fourier expansion of Eisenstein series, Hecke's formula, and ideal class
decompositions. Over a general number field, this approach becomes considerably
more delicate. Our adelic proof combines toric Fourier inversion with
the $\mathrm{GL}_1$-spectral expansion of the nonconstant term. This
first-moment argument may be viewed as a simplified specialization of the
Fourier spectral reciprocity formula in \cite{Yan26c}.
\end{remark}

\subsection{Nonvanishing of Class Group \texorpdfstring{$L$}{L}-functions}
Let $F$ be totally real and let $E/F$ be CM. Put
$\mathrm{Cl}_{E/F}:=\mathrm{Cl}_E/\operatorname{Im}(\mathrm{Cl}_F)$. Then
$\widehat{\mathcal C}_{E/F}^{\mathrm{un}}\simeq
\widehat{\mathrm{Cl}}_{E/F}$ is a finite family.

\begin{thmx}\label{thmC}
For every totally real number field $F$, there is a threshold $D_0(F)$
such that every CM quadratic extension $E/F$ with $D_E\geq D_0(F)$ satisfies
\begin{equation}\label{e1.5}
\frac{\#\big\{\chi\in \widehat{\mathrm{Cl}}_{E/F}:
L(1/2,\chi)\neq 0\big\}}{|\mathrm{Cl}_{E/F}|}
\geq 4.9\times10^{-4}.
\end{equation}
\end{thmx}

This lower bound is independent of $F$; see
Proposition~\ref{prop.explicit.basefield} in \textsection\ref{sec9}.
The threshold $D_0(F)$ is ineffective because our proof uses Siegel's
lower bound for $L(1,\eta_{E/F})$. The bound is
uniform in the relative discriminant, with no restriction on the number of
prime ideals dividing $\mathfrak D_{E/F}$. To our knowledge, such an
unconditional uniform positive-proportion bound was not previously known
for these families, even for $F=\mathbb Q$, where Blomer's lower bound
contains a discriminant-dependent product over the ramified primes.

We modify Blomer's split-prime mollifier by incorporating inverse Euler
factors at small ramified primes. Local unit-group tests realize these
factors in both moment formulas, allowing their main terms and spectral
errors to be controlled as the set of primes grows. Together with a bound
for the contribution of the remaining ramified primes, this removes the
ramified-prime product from the nonvanishing bound; see
Lemma~\ref{lem9.unit.cutoff} and \eqref{eq9.explicit.local}.

\begin{remark}
For a non-CM quadratic extension, the moment formulas naturally sum an infinite
Archimedean spectrum with smooth Gamma weights. A positive-proportion result
for characters ordered by Archimedean conductor would require an additional
localization argument, which is not developed here.
\end{remark}

\subsection{Mechanism of the Motohashi Formula}
For simplicity, we describe the mechanism when $\xi_1=\xi_2=\mathbf{1}$; the general twists are treated in the same way. For $1\leq i\leq 2$, let $E_i(\cdot,s_i)$ be the Eisenstein series attached to the induced representation $|\cdot|^{s_i}\boxplus |\cdot|^{-s_i}$ and constructed from a Bruhat--Schwartz function $\Phi_i$ on $\mathbb{A}_F^{\oplus 2}$. The starting point is the toric period
\begin{equation}\label{1.8}
J(\mathbf{s},\boldsymbol{\Phi}):=\int_{\mathcal{C}_{E/F}}E_1(\iota(x),s_1)E_2(\iota(x),s_2)d^{\times}x. 
\end{equation}

The Motohashi formula arises by evaluating \eqref{1.8} in two different ways. On the one hand, one may expand the function $x\mapsto E_1(\iota(x),s_1)$ spectrally on the compact abelian group $\mathcal{C}_{E/F}$. By Fourier inversion and Plancherel on $\mathcal{C}_{E/F}$, this gives the left-hand side of \eqref{e1.2}, namely a moment over the characters $\chi\in \widehat{\mathcal{C}}_{E/F}$.

On the other hand, one first writes the Eisenstein product, in a cone of
absolute convergence, as the sum of the Rankin--Selberg kernel and explicit
Eisenstein corrections, and then expands the kernel along the genuine
$\mathrm{GL}_2$-spectrum. Unfolding the spectral coefficients and applying
Waldspurger's formula to the toric periods give the dual $L$-function moment.
With the CM class group test data, the central reciprocity and its local
spectral weight estimates have the schematic form
\begin{multline*}
\sum_{\chi\in\widehat{\mathrm{Cl}}_{E/F}}
\lambda_\chi^*(\mathfrak n)|L(\tfrac12,\chi)|^2
\rightsquigarrow \text{main term}
+L(1,\eta)D_E^{1/4}
\sum_{\mathfrak m\mid\mathfrak n}N_F(\mathfrak m)^{-1/2}\\
\times\sum_{\pi\in\mathcal F(\mathfrak m,\mathbf1)}
L(\tfrac12,\pi_E)^{1/2}|L_F(\tfrac12,\pi)|^2
+\text{Eisenstein contribution}.
\end{multline*}
Here $\mathcal F(\mathfrak m,\mathbf1)$ is the cuspidal family of conductor
$\mathfrak m$ and trivial central character; Archimedean spectral weights
and $\varepsilon$-losses are suppressed. Whereas the first-moment reciprocity
of \cite{Yan26c} has a split Hecke-character dual family, the present
second-moment formula produces a genuine $\mathrm{GL}_2/F$-spectral moment
carrying a central base-change factor. The factor $D_E^{1/4}$,
arising from the local spectral weights, contrasts with the $D_E^{1/2}$
discriminant scale of the main term. Combining this structure with the
spectral estimates gives the power-saving error in Theorem~\ref{thmB};
see \textsection\ref{sec5}.

Theorem \ref{thmA} compares these two expansions of the same toric period. It is a non-split analogue of the strategy of Michel--Venkatesh \cite{MV10} for Motohashi's formula, realized by Nelson \cite{Nel19}. The main analytic tasks are regularization and the construction of test functions suitable for arithmetic applications.

\subsubsection{Analytic Regularization}
The function $g\mapsto E_1(g,s_1)E_2(g,s_2)$ is not square-integrable, and hence its spectral expansion is not available without regularization. This is one of the basic analytic difficulties behind Motohashi-type formulas. In the setting of Duke--Friedlander--Iwaniec \cite{DFI95}, it is also one of the barriers to a direct application of Duke's equidistribution theorem. This difficulty was later addressed by Templier \cite{Tem11} through a careful regularization, different from Zagier's method \cite{Zag82}. However, such equidistribution methods are not well suited to the twisted sums appearing in \cite[Theorem 4]{DFI95}, which are needed for the amplified second moment and the subconvexity application. Moreover, they are specific to CM points and do not directly extend to arbitrary quadratic extensions $E/F$.

We use the Rankin--Selberg kernel regularization and continuation of
\cite[\textsection 2.2, \textsection 3.2, and \textsection 3.4]{Yan26a}. Its $L^1\cap L^2$ bounds
permit spectral inversion, while meromorphic continuation produces the
four contour residues. \textsection\ref{sec2} establishes compatibility with
the non-split toric period and proves Theorem~\ref{thmA}.

\subsubsection{Construction of Test Functions}
After regularization, Theorem \ref{thmA} is an abstract reciprocity formula. To obtain the arithmetic applications, one must choose the automorphic data $\Phi_1$ and $\Phi_2$ so that the toric periods become the desired class group $L$-functions and Hecke twists. This is the second main technical point.

We construct test functions selecting everywhere unramified toric characters. The construction is adelic and local. At the Archimedean places, a translate of a spherical Schwartz function is determined by the choice of $\tau$ with $E=F(\sqrt{\tau})$. At the finite places, the test functions encode the Hecke weights in the first and second moments.

\subsection{Further Applications}
Theorem~\ref{thmA} allows nontrivial twists $\xi_1,\xi_2$, giving mixed
moments of twisted class group $L$-functions. Suitable finite and
Archimedean test vectors should also accommodate algebraic anticyclotomic
characters, as in \cite{Vat02,Cor02}, and families obtained by fixing $E/F$
and varying the conductor. We leave these applications to subsequent work.

\subsection{Outline of the Paper}
\begin{itemize}

\item In \textsection\ref{sec2} we recall the Rankin--Selberg kernel decomposition and its geometric and spectral expansions established in \cite{Yan26a}, and prove their compatibility with the non-split toric period. This gives the non-split Motohashi formula in Theorem \ref{thm2.3}; see \textsection\ref{sec2.6}. 

\item In \textsection\ref{sect3} we construct the test functions and compute the relevant local period integrals on the spectral side. 

\item \textsection\ref{sec4} evaluates the geometric terms $Z(\mathbf{s},\boldsymbol{\Phi})$ and $Z^{\Diamond}(\mathbf{s},\boldsymbol{\Phi})$. \textsection\ref{sec5} estimates the spectral terms $\mathcal{P}_{\cusp}(\mathbf{s},\boldsymbol{\Phi})$ and $\mathcal{P}_{\Eis}(\mathbf{s},\boldsymbol{\Phi})$, while \textsection\ref{sec6} computes the residual terms $\widetilde{\Psi}_{\Res}^{(i)}(\mathbf{s},\boldsymbol{\Phi})$. \textsection\ref{sec7} assembles these calculations to prove the twisted second moment formula in Theorem \ref{thmB}. 

\item In \textsection\ref{sec8} we derive the twisted first moment formula, proving Theorem \ref{thmFirstMomentIntro}. 

\item In \textsection\ref{sec9} we combine the moment formulas from \textsection\ref{sec7} and \textsection\ref{sec8} to prove Theorem~\ref{thmC}. The conditional subconvexity application is treated in \textsection\ref{sec10}.
\end{itemize}

\subsection{Notation}\label{notation}
\subsubsection{Number Fields and Measures}\label{1.1.1}

Let $F$ be a number field with ring of integers $\mathcal{O}_F$, degree $d=[F:\mathbb{Q}]$, absolute norm $N_F$, and different $\mathfrak{D}_F$. Let $r_1$ and $r_2$ denote the numbers of real and complex places of $F$, respectively.

Let $\mathbb{A}_F$ be the adele ring of $F$. Let $\Sigma_F$ be the set of places of $F$, and let $\Sigma_{F,\fin}$ and $\Sigma_{F,\infty}$ denote the sets of non-Archimedean and Archimedean places, respectively. For $v\in \Sigma_F$, let $F_v$ be the corresponding local field and $|\cdot|_v$ the normalized absolute value. If $v$ is non-Archimedean, let $\mathcal{O}_v$ be the ring of integers of $F_v$ and $\mathfrak{p}_v$ its maximal ideal. Given an integral ideal $\mathfrak{n}$, we write $v\mid \mathfrak{n}$ if $\mathfrak{n}\subseteq \mathfrak{p}_v$. Fix a uniformizer $\varpi_v\in\mathfrak{p}_v$, and let $e_v(\cdot)$ denote the valuation normalized by $e_v(\varpi_v)=1$. Let $d_{F_v}:=e_v(\mathfrak{D}_F)$, and let $q_v$ be the cardinality of $\mathcal{O}_v/\mathfrak{p}_v$. In calculations involving $\mathbb{A}_F^{\times}$ or $F^{\times}\backslash \mathbb{A}_F^{\times}$, we simply write $|\cdot|$ for $|\cdot|_{\mathbb{A}_F}:=\prod_{v\in \Sigma_F}|\cdot|_v$. For an integral ideal $\mathfrak{n}$, let $N_F(\mathfrak{n})=N_{F/\mathbb Q}(\mathfrak{n})=\#(\mathcal O_F/\mathfrak{n})$ denote its absolute norm.

For $\Re(s)>1$, let $\zeta_F(s)=\prod_{v\in \Sigma_{F,\fin}}\zeta_{F_v}(s)$
be the finite Dedekind zeta function, where $\zeta_{F_v}(s)$ denotes the local Dedekind zeta factor. Denote by $\Lambda_F(s)=\prod_{v\in \Sigma_F}\zeta_{F_v}(s)$
the completed Dedekind zeta function. We identify both $\zeta_F(s)$ and $\Lambda_F(s)$ with their meromorphic continuations.

Let $\psi=\prod_{v}\psi_v$ be the additive character $\psi=\psi_{\mathbb{Q}}\circ \Tr_F$ of $F\backslash\mathbb{A}_F$, where $\Tr_F$ is the trace map and $\psi_{\mathbb{Q}}$ is the additive character of $\mathbb{Q}\backslash\mathbb{A}_{\mathbb{Q}}$ given by $x\mapsto e^{2\pi i x}$ on $\mathbb{R}$. For $v\in \Sigma_F$, let $dt_v$ be the additive Haar measure on $F_v$ which is self-dual with respect to $\psi_v$. Then $dt=\prod_{v}dt_v$ is the standard Tamagawa measure on $\mathbb{A}_F$. Let
\begin{align*}
&d^{\times}t_v=|t_v|_v^{-1}dt_v, \ \ \ v\in \Sigma_{F,\infty},\\
&d^{\times}t_v=\zeta_{F_v}(1)|t_v|_v^{-1}dt_v, \ \ \ v\in \Sigma_{F,\fin}.
\end{align*}
In particular, $\Vol(\mathcal{O}_v^{\times},d^{\times}t_v)=\Vol(\mathcal{O}_v,dt_v)=q_v^{-d_{F_v}/2}$ 
for all finite places $v$. Moreover,
\begin{align*}
\Vol(F\backslash\mathbb{A}_F; dt)=1,\quad
\Vol(F^{\times}\backslash\mathbb{A}_F^{(1)},d^{\times}t)=\underset{s=1}{\Res}\ \zeta_F(s),
\end{align*}
where $\mathbb{A}_F^{(1)}$ is the subgroup of ideles in $\mathbb{A}_F^{\times}$ of norm $1$. Write $C_F:=F^{\times}\backslash\mathbb{A}_F^{(1)}$, denote its Pontryagin dual by $\widehat C_F$, and write $\rho_F=\underset{s=1}{\Res}\,\zeta_F(s)$. We extend characters in $\widehat C_F$ trivially across the norm factor when viewing them as Hecke characters.

Let $E/F$ be a quadratic extension and let $\eta=\eta_{E/F}$ be the associated quadratic character. Write $\mathfrak D_{E/F}\subseteq\mathcal O_F$ for the relative discriminant ideal, $D_F=|\operatorname{disc}(F)|$ and $D_E=|\operatorname{disc}(E)|$, and write $D_{E/F}=N_F(\mathfrak D_{E/F})$, so that $D_E=D_F^2D_{E/F}$. We use the analogous conventions for measures over $E$. For $w\in \Sigma_{E,\fin}$, let $\mathcal{O}_{E_w}$ denote the ring of integers of $E_w$, and let $d_{E_w}:=e_w(\mathfrak{D}_E)$ be the valuation of the different ideal of $E$.
	
For a finite place $v$ of $F$, write $E_v=\prod_{w\mid v}E_w$ and set
$d_{E_v}:=\sum_{w\mid v}f(w/v)d_{E_w}$, where $q_w=q_v^{f(w/v)}$.
Thus $q_v^{d_{E_v}}=\prod_{w\mid v}q_w^{d_{E_w}}$ and
$\vol(\mathcal O_{E_v}^{\times})=q_v^{-d_{E_v}/2}$ for the product measure.

Partition the Archimedean places as follows:
\begin{equation}\label{1.2}
\begin{cases}
\Sigma_{F,\infty}^{\mathrm{r,ns}}:=\big\{v\in \Sigma_{F,\infty}:\ F_v\simeq \mathbb{R},\  E\otimes_FF_v\simeq \mathbb{C}\big\},\\
\Sigma_{F,\infty}^{\mathrm{r,sp}}:=\big\{v\in \Sigma_{F,\infty}:\ F_v\simeq \mathbb{R},\  E\otimes_FF_v\simeq \mathbb{R}\times \mathbb{R}\big\},\\
\Sigma_{F,\infty}^{\mathrm{c}}:=\big\{v\in \Sigma_{F,\infty}:\ F_v\simeq \mathbb{C},\  E\otimes_FF_v\simeq \mathbb{C}\times \mathbb{C}\big\}.
\end{cases} 
\end{equation}

Write
\begin{align*}
r_{\mathrm{ns}}:=\#\Sigma_{F,\infty}^{\mathrm{r,ns}}.
\end{align*}
Here $\mathrm r$ and $\mathrm c$ stand for real and complex, and $\mathrm{ns}$ and $\mathrm{sp}$ for nonsplit and split, respectively.
	
\subsubsection{The Haar Measure on $G(\mathbb{A}_F)$}

Throughout this paper, let $G=\mathrm{GL}(2)$ and $A=\diag(\mathrm{GL}(1),1)$. Let $B$ be the subgroup of upper triangular matrices in $G$. Let $\overline{G}=Z\backslash G$, where $Z$ is the center of $G$, and let $N$ be the unipotent radical of $B$. Let $K=\bigotimes_{v}K_v$ be a maximal compact subgroup of $G(\mathbb{A}_F)$, where $K_v=\mathrm{U}_2(\mathbb{C})$ if $v$ is complex, $K_v=\mathrm{O}_2(\mathbb{R})$ if $v$ is real, and $K_v=G(\mathcal{O}_v)$ if $v<\infty$.
    
For $v \in \Sigma_F$, the Iwasawa decomposition
\begin{align*}
G(F_v)=Z(F_v)N(F_v)A(F_v)K_v
\end{align*}
gives coordinates
\begin{align*}
g_v=z_v 
\begin{pmatrix}
1 & b_v \\
& 1
\end{pmatrix}
\begin{pmatrix}
a_v \\
& 1
\end{pmatrix}
k_v,
\end{align*}
where $z_v, a_v \in F_v^{\times}$, $b_v \in F_v$, and $k_v \in K_v$. We define a Haar measure on $G(F_v)$ by
\begin{align*}
dg_v = d^{\times}z_v \, db_v \, |a_v|_v^{-1}d^{\times}a_v \, dk_v,
\end{align*}
where $dk_v$ is the probability Haar measure on $K_v$.

\subsubsection{Automorphic Data}
Let $\omega\in\widehat C_F$. For a subgroup $H$ of $G$, we write
\begin{align*}
[H] := H(F) \backslash H(\mathbb{A}_F)
\end{align*}
for the corresponding automorphic quotient. Denote by $\mathcal{A}_0\left([G], \omega \right)$ the set of cuspidal representations of $G(\mathbb{A}_F)$ with central character $\omega$.

Given a suitable representation $\pi=\otimes_{v\in \Sigma_F}'\pi_v$, let $L(s,\pi)$ denote its finite $L$-function, and let $L_{\infty}(s,\pi)$ denote its Archimedean partial $L$-function. Define the complete $L$-function by
\begin{align*}
\Lambda(s,\pi)=L_{\infty}(s,\pi)L(s,\pi).
\end{align*}
In particular, when $\pi=\chi$ is a character, $\Lambda(s,\pi)=\Lambda(s,\chi)$ is the complete Hecke $L$-function. We add a field subscript when needed to specify the base field.

For $\eta_1,\eta_2\in\widehat C_F$, write $\eta_1\boxplus\eta_2$ for
the normalized parabolic induction from $B(\mathbb A_F)$ to
$G(\mathbb A_F)$ associated with $\eta_1\otimes\eta_2$.

We identify $N$ with the additive group via $\left(\begin{smallmatrix}1&b\\&1\end{smallmatrix}\right)\mapsto b$ and use $\psi$ and $\psi_v$ also for the corresponding characters of $[N]$ and $N(F_v)$.

Define the local Gamma functions: 
\begin{equation}\label{eq1.1}
\Gamma_{\mathbb{R}}(s):=\pi^{-s/2}\Gamma(s/2),\quad  \Gamma_{\mathbb{C}}(s):=2(2\pi)^{-s}\Gamma(s).
\end{equation}

Throughout, we follow the $\varepsilon$-convention: $\varepsilon$ always denotes a positive number which may be taken arbitrarily small, and may vary from one occurrence to another.

\subsection*{Acknowledgements}
I am grateful to Valentin Blomer for his interest in this work and for his
helpful comments.

\section{The Non-Split Motohashi Formula}\label{sec2}

\subsection{Eisenstein Series}

For $j=1,2$, let $\chi_j$ and $\omega_j$ be characters of
$F^{\times}\backslash\mathbb{A}_F^{\times}$, let $s_j\in\mathbb{C}$, and let
$\Phi_j$ be a Schwartz--Bruhat function on $\mathbb{A}_F^2$. For
$g\in G(\mathbb{A}_F)$, define
\begin{equation}\label{2.1}
\Eis(h_j(\cdot,\Phi_j,\chi_j,\omega_j,s_j))(g)
:=
\sum_{\delta\in B(F)\backslash G(F)}
h_j(\delta g,\Phi_j,\chi_j,\omega_j,s_j),
\end{equation}
where $h_j(g,\Phi_j,\chi_j,\omega_j,s_j)$ is defined by 
\begin{equation}\label{e2.2}
\chi_j(\det g)|\det g|^{\frac{1}{2}+s_j}
\int_{\mathbb{A}_F^{\times}}
\Phi_j((0,t)g)\chi_j^2\omega_j^{-1}(t)|t|^{1+2s_j}d^{\times}t. 
\end{equation}
Then $\Eis(h_j(\cdot,\Phi_j,\chi_j,\omega_j,s_j))(g)$ converges absolutely for
$\Re(s_j)\gg 1$ and admits meromorphic continuation to $s_j\in\mathbb{C}$.

For brevity, we write
\begin{align*}
h_j(g,s_j):=h_j(g,\Phi_j,\chi_j,\omega_j,s_j),
\end{align*}
and
\begin{align*}
E_j(g,s_j)
:=
\Eis(h_j(\cdot,s_j))(g)
=
\Eis(h_j(\cdot,\Phi_j,\chi_j,\omega_j,s_j))(g).
\end{align*}
Henceforth $E_j(g,s_j)$ is identified with its meromorphic continuation.

\subsection{Fourier Expansion}
Let
\begin{align*}
W_j(g,s_j):=\int_{[N]}E_j(ug,s_j)\overline{\psi}(u)du
\end{align*}
be the Whittaker function attached to $E_j$. Explicitly, it is given by
\begin{equation}\label{2.2}
\chi_j(\det g)|\det g|^{1/2+s_j}
\int_{\mathbb{A}_F^{\times}}\int_{\mathbb{A}_F}
\Phi_j((t,ct)g)\overline{\psi}(c)dc
\chi_j^{2}\omega_j^{-1}(t)|t|^{1+2s_j}d^{\times}t.
\end{equation}

The Fourier expansion of $E_j$ takes the form
\begin{equation}\label{eq2.6}
E_j(g,s_j)=h_j(g,s_j)+h_j^{\Diamond}(g,s_j)+\mathcal{G}_j(g,s_j),
\end{equation}
where
\begin{align*}
h_j^{\Diamond}(g,s_j)
:=
\int_{N(\mathbb{A}_F)}h_j(wug,s_j)du,\quad 
\mathcal{G}_j(g,s_j)
:=
\sum_{\alpha\in F^{\times}}
W_j\left(\begin{pmatrix}
\alpha\\
& 1
\end{pmatrix}g,s_j\right).
\end{align*}

\subsection{Period Integrals}
\subsubsection{Hecke Integrals}\label{sec2.3.1}
For $\Re(s)\gg 1$ and $\xi\in \widehat{C}_F$, define 
\begin{align*}
\Psi(W_2(\cdot,s_2),\xi|\cdot|^s):=\int_{\mathbb{A}_F^{\times}}
W_2\left(\begin{pmatrix}
a\\
& 1
\end{pmatrix},s_2\right)\xi(a)|a|^{s}d^{\times}a.
\end{align*}
This integral converges absolutely when $\Re(s)-|\Re(s_2)|$ is sufficiently large, and is a holomorphic multiple of
\begin{align*}
\Lambda(1/2+s+s_2,\chi_2\xi)\Lambda(1/2+s-s_2,\chi_2^{-1}\omega_2\xi).
\end{align*} 
It admits a meromorphic continuation $\widetilde{\Psi}(W_2(\cdot,s_2),\xi|\cdot|^s)$ to $(s,s_2)\in \mathbb{C}^2$. 

\subsubsection{Rankin--Selberg Convolutions}
Let $\pi$ be a generic automorphic representation of $[G]$ with central
character $\omega^{-1}=\omega_1^{-1}\omega_2^{-1}$, and let $W$ belong to
its Whittaker model for $\overline{\psi}$. For $\Re(s_2)$ sufficiently large, define
\begin{align*}
\Psi(W,W_1(\cdot,s_1),h_2(\cdot,s_2))
:=
\int_{N(\mathbb{A}_F)\backslash\overline{G}(\mathbb{A}_F)}
W(x)W_1(x,s_1)h_2(x,s_2)dx.
\end{align*}
By Rankin--Selberg theory, in the region where
$\Re(s_2)-|\Re(s_1)|$ is sufficiently large, this period is a holomorphic
multiple of
\begin{align*}
\Lambda(1/2+s_2+s_1,\pi\times\chi_1\chi_2)
\Lambda(1/2+s_2-s_1,\pi\times\chi_1^{-1}\omega_1\chi_2).
\end{align*}
It therefore admits a meromorphic continuation to $(s_1,s_2)\in\mathbb{C}^2$,
which we denote by $\widetilde{\Psi}(W,W_1(\cdot,s_1),h_2(\cdot,s_2))$.

\subsection{The Rankin--Selberg Kernel and Its Spectral Decomposition}

\subsubsection{The Rankin--Selberg Kernel}\label{sec2.kernel}
With $\mathbf{s}=(s_1,s_2)$ and $\omega=\omega_1\omega_2$, in the region
\begin{align*}
\Re(s_1)\geq 10,\qquad \Re(s_2)-|\Re(s_1)|\geq 10,
\end{align*}
define the Rankin--Selberg kernel by
\begin{align*}
\mathcal K_{12}(g;\mathbf{s},\boldsymbol{\Phi}):=\Eis\bigl(\mathcal G_1(\cdot,s_1)h_2(\cdot,s_2)\bigr)(g)=\sum_{\delta\in B(F)\backslash G(F)}\mathcal G_1(\delta g,s_1)h_2(\delta g,s_2).
\end{align*}

By \cite[Lemmas~2.1--2.2]{Yan26a}, the defining series converges absolutely and locally uniformly in this region and defines a smooth automorphic function with central character $\omega$ in $L^1\cap L^2$ modulo the center. The same Siegel-set estimates, applied after termwise differentiation and using the Schwartz decay of the differentiated Whittaker functions, give $D\mathcal K_{12}\in L^1\cap L^2$ for every fixed Archimedean right-invariant differential operator $D$.

Since $E_1(g,s_1)E_2(g,s_2)=\Eis(E_1(\cdot,s_1)h_2(\cdot,s_2))(g)$, the Fourier expansion \eqref{eq2.6} takes the form
\begin{align*}
\mathcal K_{12}(g;\mathbf{s},\boldsymbol{\Phi})=E_1(g,s_1)E_2(g,s_2)-\Eis\bigl((h_1(\cdot,s_1)+h_1^{\Diamond}(\cdot,s_1))h_2(\cdot,s_2)\bigr)(g).
\end{align*}
We call the two subtracted series, induced from $h_1h_2$ and
$h_1^{\Diamond}h_2$, the \emph{Eisenstein corrections}.
For every generic automorphic form $\phi$ of central character $\omega$,
its spectral coefficient unfolds to
\begin{align*}
\langle \mathcal K_{12}(\cdot;\mathbf{s},\boldsymbol{\Phi}),\phi\rangle=\Psi\bigl(\overline{W_{\phi}},W_1(\cdot,s_1),h_2(\cdot,s_2)\bigr),
\end{align*}
and $\mathcal K_{12}$ is orthogonal to the residual spectrum, which is nongeneric.

\subsubsection{The \texorpdfstring{$\mathrm{GL}_2$}{GL2}-spectrum}

Let $\mu_1,\mu_2\in \widehat{F^{\times}\backslash \mathbb{A}_F^{(1)}}$.
Let $H(\mu_1,\mu_2)$ be the space of functions $h\in L^2(K)$ satisfying
\begin{align*}
h\left(\begin{pmatrix}
a& b\\
&d
\end{pmatrix}k\right)
=
\mu_1(a)\mu_2(d)h(k),
\quad
k\in K,\quad
\begin{pmatrix}
a& b\\
&d
\end{pmatrix}\in B(\mathbb{A}_F)\cap K .
\end{align*}
For $h\in H(\mu_1,\mu_2)$ and $\lambda\in\mathbb{C}$, define
\begin{align*}
S(h)\left(\begin{pmatrix}
a& b\\
&d
\end{pmatrix}k,\lambda\right)
:=
\mu_1(a)\mu_2(d)\Big|\frac{a}{d}\Big|^{\lambda+\frac{1}{2}}h(k),
\end{align*}
where $a,d\in\mathbb{A}_F^{\times}$, $b\in\mathbb{A}_F$, and $k\in K$.
The associated Eisenstein series is
\begin{equation}\label{e2.5}
E(g,h,\lambda)
:=
\sum_{\gamma\in B(F)\backslash G(F)}
S(h)(\gamma g,\lambda),
\quad
\Re(\lambda)>1/2.
\end{equation}
It converges absolutely for $\Re(\lambda)>1/2$ and admits meromorphic
continuation to $\lambda\in\mathbb{C}$; we henceforth regard it as a
meromorphic function of $\lambda$. If $h$ is right-$K$-finite, then
$E(g,h,\lambda)$ is a finite linear combination of Eisenstein series of the
form \eqref{2.1}, by \cite[page 6]{JZ87}.

We fix once and for all an orthonormal basis
$\mathfrak{B}(\mu_1,\mu_2)$ of $K$-finite functions in $H(\mu_1,\mu_2)$.
For $\pi\in\mathcal{A}_0([G],\omega)$, let $\mathfrak{B}(\pi)$ be an
orthonormal basis of smooth $K$-finite vectors, chosen compatibly with the Archimedean Casimir operators.

\subsubsection{The Spectral Expansion of the Kernel}\label{sec2.4.2}

For $s_1,s_2,\lambda\in\mathbb{C}$ and $h\in H(\mu,\overline{\mu}\omega)$, put
\begin{align*}
\widetilde{\Psi}(\lambda,h;\mathbf{s}):=\widetilde{\Psi}\bigl(\overline{W_{E(\cdot,h,-\overline{\lambda})}},W_1(\cdot,s_1),h_2(\cdot,s_2)\bigr).
\end{align*}
By the kernel decomposition above and smooth spectral inversion, its genuine spectral expansion in the initial region is
\begin{align*}
\mathcal K_{12}(g;\mathbf{s},\boldsymbol{\Phi})=\mathcal E_{\cusp}(g;\mathbf{s},\boldsymbol{\Phi})+\mathcal E_{\Eis}(g;\mathbf{s},\boldsymbol{\Phi}),
\end{align*}
where
\begin{align*}
&\mathcal{E}_{\cusp}(g;\mathbf{s},\boldsymbol{\Phi}):=
\sum_{\pi\in \mathcal{A}_0([G],\omega)}
\sum_{\phi\in\mathfrak{B}(\pi)}
\widetilde{\Psi}\bigl(\overline{W_{\phi}},W_1(\cdot,s_1),h_2(\cdot,s_2)\bigr)\phi(g),\\
&\mathcal{E}_{\Eis}(g;\mathbf{s},\boldsymbol{\Phi}):=
\sum_{\mu\in \widehat{C}_F}\frac{1}{4\pi i\rho_F}\int_{i\mathbb{R}}
\sum_{h\in \mathfrak{B}(\mu,\overline{\mu}\omega)}
\widetilde{\Psi}(\lambda,h;\mathbf{s})E(g,h,\lambda)d\lambda.
\end{align*}
By \cite[Lemma~3.1]{Yan26a}, the displayed identity holds in $L^2$. The derivative estimates above imply convergence in every automorphic Sobolev norm, hence locally uniformly with derivatives by Sobolev embedding. Since the image of $\mathcal C_{E/F}$ is compact, the expansion may be restricted to it and integrated termwise.

\begin{defn}\label{defn2.1}
Write $\mu_1=\chi_1\chi_2$, $\mu_2=\omega\overline{\chi_1\chi_2}$, $\mu_3=\overline{\chi_1}\omega_1\chi_2$, and $\mu_4=\omega\chi_1\overline{\omega_1\chi_2}$. Since the continuous-spectrum density is $d\lambda/(4\pi i\rho_F)$, a contour crossing contributes $2\pi i/(4\pi i\rho_F)=1/(2\rho_F)$ times the signed residue. We define the residual terms
\begin{align*}
\widetilde{\Psi}_{\Res}^{(1)}(g;\mathbf{s},\boldsymbol{\Phi})
&:=(2\rho_F)^{-1}
\sum_{h\in \mathfrak{B}(\mu_1,\overline{\mu_1}\omega)}
\underset{\lambda=s_2+s_1-\frac{1}{2}}{\Res}\,
\widetilde{\Psi}(\lambda,h;\mathbf{s})
E(g,h,\lambda),\\
\widetilde{\Psi}_{\Res}^{(2)}(g;\mathbf{s},\boldsymbol{\Phi})
&:=-(2\rho_F)^{-1}
\sum_{h\in \mathfrak{B}(\mu_2,\overline{\mu_2}\omega)}
\underset{\lambda=\frac{1}{2}-s_2-s_1}{\Res}\,
\widetilde{\Psi}(\lambda,h;\mathbf{s})
E(g,h,\lambda),\\
\widetilde{\Psi}_{\Res}^{(3)}(g;\mathbf{s},\boldsymbol{\Phi})
&:=(2\rho_F)^{-1}
\sum_{h\in \mathfrak{B}(\mu_3,\overline{\mu_3}\omega)}
\underset{\lambda=s_2-s_1-\frac{1}{2}}{\Res}\,
\widetilde{\Psi}(\lambda,h;\mathbf{s})
E(g,h,\lambda),\\
\widetilde{\Psi}_{\Res}^{(4)}(g;\mathbf{s},\boldsymbol{\Phi})
&:=-(2\rho_F)^{-1}
\sum_{h\in \mathfrak{B}(\mu_4,\overline{\mu_4}\omega)}
\underset{\lambda=\frac{1}{2}+s_1-s_2}{\Res}\,
\widetilde{\Psi}(\lambda,h;\mathbf{s})
E(g,h,\lambda).
\end{align*}
\end{defn}

\begin{prop}\label{prop2.2}
Let $|\Re(s_1)|+|\Re(s_2)|<1/2$ and $\omega=\omega_1\omega_2$. Then, for
$g\in G(\mathbb{A}_F)$, one has the following identity of meromorphic functions of $\mathbf{s}$:
\begin{multline}\label{2.3}
E_1(g,s_1)E_2(g,s_2)
=
\Eis(h_1(\cdot,s_1)h_2(\cdot,s_2))(g)
+
\Eis(h_1^{\Diamond}(\cdot,s_1)h_2(\cdot,s_2))(g)\\
+\mathcal{E}_{\cusp}(g;\mathbf{s},\boldsymbol{\Phi})
+\mathcal{E}_{\Eis}(g;\mathbf{s},\boldsymbol{\Phi})
-
\sum_{i=1}^{4}\widetilde{\Psi}_{\Res}^{(i)}(g;\mathbf{s},\boldsymbol{\Phi}), 
\end{multline}
\end{prop}
\begin{proof}
Combining the geometric and spectral expansions of the Rankin--Selberg kernel above proves \eqref{2.3} in the initial region, without residual terms. We then use the one-kernel specialization of the continuation procedure in \cite[\textsection 3.4]{Yan26a}. On every compact subset in $g$, derivatives of $E(g,h,\lambda)$ have at most polynomial growth in the spectral parameter and the $K$-type, whereas repeated integration by parts in the Rankin--Selberg local integrals gives arbitrary decay of $\widetilde{\Psi}$. The Weyl law and the continuous Plancherel formula therefore give locally normal $C^r$-convergence on the locally deformed contours, for every fixed $r$. Thus the cuspidal sum and the deformed continuous-spectrum integral continue meromorphically as $C^\infty_{\mathrm{loc}}$-valued functions. As $\mathbf{s}$ leaves the region of absolute convergence, the $\lambda$-contour crosses precisely the four polar divisors displayed in Definition \ref{defn2.1}; with the conventions there, the total contour correction is $-\sum_{i=1}^{4}\widetilde{\Psi}_{\Res}^{(i)}$. This gives \eqref{2.3} in $|\Re(s_1)|+|\Re(s_2)|<1/2$ as an identity of meromorphic functions, and hence pointwise at every regular value of $\mathbf{s}$.
\end{proof}

\subsection{Hecke's Formula via Eisenstein Periods}\label{sec2.5}
Let $E/F$ be a quadratic extension of number fields. Write $E=F(\sqrt{\tau})$,
where $\tau$ is a fixed nonsquare in $F^{\times}$. Fix $\theta\in E^{\times}$
such that $\theta^2=\tau$, and consider the embedding
\begin{align*}
\mathbb{A}_E^{\times}\hookrightarrow M_{2\times 2}(\mathbb{A}_F),\quad  x=a+b\theta\mapsto \iota(x):=\begin{pmatrix}
a & b\tau \\
b & a
\end{pmatrix},\ \ a, b\in \mathbb{A}_F.
\end{align*}

Let
$\mathcal{C}_{E/F}=E^{\times}\mathbb{A}_F^{\times}\backslash\mathbb{A}_E^{\times}$,
which is compact, and let $\widehat{\mathcal{C}}_{E/F}$ denote its
Pontryagin dual. Since 
\begin{align*}
\mathcal{C}_{E/F}\simeq (E^{\times}\backslash\mathbb{A}_E^{(1)})/(F^{\times}\backslash\mathbb{A}_F^{(1)}).
\end{align*}
we have
\begin{equation}\label{e2.8}
\Vol(\mathcal{C}_{E/F})=2\frac{\Vol(E^{\times}\backslash\mathbb{A}_E^{(1)})}{\Vol(F^{\times}\backslash\mathbb{A}_F^{(1)})}=2\frac{\underset{s=1}{\Res}\ \zeta_E(s)}{\underset{s=1}{\Res}\ \zeta_F(s)}=2L(1,\eta_{E/F}),
\end{equation}
where $\eta_{E/F}$ is the quadratic character associated with $E/F$. The factor $2$ is the radial Jacobian: for $a\in\mathbb A_F^\times$ one has $|a|_E=|a|_F^2$.

Let $\chi$ be a unitary Hecke character of
$E^{\times}\backslash\mathbb{A}_E^{\times}$ such that
$\omega_j^{-1}=\chi|_{\mathbb{A}_F^{\times}}$. Define
\begin{equation}\label{2.8}
Z(s_j,\Phi_j,\chi):=\int_{\mathcal{C}_{E/F}}E_j(\iota(x),s_j)
\chi(x)d^{\times}x.
\end{equation}
Since $\mathcal{C}_{E/F}$ is compact, $Z(s_j,\Phi_j,\chi)$ is meromorphic,
with its analytic behavior inherited from that of $E_j(\iota(x),s_j)$.

Suppose $\Re(s_j)\gg 1$. Since $G(F)=B(F)\iota(E^{\times})$ and
$B(F)\cap \iota(E^{\times})\simeq F^{\times}$, it follows from \eqref{2.1} that
\begin{multline}\label{eq2.4}
Z(s_j,\Phi_j,\chi)=\int_{E^{\times}\mathbb{A}_F^{\times}\backslash\mathbb{A}_E^{\times}}\sum_{\delta\in F^{\times}\backslash E^{\times}}\chi_j(\det \iota(x))|\det \iota(x)|_{\mathbb{A}_F}^{\frac{1}{2}+s_j}\\
\int_{\mathbb{A}_F^{\times}}
\Phi_j((0,t)\iota(\delta x))\chi_j^2\omega_j^{-1}(t)|t|_{\mathbb{A}_F}^{1+2s_j}d^{\times}t\chi(x)d^{\times}x.
\end{multline}

By a change of variables, and using the assumption
$\omega_j^{-1}=\chi|_{\mathbb{A}_F^{\times}}$, we obtain from \eqref{eq2.4} that
\begin{align*}
Z(s_j,\Phi_j,\chi)=\int_{\mathbb{A}_E^{\times}}
\Phi_j((0,1)\iota(x))\chi_j(\det \iota(x))\chi(x)|\det \iota(x)|_{\mathbb{A}_F}^{\frac{1}{2}+s_j}d^{\times}x.	
\end{align*}
Together with the identity $N_{E/F}(x)=\det \iota(x)$, this gives
\begin{align*}
Z(s_j,\Phi_j,\chi)=\int_{\mathbb{A}_E^{\times}}
\Phi_j((0,1)\iota(x))\widetilde{\chi}_j\chi(x)|x|_{\mathbb{A}_E}^{\frac{1}{2}+s_j}d^{\times}x,	
\end{align*}
where $\widetilde{\chi}_j:=\chi_j\circ N_{E/F}$. Consequently, $Z(s_j,\Phi_j,\chi)$ is a holomorphic multiple of
\begin{align*}
\Lambda(1/2+s_j,\widetilde{\chi}_j\chi).
\end{align*}

For any unitary Hecke character $\xi$ of
$E^{\times}\backslash\mathbb{A}_E^{\times}$ satisfying
$\xi|_{\mathbb{A}_F^{\times}}=(\omega_1\omega_2)^{-1}$, define
\begin{align*}
&Z(\mathbf{s},\boldsymbol{\Phi},\xi):=\int_{\mathcal{C}_{E/F}}\Eis(h_1(\cdot,s_1)h_2(\cdot,s_2))(\iota(x))\xi(x)d^{\times}x,\\
&Z^{\Diamond}(\mathbf{s},\boldsymbol{\Phi},\xi):=\int_{\mathcal{C}_{E/F}}\Eis(h_1^{\Diamond}(\cdot,s_1)h_2(\cdot,s_2))(\iota(x))\xi(x)d^{\times}x.
\end{align*}
For $1\leq i\leq 4$, we define 
\begin{align*}
\widetilde{\Psi}_{\Res}^{(i)}(\mathbf{s},\boldsymbol{\Phi},\xi):=\int_{\mathcal{C}_{E/F}}\widetilde{\Psi}_{\Res}^{(i)}(\iota(x);\mathbf{s},\boldsymbol{\Phi})\xi(x)d^{\times}x.
\end{align*}

\subsection{The Motohashi Identity}\label{sec2.6}
We deduce from Proposition \ref{prop2.2} the following non-split Motohashi formula for toric periods.
\begin{thm}\label{thm2.3}
Let $\xi_1$ and $\xi_2$ be unitary Hecke characters of $E^{\times}\backslash\mathbb{A}_E^{\times}$, put $\omega_i=(\xi_i|_{\mathbb{A}_F^{\times}})^{-1}$ for $1\leq i\leq2$, and take $\chi_1=\chi_2=\mathbf{1}$ in the Eisenstein data above. Let $\boldsymbol{\Phi}:=(\Phi_1, \Phi_2)$ be a pair of Bruhat--Schwartz functions on $\mathbb{A}_F^{\oplus 2}$, and let $\mathbf{s}=(s_1,s_2)\in \mathbb{C}^2$ with $|\Re(s_1)|+|\Re(s_2)|<1/2$. Then 
\begin{multline}\label{2.4}
\sum_{\chi\in \widehat{\mathcal{C}}_{E/F}}\frac{Z(s_1,\Phi_1,\xi_1\chi)Z(s_2,\Phi_2,\xi_2\overline{\chi})}{2L(1,\eta_{E/F})}=Z(\mathbf{s},\boldsymbol{\Phi},\xi_1\xi_2)+ Z^{\Diamond}(\mathbf{s},\boldsymbol{\Phi},\xi_1\xi_2)\\
+\mathcal{P}_{\cusp}(\mathbf{s},\boldsymbol{\Phi},\xi_1\xi_2)
+\mathcal{P}_{\Eis}(\mathbf{s},\boldsymbol{\Phi},\xi_1\xi_2)
-
\sum_{i=1}^{4}\widetilde{\Psi}_{\Res}^{(i)}(\mathbf{s},\boldsymbol{\Phi},\xi_1\xi_2),
\end{multline}	
where, for any unitary Hecke character $\xi$ of $E^{\times}\backslash\mathbb{A}_E^{\times}$ satisfying $\xi|_{\mathbb{A}_F^{\times}}=(\omega_1\omega_2)^{-1}$, 
\begin{align*}
&\mathcal{P}_{\cusp}(\mathbf{s},\boldsymbol{\Phi},\xi)
:=
\sum_{\pi\in \mathcal{A}_0([G],\omega_1\omega_2)}
\sum_{\phi\in\mathfrak{B}(\pi)}
\widetilde{\Psi}
\bigl(\overline{W_{\phi}},W_1(\cdot,s_1),h_2(\cdot,s_2)\bigr)\int_{\mathcal{C}_{E/F}}\phi(\iota(x))\xi(x)d^{\times}x,\\
&\mathcal{P}_{\Eis}(\mathbf{s},\boldsymbol{\Phi},\xi)
:=
\sum_{\mu\in \widehat{C}_F}
\frac{1}{4\pi i\rho_F}
\int_{i\mathbb{R}}
\sum_{h\in \mathfrak{B}(\mu,\overline{\mu}\omega_1\omega_2)}
\widetilde{\Psi}(\lambda,h;\mathbf{s})\int_{\mathcal{C}_{E/F}}E(\iota(x),h,\lambda)\xi(x)d^{\times}x\,d\lambda.
\end{align*}
\end{thm}

\begin{proof}
By the toric Fourier expansion and \eqref{2.8}, for $x\in \mathbb{A}_E^{\times}$, we have
\begin{align*}
\xi_1(x)E_1(\iota(x),s_1)=\frac{1}{\Vol(\mathcal{C}_{E/F})}\sum_{\chi\in \widehat{\mathcal{C}}_{E/F}}Z(s_1,\Phi_1,\xi_1\chi)\overline{\chi}(x).
\end{align*}

It follows that
\begin{equation}\label{2.12}
\int_{\mathcal{C}_{E/F}}E_1(\iota(x),s_1)
E_2(\iota(x),s_2)\xi_1\xi_2(x)d^{\times}x=\sum_{\chi\in \widehat{\mathcal{C}}_{E/F}}\frac{Z(s_1,\Phi_1,\xi_1\chi)Z(s_2,\Phi_2,\xi_2\overline{\chi})}{\Vol(\mathcal{C}_{E/F})}.
\end{equation}

Therefore, \eqref{2.4} follows from Proposition \ref{prop2.2}, \eqref{e2.8}, and \eqref{2.12}.
\end{proof}

When $\xi_1=\xi_2=1$, we omit the final character argument in
$Z$, $Z^\Diamond$, $\mathcal P_{\cusp}$, $\mathcal P_{\Eis}$, and
$\widetilde\Psi_{\Res}^{(i)}$.
We refer to the left-hand side of \eqref{2.4} as the \textit{spectral moment}, to $Z(\mathbf{s},\boldsymbol{\Phi},\xi_1\xi_2)+ Z^{\Diamond}(\mathbf{s},\boldsymbol{\Phi},\xi_1\xi_2)$ as the \textit{geometric terms}, to $\mathcal{P}_{\cusp}(\mathbf{s},\boldsymbol{\Phi},\xi_1\xi_2)
+\mathcal{P}_{\Eis}(\mathbf{s},\boldsymbol{\Phi},\xi_1\xi_2)$ as the \textit{dual moment}, and to $\sum_{i=1}^{4}\widetilde{\Psi}_{\Res}^{(i)}(\mathbf{s},\boldsymbol{\Phi},\xi_1\xi_2)$ as the \textit{residual terms}.   

\section{Explicit Automorphic Data and Period Formula}\label{sect3}
Let $E/F$ be a quadratic extension of number fields, and let
$\mathrm{St}(E/F)$ denote the Steinitz class of $\mathcal{O}_E$ over
$\mathcal{O}_F$. Let $\mathfrak{n}\subseteq \mathcal{O}_F$ be an ideal all of
whose prime divisors split in $E$ and are coprime to $2\mathfrak{D}_{E/F}\mathfrak{D}_F$.

Choose an integral ideal
$\mathfrak{a}\in \mathrm{St}(E/F)^{-1}$ satisfying
\begin{align*}
\mathfrak{a}+2\mathfrak{n}\mathfrak{D}_{E/F}\mathfrak{D}_F=\mathcal{O}_F,
\end{align*}
and choose $\tau\in \mathcal{O}_F$ such that
\begin{align*}
E=F(\sqrt{\tau}),\quad
(\tau)=\mathfrak{D}_{E/F}\mathfrak{a}^2.
\end{align*}
We fix such choices of $\mathfrak{a}$ and $\tau$ throughout the paper.

\subsection{Construction of Automorphic Data}\label{sect3.1}
\subsubsection{Archimedean Components}\label{sec3.1}
We decompose the set of Archimedean places as
\begin{align*}
\Sigma_{F,\infty}=\Sigma_{F,\infty}^{\mathrm{r,sp}}\bigsqcup \Sigma_{F,\infty}^{\mathrm{r,ns}}\bigsqcup \Sigma_{F,\infty}^{\mathrm{c}},
\end{align*}
where $\Sigma_{F,\infty}^{\mathrm{r,sp}}$, $\Sigma_{F,\infty}^{\mathrm{r,ns}}$ and $\Sigma_{F,\infty}^{\mathrm{c}}$ are defined as in \eqref{1.2}. 

\begin{itemize}
\item Let $v\in \Sigma_{F,\infty}^{\mathrm{r,ns}}$. Under $F_v\simeq \mathbb{R}$ we have $\tau_v<0$. Define 
\begin{align*}
\Phi_v(t_1,t_2):=\Phi_v^{\circ}((t_1,t_2)g_v^{\dag}),\ \ t_1, t_2\in F_v,
\end{align*}
where 
\begin{align*}
\Phi_v^{\circ}(t_1,t_2):=\pi^{-1}e^{-2\pi(t_1^2+t_2^2)},\quad \ \ g_v^{\dag}:=\begin{pmatrix}
|\tau_v|_v^{1/2}\\
& 1
\end{pmatrix}.
\end{align*}

\item Let $v\in \Sigma_{F,\infty}^{\mathrm{r,sp}}$. Under $F_v\simeq \mathbb{R}$ we have $\tau_v>0$. Define 
\begin{align*}
\Phi_v(t_1,t_2):=\Phi_v^{\circ}((t_1,t_2)g_v^{\dag}), \ \ t_1, t_2\in F_v,
\end{align*}
where 
\begin{align*}
\Phi_v^{\circ}(t_1,t_2):=e^{-\pi(t_1^2+t_2^2)},\quad g_v^{\dag}:=\begin{pmatrix}
|\tau_v|_v^{1/2}\\
& 1
\end{pmatrix}\mathbf{B},\quad \mathbf{B}:=\begin{pmatrix}
1 & -1\\
1 & 1
\end{pmatrix}.
\end{align*}

\item Let $v\in \Sigma_{F,\infty}^{\mathrm{c}}$. Choose $\theta_v\in F_v$ with $\theta_v^2=\tau_v$ and define
\begin{align*}
\Phi_v((z_1, z_2)):=\Phi_v^{\circ}((z_1, z_2)g_v^{\dag}),\ \ z_1, z_2\in F_v,
\end{align*}
where 
\begin{align*}
\Phi_v^{\circ}((z_1, z_2)):=\pi^{-2}e^{-2\pi(z_1\overline{z_1}+z_2\overline{z_2})},\quad \ \ g_v^{\dag}:=\begin{pmatrix}
\theta_v\\
& 1
\end{pmatrix}\mathbf{B}.
\end{align*}
\end{itemize}

\subsubsection{Non-Archimedean Components}\label{sec3.2}
For each $v\in \Sigma_{F,\fin}$, choose $u_{0,v}\in F_v$ and $u_v\in F_v^\times$ such that $1,u_{0,v}+u_v\theta$ is an $\mathcal O_v$-basis of $\mathcal O_{E_v}$. For $t_1,t_2\in F_v$, set
\begin{equation}\label{eq3.1}
\Phi_v^{\circ}(t_1,t_2)=\mathbf{1}_{\mathcal{O}_v}(t_1)\mathbf{1}_{\mathcal{O}_v}(t_2),\qquad 
\Phi_v^*(t_1,t_2)=\mathbf{1}_{\mathcal O_{E_v}}(t_2+t_1\theta).
\end{equation}
Then 
\begin{align*}
\Phi_v^*(t_1,t_2)=\Phi_v^\circ\bigl((t_1,t_2)g_{0,v}\bigr),\quad g_{0,v}=\begin{pmatrix}u_v^{-1}&-u_{0,v}/u_v\\0&1\end{pmatrix}.
\end{align*}
This definition applies also above $2$; it is the actual integer-ring lattice rather than a rectangular coordinate lattice.
We now define $\Phi_v$ as follows.

\begin{itemize}
\item Let $v\mid\mathfrak n$. Let  $e_v(\mathfrak{n})=r_v>0$. Define 
\begin{equation}\label{3.2}
\Phi_v(t_1,t_2)=\sum_{j=1}^{2}\Phi_v^{(j)}(t_1,t_2),\quad t_1,t_2\in F_v,
\end{equation}
where 
\begin{align*}
\Phi_v^{(j)}(t_1,t_2):=
\begin{cases}
q_v^{r_v/2}\Phi_v^{\circ}\left((t_1,t_2)g_v^{\dag}\begin{pmatrix}
\varpi_v^{-r_v}\\
& 1
\end{pmatrix}\right), &\quad \text{$j=1$,}\\
q_v^{r_v/2}\Phi_v^{\circ}\left((t_1,t_2)g_v^{\dag}\begin{pmatrix}
1\\
& \varpi_v^{-r_v}
\end{pmatrix}\right), &\quad \text{$j=2$}.
\end{cases}
\end{align*}
Here the matrix $g_v^{\dag}$ is defined by 
\begin{equation}\label{eq3.3}
g_v^{\dag}:=\begin{pmatrix}
\theta_v\\
& 1
\end{pmatrix}\mathbf{B}.
\end{equation}

\item Let $v<\infty$ with $v\nmid\mathfrak n$. Define
\begin{align*}
\Phi_v(t_1,t_2)=q_v^{d_{E_v}/2}\Phi_v^*(t_1,t_2),\quad t_1,t_2\in F_v.
\end{align*}
\end{itemize}

\subsubsection{Construction of Automorphic Data}\label{sec3.1.3}
Let $\Phi=\otimes_{v\in \Sigma_{F}}\Phi_v$, where $\Phi_v$ is defined as in \textsection\ref{sec3.1}--\textsection\ref{sec3.2}. We set
\begin{align*}
&\Phi_1:=\bigotimes_{v\in \Sigma_{F,\infty}}\Phi_v\bigotimes_{v\in \Sigma_{F,\fin}}q_v^{d_{E_v}/2}\Phi_v^*,\\
&\Phi_2=\Phi,\ \ \ \chi_1=\chi_2=\omega_1=\omega_2=\xi_1=\xi_2=\mathbf{1}.
\end{align*}

For an integral ideal $\mathfrak{m}\subset\mathcal{O}_F$, let $\mathcal{F}(\mathfrak{m},\mathbf{1})$ denote the set of unitary cuspidal automorphic representations of $\mathrm{PGL}_2/F$ that are spherical at all Archimedean places and have arithmetic conductor $\mathfrak{m}$.

Let $\widehat{C}_F(\mathfrak{n})\subset \widehat{C}_F$ denote the subset of characters $\mu=\otimes_{v\in \Sigma_F}'\mu_v$ such that $\mu_v$ is unramified at all $v\nmid\mathfrak n$, and such that the exponent of its arithmetic conductor is $\leq e_v(\mathfrak{n})/2$ at all $v\mid\mathfrak n$. 

Let $\widehat{\mathcal{C}}_{E/F}^{\text{un}}$ denote the subgroup of everywhere unramified characters in $\widehat{\mathcal{C}}_{E/F}$. For $\chi\in \widehat{\mathcal{C}}_{E/F}$, define 
\begin{equation}\label{f3.4}
\lambda_{\chi}^*(\mathfrak{n}):=2^{-\omega_F(\mathfrak{n})}\lambda_{\chi}(\mathfrak{n}),\quad 
\lambda_{\chi}(\mathfrak{n}):=\prod_{\mathfrak{p}\mid\mathfrak{n}}\big(\chi(\mathfrak{p})^{e_{\mathfrak{p}}(\mathfrak{n})}+\overline{\chi}(\mathfrak{p})^{e_{\mathfrak{p}}(\mathfrak{n})}\big),
\end{equation}
where $\omega_F(\mathfrak{n})$ is the number of distinct prime ideal divisors of $\mathfrak{n}$, $e_{\mathfrak{p}}(\mathfrak{n})=e_{v}(\mathfrak{n})$, and $v$ corresponds to $\mathfrak{p}$. At each split prime, $\chi(\mathfrak p)$ denotes the value at either chosen prime of $E$ above $\mathfrak p$; the symmetric expression is independent of this choice. Here $\chi$ is unramified at the primes dividing $\mathfrak n$.

In particular, when $\mathfrak{n}$ is squarefree, $\lambda_{\chi}(\mathfrak{n})$ is precisely the normalized $\mathfrak{n}$-th Hecke eigenvalue of the automorphic induction associated with $\chi$. 

\begin{remark}
We restrict the twisting ideal to split primes because inert primes give no character-dependent Hecke weights. If $\mathfrak p$ is inert, then $\mathfrak P=\mathfrak p\mathcal O_E$ and $\chi(\mathfrak P)=1$ for $\chi\in\widehat{\mathcal C}_{E/F}^{\mathrm{un}}$, so the local Euler factor $(1-N_F(\mathfrak p)^{-2s})^{-1}$ is independent of $\chi$. Inert places are nevertheless included in all global formulas.
\end{remark}

For $v\in \Sigma_{F,\fin}$ and $l\in \mathbb{Z}_{\geq 0}$, we define 
\begin{equation}\label{3.5}
K_v[l]:=\bigg\{\begin{pmatrix}
a & b\\
c & d
\end{pmatrix}\in K_v:\ c\in \mathfrak{p}_v^{l}\bigg\}.
\end{equation}

\subsection{Arithmetic Weights}\label{sec3.6}
Let $(s_1,s_2)\in \mathbb{C}^2$ and  $v\mid\mathfrak n$. Write $e_v(\mathfrak{n})=r_v$. 
\subsubsection{Weights in Geometric Terms}\label{3.2.1}
We define
\begin{multline*}
H_v(s_1,s_2)
:=\frac{
1}{
1-q_v^{-2-2s_1-2s_2}
}
\bigg[
1-q_v^{-1-2s_2}+
q_v^{r_v(s_1-s_2)}q_v^{-1-2s_2}
(1-q_v^{-1-2s_1})\\
+
(1-q_v^{-1-2s_1})(1-q_v^{-1-2s_2})
\frac{q_v^{s_1-s_2}(1-q_v^{r_v(s_1-s_2)})}
{1-q_v^{s_1-s_2}}
\bigg].
\end{multline*}

We then define the holomorphic factor
\begin{equation}\label{e3.6}
H_{\mathfrak{n}}(s_1,s_2):=\prod_{v\mid\mathfrak n} H_v(s_1,s_2). 
\end{equation}

\subsubsection{Weights in Residual Terms}\label{sec3.2.2}
For $v\mid\mathfrak n$, write $q=q_v$, $r=e_v(\mathfrak n)$,
$x=q^{-1/2+\lambda}$, $y=q^{-1/2-\lambda}$, and
$S_s(j)=\sum_{k=0}^jq^{(j-2k)s}$. Define
\begin{equation}\label{eq3.Tl}
T_l=q^{l-r/2}\bigg(\sum_{j\geq l}x^jS_{s_1}(j)S_{s_2}(j+r-2l)-\mathbf1_{l\geq1}\frac{x^{l-1}S_{s_1}(l-1)S_{s_2}(r-l-1)}{q-1}\bigg)
\end{equation}
for $0\leq l<r$, and 
\begin{equation}\label{eq3.Tr}
T_r=q^{r/2}\sum_{j\geq r}x^jS_{s_1}(j)S_{s_2}(j-r).
\end{equation}
The series defining $T_l$, $0\leq l\leq r$, are initially evaluated in their common region of absolute convergence and then continued meromorphically using their rational expressions in $q^\lambda,q^{s_1},q^{s_2}$.
Define
\begin{equation}\label{equ3.8}
\mathbf L_v^-(\mathbf s,\lambda)=\prod_{\epsilon,\delta=\pm1}\zeta_{F_v}(1/2-\lambda+\epsilon s_1+\delta s_2)
\end{equation}
and
\begin{equation}\label{eq3.Rr}
R_{\mathfrak n}(s_1,s_2;\lambda)
=\prod_{v\mid\mathfrak n}\mathscr R_{v,r_v}(s_1,s_2;\lambda),
\end{equation}
where 
\begin{align*}
\mathscr R_{v,r}(s_1,s_2;\lambda)
=\frac{q^{r/2}\zeta_{F_v}(1-2\lambda)}{(1+y)\mathbf L_v^-(\mathbf s,\lambda)}
\bigg(T_0+y^rT_r+(1-y)\sum_{l=1}^{r-1}y^lT_l\bigg).
\end{align*}
The subtraction in \eqref{eq3.Tl} is the mean of the
additive character on the outer valuation layer. The value of the Gaussian
section at the identity contributes $\zeta_{F_v}(1+2\lambda)$; this factor
cancels in \eqref{eq3.Rr}. Consequently
\begin{equation}\label{eq3.Ridentities}
\begin{aligned}
\mathscr R_{v,r}(s_1,s_2;\lambda)
&=\mathscr R_{v,r}(s_1,s_2;-\lambda),\\
\mathscr R_{v,r}(s_1,s_2;1/2-s_1-s_2)
&=q^{rs_1}H_v(-s_1,-s_2),\\
\mathscr R_{v,r}(s_1,s_2;1/2+s_1-s_2)
&=q^{-rs_1}H_v(s_1,-s_2).
\end{aligned}
\end{equation}
These identities follow by summing the four geometric series in
\eqref{eq3.Tl}--\eqref{eq3.Tr}; both sides satisfy the same fourth-order
recurrence in $r$, and direct substitution for $1\leq r\leq4$ initializes
the recurrence. In particular
\[
\mathscr R_{v,r}(0,0;\pm1/2)=H_v(0,0)=1+r\frac{q-1}{q+1}.
\]
For squarefree $\mathfrak n$ this becomes
\begin{equation}\label{eq3.9}
R_{\mathfrak n}(s_1,s_2;\lambda)=\prod_{v\mid\mathfrak n}\frac{q_v^{s_1}+q_v^{-s_1}+q_v^{s_2}+q_v^{-s_2}}{(1+q_v^{\lambda-1/2})(1+q_v^{-\lambda-1/2})}.
\end{equation}

\subsection{Explicit Formula for \texorpdfstring{$Z(s,\Phi_j,\chi)$}{Z(s,Phi,chi)}}
Let $\chi\in \widehat{\mathcal{C}}_{E/F}$ and let $\Phi_j$, $j=1,2$, be as above. For $\Re(s)\gg 1$, recall that 
\begin{align*}
Z(s,\Phi_j,\chi)=\prod_{v\in \Sigma_F}Z_v(s,\Phi_{j,v},\chi_v),
\end{align*}
where 
\begin{equation}\label{3.1}
Z_v(s,\Phi_{j,v},\chi_v):=\int_{E_v^{\times}}
\Phi_{j,v}((0,1)\iota(x_v))\chi_v(x_v)|x_v|_{E_v}^{1/2+s}d^{\times}x_v.
\end{equation}
Here $\Phi_{j,v}$ is the $v$-th component of $\Phi_j$. 

\begin{prop}\label{lem3.1}
Let $\chi\in \widehat{\mathcal{C}}_{E/F}$. Then, for $s\in \mathbb{C}$, we have an identity of meromorphic functions:
\begin{align*}
&Z(s,\Phi_1,\chi)=\mathbf{1}_{\chi\in \widehat{\mathcal{C}}_{E/F}^{\text{un}}}\Lambda(1/2+s,\chi),\\
&Z(s,\Phi_2,\chi)=N_F(\mathfrak{n})^{-s}\lambda_{\chi}(\mathfrak{n})\mathbf{1}_{\chi\in \widehat{\mathcal{C}}_{E/F}^{\text{un}}}\Lambda(1/2+s,\chi).
\end{align*}
\end{prop}
\begin{proof}
We compute the local factors according to the construction of $\Phi_v$ in \textsection\ref{sec3.1}--\textsection\ref{sec3.2}.
\begin{itemize}
\item Suppose $v\in \Sigma_{F,\infty}^{\mathrm{r,sp}}$. Write $x_v=a_v+\theta_vb_v$, where $\theta_v=\sqrt{\tau_v}$. Then 
\begin{align*}
\Phi_v((0,1)\iota(x_v))=\Phi_v^{\circ}((\theta_vb_v,a_v)\mathbf{B})=e^{-\pi(a_v+\theta_vb_v)^2}e^{-\pi(a_v-\theta_vb_v)^2}.
\end{align*}

In this case, $\chi_v(x_v)=\chi_v'(a_v+\theta_vb_v)\chi_v'^{-1}(a_v-\theta_vb_v)$ for some $\chi_v'$. With this substitution, \eqref{3.1} becomes
\begin{multline}\label{3.3}
Z_v(s,\Phi_v,\chi_v)=\int_{\mathbb{R}^{\times}}\int_{\mathbb{R}^{\times}}
e^{-\pi(a_v+\theta_vb_v)^2}e^{-\pi(a_v-\theta_vb_v)^2}\chi_v'(a_v+\theta_vb_v)\\\chi_v'^{-1}(a_v-\theta_vb_v)
|a_v+\theta_vb_v|_{v}^{\frac{1}{2}+s}|a_v-\theta_vb_v|_{v}^{\frac{1}{2}+s}d^{\times}(a_v+\theta_vb_v)d^{\times}(a_v-\theta_vb_v).
\end{multline}

A direct calculation from \eqref{3.3} yields
\begin{align*}
Z_v(s,\Phi_v,\chi_v)=\mathbf{1}_{\text{$\chi_v$ unram}}\prod_{w\mid v}L_w(s+1/2,\chi_w),
\end{align*}
where the product runs over the two places $w\in \Sigma_E$ above $v\in \Sigma_F$.

\item Suppose $v\in \Sigma_{F,\infty}^{\mathrm{r,ns}}$. We have 
\begin{align*}
\Phi_v((0,1)\iota(x_v))=\Phi_v^{\circ}((|\tau_v|_v^{1/2}b_v,a_v))=\pi^{-1}e^{-2\pi(a_v^2-\tau_vb_v^2)}.
\end{align*}

With this substitution, \eqref{3.1} becomes
\begin{equation}\label{e3.1}
Z_v(s,\Phi_v,\chi_v)=\pi^{-1}\int_{\mathbb{C}^{\times}}
e^{-2\pi|z|_{\mathbb{C}}}\chi_v(z)|z|_{\mathbb{C}}^{\frac{1}{2}+s}d^{\times}z.
\end{equation} 

Recall that $\chi_v$ has the shape 
\begin{align*}
\chi_v(z)=\left(\frac{z}{|z|}\right)^{m_v}|z|_{\mathbb C}^{it_v/2}
\end{align*}
for some $m_v\in\mathbb{Z}$ and $t_v\in \mathbb{R}$. It follows from \eqref{e3.1} that 
\begin{align*}
Z_v(s,\Phi_v,\chi_v)=\mathbf{1}_{m_v=0}\Gamma_{\mathbb C}\left(s+1/2+it_v/2\right)=\mathbf{1}_{\text{$\chi_v$ unram}}L_v(s+1/2,\chi_v),
\end{align*}
where $\Gamma_{\mathbb{C}}$ is defined by \eqref{eq1.1}.

\item Suppose $v\in \Sigma_{F,\infty}^{\mathrm{c}}$. Write $x_v=a_v+\theta_vb_v$. Then 
\begin{align*}
\Phi_v((0,1)\iota(x_v))=\Phi_v^{\circ}((\theta_vb_v,a_v)\mathbf{B})=\pi^{-2}e^{-2\pi|a_v+\theta_vb_v|^2}e^{-2\pi|a_v-\theta_vb_v|^2},
\end{align*}
where $a_v, b_v\in \mathbb{C}$. Thus, as in \eqref{3.3}, we obtain
\begin{align*}
Z_v(s,\Phi_v,\chi_v)=\mathbf{1}_{\text{$\chi_v$ unram}}\prod_{w\mid v}L_w(s+1/2,\chi_w).
\end{align*}

\item Suppose $v<\infty$ is inert or ramified in $E$. By the definition of $\Phi_v^*$ in \eqref{eq3.1}, we have
\begin{equation}\label{e3.3}
\Phi_v((0,1)\iota(x_v))=q_v^{d_{E_v}/2}\mathbf{1}_{\mathcal{O}_{E_v}}(x_v),
\end{equation}
and hence
\begin{align*}
Z_v(s,\Phi_v,\chi_v)=\mathbf{1}_{\text{$\chi_v$ unram}}L_v(s+1/2,\chi_v).
\end{align*} 

\item Suppose $v<\infty$ splits in $E$ and $v\nmid\mathfrak n$. By construction, \eqref{e3.3} also holds in this case. Hence,
\begin{align*}
Z_v(s,\Phi_v,\chi_v)=\mathbf{1}_{\text{$\chi_v$ unram}}\prod_{w\mid v}L_w(s+1/2,\chi_w).
\end{align*}

\item Suppose $v\mid\mathfrak n$ with $e_v(\mathfrak{n})=r_v$. Write $x_v=a_v+\theta_vb_v$. We have  
\begin{equation}\label{3.14}
\Phi_v^{(1)}((0,1)\iota(x_v))=q_v^{r_v/2}\mathbf{1}_{\mathcal{O}_v}((a_v+\theta_vb_v)\varpi_v^{-r_v})\mathbf{1}_{\mathcal{O}_v}(a_v-\theta_vb_v).
\end{equation}

Let $\alpha_v:=a_v+\theta_vb_v$ and $\beta_v=a_v-\theta_vb_v$. Then $\chi_v(x_v)=\chi_v'(\alpha_v)\chi_v'^{-1}(\beta_v)$ for some $\chi_v'$. By the change of variables
\begin{align*}
(\alpha_v,\beta_v)\mapsto (\varpi_v^{r_v}\alpha_v,\beta_v),
\end{align*}
we obtain
\begin{multline*}
\int_{E_v^{\times}}
\Phi_v^{(1)}((0,1)\iota(x_v))\chi_v(x_v)|x_v|_{E_v}^{\frac{1}{2}+s}d^{\times}x_v=q_v^{-r_vs}\mathbf{1}_{\text{$\chi_v$ unramified}}\\
\chi_v'(\varpi_v^{r_v})\prod_{w\mid v}L_w(s+1/2,\chi_w).
\end{multline*}

The remaining case $j=2$ is treated similarly. Therefore, 
\begin{align*}
Z_v(s,\Phi_v,\chi_v)=q_v^{-r_vs}\big(\chi_v'(\varpi_v^{r_v})+\chi_v'^{-1}(\varpi_v^{r_v})\big)\mathbf{1}_{\text{$\chi_v$ unramified}}\prod_{w\mid v}L_w(s+1/2,\chi_w).
\end{align*} 
\end{itemize}

Combining the preceding local computations, we obtain
\begin{align*}
Z(s,\Phi_2,\chi)=Z(s,\Phi,\chi)
=N_F(\mathfrak{n})^{-s}
\lambda_{\chi}(\mathfrak{n})
\mathbf{1}_{\chi\in \widehat{\mathcal{C}}_{E/F}^{\text{un}}}
\Lambda(1/2+s,\chi)
\end{align*}
for $\Re(s)\gg 1$. The identity for general $s$ then follows by meromorphic continuation. The formula for $Z(s,\Phi_1,\chi)$ is obtained in the same way.
\end{proof}

\section{The Terms \texorpdfstring{$Z(\mathbf{s},\boldsymbol{\Phi})$}{} and \texorpdfstring{$Z^{\Diamond}(\mathbf{s},\boldsymbol{\Phi})$}{}}\label{sec4}

In this section we explicitly compute the integrals
$Z(\mathbf{s},\boldsymbol{\Phi})$ and $ Z^{\Diamond}(\mathbf{s},\boldsymbol{\Phi})$
for the automorphic data constructed in \textsection\ref{sect3}. 

At every untwisted finite
place let $g_{0,v}$ be the integral-basis matrix defined in \textsection\ref{sec3.2}; at a
twisting prime use the split eigenbasis matrix in \eqref{eq3.3}. At a
real nonsplit place use $\diag(\sqrt{|\tau_v|},1)$, and at a split
Archimedean place use
$\left(\begin{smallmatrix}\theta_v&-\theta_v\\1&1\end{smallmatrix}\right)$.
Denote their adelic product by $g_0$. The discriminant of the integral basis
$1,u_{0,v}+u_v\theta$ gives
\begin{equation}\label{eq4.det}
|\det g_0|=2^{-r_{\mathrm{ns}}}D_{E/F}^{1/2}.
\end{equation}
If $h_j^0,W_j^0$ are the translated data and
$\phi_0=\pi(g_0^{-1})\phi$, then
\begin{align*}
\Psi(\overline{W_\phi},W_1,h_2)=|\det g_0|^{-1-s_1-s_2}\Psi(\overline{W_{\phi_0}},W_1^0,h_2^0).
\end{align*}
The paired toric functional is
$\int_{\mathcal C_{E/F}}\phi_0(\iota(x)g_0)d^\times x$; hence the spectral
summand is independent of the phase of $\phi_0$.

We package the geometric coefficient as
\begin{equation}\label{eq4.Z}
\mathscr Z(s_1,s_2)=\frac{2^{\omega_F(\mathfrak n)}D_E^{1/2}\Lambda_F(1+2s_1)\Lambda_F(1+2s_2)\Lambda_E(1+s_1+s_2)}
{\pi^{r_{\mathrm{ns}}+r_2}D_F^{1/2}N_F(\mathfrak n)^{1/2+s_1+s_2}\Lambda_F(2+2s_1+2s_2)}\cdot H_{\mathfrak n}(s_1,s_2),
\end{equation}
where $H_{\mathfrak{n}}(s_1,s_2)$ is defined by \eqref{e3.6} in \textsection\ref{sec3.6}. 

\subsection{Calculation of \texorpdfstring{$Z(\mathbf{s},\boldsymbol{\Phi})$}{}}
Recall that
\begin{align*}
Z(\mathbf{s},\boldsymbol{\Phi}):=\int_{\mathcal{C}_{E/F}}\Eis(h_1(\cdot,s_1)h_2(\cdot,s_2))(\iota(x))d^{\times}x.
\end{align*}

Suppose $\Re(s_1)\gg 1$ and $\Re(s_2)\gg 1$. Unfolding as in \eqref{eq2.4}, we obtain
\begin{align*}
Z(\mathbf{s},\boldsymbol{\Phi})=\prod_{v\in \Sigma_F}Z_v(\mathbf{s},\Phi_{1,v},\Phi_{2,v}),
\end{align*}
where  
\begin{align*}
Z_v(\mathbf{s},\Phi_{1,v},\Phi_{2,v}):=\int_{E_v^{\times}}h_{1,v}(\iota(x_v),s_1)\Phi_{2,v}((0,1)\iota(x_v))|\det \iota(x_v)|_{v}^{\frac{1}{2}+s_2}
d^{\times}x_v.
\end{align*}

\begin{prop}\label{prop4.1}
For $\mathbf{s}\in \mathbb{C}^2$, we have the following identity of meromorphic functions:
\begin{equation}\label{fc4.1}
Z(\mathbf{s},\Phi_{1},\Phi_{2})=\mathscr Z(s_1,s_2).
\end{equation}
\end{prop}
\begin{proof}
Suppose $\Re(s_1)\gg 1$ and $\Re(s_2)\gg 1$. By definition of $h_{1,v}(\cdot,s_1)$ we have 
\begin{multline*}
Z_v(\mathbf{s},\Phi_{1,v},\Phi_{2,v})=\int_{E_v^{\times}}
\int_{F_v^{\times}}
\Phi_{1,v}((0,t)\iota(x_v))|t|_v^{1+2s_1}d^{\times}t\\
\Phi_{2,v}((0,1)\iota(x_v))|x_v|_{E_v}^{1+s_1+s_2}d^{\times}x_v.
\end{multline*}

We compute $Z_v(\mathbf{s},\Phi_{1,v},\Phi_{2,v})$ according to the definitions of
$\Phi_{1,v}$ and $\Phi_{2,v}$.

\begin{itemize}
\item Suppose $v\in \Sigma_{F,\infty}^{\mathrm{r,sp}}$. Write
$\alpha_v=a_v+\theta_vb_v$ and $\beta_v=a_v-\theta_vb_v$, where $\theta_v=\sqrt{\tau_v}$. In parallel with \eqref{3.3}, we obtain
\begin{multline}\label{4.1}
Z_v(\mathbf{s},\Phi_{1,v},\Phi_{2,v})=\int_{\mathbb{R}^{\times}}\int_{\mathbb{R}^{\times}}\int_{\mathbb{R}^{\times}}
e^{-\pi(1+t^2)\alpha_v^2}e^{-\pi(1+t^2)\beta_v^2}\\
|\alpha_v|_{v}^{1+s_1+s_2}|\beta_v|_{v}^{1+s_1+s_2}d^{\times}\alpha_vd^{\times}\beta_v|t|_v^{1+2s_1}d^{\times}t.
\end{multline}

The integral in \eqref{4.1} evaluates to
\begin{align*}
Z_v(\mathbf{s},\Phi_{1,v},\Phi_{2,v})=\prod_{w\mid v}\zeta_{E_w}(1+s_1+s_2)\int_{\mathbb{R}^{\times}}\frac{|t|_v^{1+2s_1}}{(1+t^2)^{1+s_1+s_2}}d^{\times}t.
\end{align*}

Since $\Re(s_1)\gg 1$ and $\Re(s_2)\gg 1$, we have
\begin{equation}\label{eq4.2}
\int_{\mathbb{R}^{\times}}\frac{|t|_v^{1+2s_1}}{(1+t^2)^{1+s_1+s_2}}d^{\times}t=\frac{\Gamma_{\mathbb{R}}(1+2s_1)\Gamma_{\mathbb{R}}(1+2s_2)}{\Gamma_{\mathbb{R}}(2+2s_1+2s_2)},
\end{equation}
where $\Gamma_{\mathbb{R}}$ is defined by \eqref{eq1.1}. Therefore,
\begin{equation}\label{eq4.3}
Z_v(\mathbf{s},\Phi_{1,v},\Phi_{2,v})=\frac{\Gamma_{\mathbb{R}}(1+2s_1)\Gamma_{\mathbb{R}}(1+2s_2)}{\Gamma_{\mathbb{R}}(2+2s_1+2s_2)}\prod_{w\mid v}\zeta_{E_w}(1+s_1+s_2).
\end{equation}

\item Suppose $v\in \Sigma_{F,\infty}^{\mathrm{r,ns}}$. In parallel with \eqref{e3.1}, we obtain
\begin{equation}\label{e4.2}
Z_v(\mathbf{s},\Phi_{1,v},\Phi_{2,v})=\pi^{-2}\int_{\mathbb{R}^{\times}}|t|_v^{1+2s_1}\int_{\mathbb{C}^{\times}}
e^{-2\pi(1+t^2)|z|_{\mathbb{C}}}|z|_{\mathbb{C}}^{1+s_1+s_2}d^{\times}zd^{\times}t.
\end{equation}

By a change of variables and \eqref{eq4.2}, it follows from \eqref{e4.2} that
\begin{equation}\label{eq4.4}
Z_v(\mathbf{s},\Phi_{1,v},\Phi_{2,v})=\pi^{-1}\frac{\Gamma_{\mathbb{R}}(1+2s_1)\Gamma_{\mathbb{R}}(1+2s_2)}{\Gamma_{\mathbb{R}}(2+2s_1+2s_2)}\zeta_{E_v}(1+s_1+s_2).
\end{equation}

\item Suppose $v\in \Sigma_{F,\infty}^{\mathrm{c}}$. As in \eqref{eq4.3} and \eqref{eq4.4}, we obtain
\begin{align*}
Z_v(\mathbf{s},\Phi_{1,v},\Phi_{2,v})=\pi^{-2}\prod_{w\mid v}\zeta_{E_w}(1+s_1+s_2)\int_{\mathbb{C}^{\times}}\frac{|t|_v^{1+2s_1}}{(1+|t|_v)^{2+2s_1+2s_2}}d^{\times}t.
\end{align*}
Recall that $|t|_v=t\bar{t}$. It follows from a change of variable and the Beta function that 
\begin{align*}
\int_{\mathbb{C}^{\times}}\frac{|t|_v^{1+2s_1}}{(1+|t|_v)^{2+2s_1+2s_2}}d^{\times}t=\frac{2\pi\Gamma(1+2s_1)\Gamma(1+2s_2)}{\Gamma(2+2s_1+2s_2)}.
\end{align*}

With $\Gamma_{\mathbb{C}}(s)=2(2\pi)^{-s}\Gamma(s)$, this is 
\begin{equation}\label{eq4.8}
Z_v(\mathbf{s},\Phi_{1,v},\Phi_{2,v})=\pi^{-1}\frac{\Gamma_{\mathbb{C}}(1+2s_1)\Gamma_{\mathbb{C}}(1+2s_2)}{\Gamma_{\mathbb{C}}(2+2s_1+2s_2)}\prod_{w\mid v}\zeta_{E_w}(1+s_1+s_2).
\end{equation}

\item Suppose $v<\infty$ is inert or ramified in $E$. By construction,
\begin{equation}\label{4.6}
Z_v(\mathbf{s},\Phi_{1,v},\Phi_{2,v})=\sum_{l\in \mathbb{Z}}\int_{\varpi_v^l\mathcal{O}_v^{\times}}Z_v(t;\mathbf{s},\Phi_{1,v},\Phi_{2,v})|t|_v^{1+2s_1}d^{\times}t,
\end{equation}
where
\begin{align*}
Z_v(t;\mathbf{s},\Phi_{1,v},\Phi_{2,v}):=\int_{E_v^{\times}}
\Phi_{1,v}((0,t)\iota(x_v))\Phi_{2,v}((0,1)\iota(x_v))|x_v|_{E_v}^{1+s_1+s_2}d^{\times}x_v.
\end{align*}

By the definitions of $\Phi_{1,v}$ and $\Phi_{2,v}$, we have the following.

\begin{itemize}
\item Suppose $l\geq 0$. Then for $t\in \varpi_v^l\mathcal{O}_v^{\times}$, 
\begin{align*}
Z_v(t;\mathbf{s},\Phi_{1,v},\Phi_{2,v})=q_v^{d_{E_v}}\int_{\mathcal{O}_{E_v}\cap E_v^{\times}}
|x_v|_{E_v}^{1+s_1+s_2}d^{\times}x_v=q_v^{d_{E_v}/2}\zeta_{E_v}(1+s_1+s_2).
\end{align*}

\item Suppose $l\leq -1$. Then for $t\in \varpi_v^l\mathcal{O}_v^{\times}$,
\begin{align*}
Z_v(t;\mathbf{s},\Phi_{1,v},\Phi_{2,v})=q_v^{d_{E_v}}\int_{\varpi_v^{-l}\mathcal{O}_{E_v}\cap E_v^{\times}}
|x_v|_{E_v}^{1+s_1+s_2}d^{\times}x_v=\frac{q_v^{d_{E_v}/2}\zeta_{E_v}(1+s_1+s_2)}{q_v^{(2+2s_1+2s_2)|l|}}.
\end{align*}
\end{itemize}

Substituting these identities into \eqref{4.6} yields
\begin{align*}
Z_v(\mathbf{s},\Phi_{1,v},\Phi_{2,v})=q_v^{\frac{d_{E_v}-d_{F_v}}{2}}\zeta_{E_v}(1+s_1+s_2)\Big[\sum_{l\geq 0}q_v^{-(1+2s_1)l}+\sum_{l\leq -1}q_v^{(1+2s_2)l}\Big].
\end{align*}
Here the factor $q_v^{-d_{F_v}/2}$ comes from the volume of $\mathcal{O}_v^{\times}$. Since
\begin{align*}
\sum_{l\geq 0}q_v^{-(1+2s_1)l}+\sum_{l\leq -1}q_v^{(1+2s_2)l}=\frac{\zeta_{F_v}(1+2s_1)\zeta_{F_v}(1+2s_2)}{\zeta_{F_v}(2+2s_1+2s_2)}.
\end{align*}
Therefore,
\begin{equation}\label{4.7}
Z_v(\mathbf{s},\Phi_{1,v},\Phi_{2,v})=\frac{\zeta_{F_v}(1+2s_1)\zeta_{F_v}(1+2s_2)}{\zeta_{F_v}(2+2s_1+2s_2)}q_v^{(d_{E_v}-d_{F_v})/2}\zeta_{E_v}(1+s_1+s_2).
\end{equation}

\item Suppose $v<\infty$ splits in $E$ and $v\nmid\mathfrak n$. Similar to \eqref{4.7}, we have
\begin{equation}\label{e4.10}
Z_v(\mathbf{s},\Phi_{1,v},\Phi_{2,v})=\frac{\zeta_{F_v}(1+2s_1)\zeta_{F_v}(1+2s_2)}{\zeta_{F_v}(2+2s_1+2s_2)}q_v^{(d_{E_v}-d_{F_v})/2}\prod_{w\mid v}\zeta_{E_w}(1+s_1+s_2).
\end{equation}

\item Suppose $v\mid\mathfrak n$. In parallel with \eqref{4.6}, we have the decomposition
\begin{equation}\label{4.8}
Z_v(\mathbf{s},\Phi_{1,v},\Phi_{2,v})=\sum_{l\in\mathbb Z}\int_{\varpi_v^l\mathcal{O}_v^{\times}}Z_v(t;\mathbf{s},\Phi_{1,v},\Phi_{2,v})|t|_v^{1+2s_1}d^{\times}t.
\end{equation}

By the definitions of $\Phi_{1,v}$ and $\Phi_{2,v}$, we have the following.

\begin{itemize}
\item Suppose $l\geq 0$. Then
\begin{align*}
Z_v(t;\mathbf{s},\Phi_{1,v},\Phi_{2,v})=\int_{E_v^{\times}}\Phi_{2,v}((0,1)\iota(x_v))
|x_v|_{E_v}^{1+s_1+s_2}d^{\times}x_v.
\end{align*}
By the definition \eqref{3.2}, together with a change of variables, we obtain 
\begin{align*}
Z_v(t;\mathbf{s},\Phi_{1,v},\Phi_{2,v})=2q_v^{-r_v(1/2+s_1+s_2)}\prod_{w\mid v}\zeta_{E_w}(1+s_1+s_2).
\end{align*}

\item Suppose $l\leq -1$. Then
\begin{equation}\label{4.14}
Z_v(t;\mathbf{s},\Phi_{1,v},\Phi_{2,v})=\int_{\varpi_v^{-l}\mathcal{O}_{E_v}\cap E_v^{\times}}\Phi_{2,v}((0,1)\iota(x_v))
|x_v|_{E_v}^{1+s_1+s_2}d^{\times}x_v.
\end{equation}

Substituting the definition \eqref{3.2} into \eqref{4.14} yields 
\begin{align*}
Z_v(t;\mathbf{s},\Phi_{1,v},\Phi_{2,v})=R_v(l)\cdot 
\prod_{w\mid v}\zeta_{E_w}(1+s_1+s_2),
\end{align*}
where
\begin{align*}
R_v(l):=2q_v^{r_v/2}
q_v^{-(1+s_1+s_2)\max\{|l|,r_v\}}q_v^{-(1+s_1+s_2)|l|}.
\end{align*}
\end{itemize}

By a straightforward calculation, 
\begin{multline*}
2q_v^{-r_v(1/2+s_1+s_2)}\sum_{l\geq 0}q_v^{-(1+2s_1)l}
+\sum_{l\geq 1}
R_v(l)q_v^{(1+2s_1)l}=\\
2H_v(s_1,s_2)\frac{\zeta_{F_v}(1+2s_1)\zeta_{F_v}(1+2s_2)}{q_v^{(1/2+s_1+s_2)r_v}\zeta_{F_v}(2+2s_1+2s_2)},
\end{multline*}
where $H_v(s_1,s_2)$ is defined in \textsection\ref{3.2.1}.  

It follows from the preceding computation and \eqref{4.8} that
\begin{equation}\label{e4.12}
Z_v(\mathbf{s},\Phi_{1,v},\Phi_{2,v})=H_v(s_1,s_2)\frac{2\zeta_{F_v}(1+2s_1)\zeta_{F_v}(1+2s_2)\zeta_{E_v}(1+s_1+s_2)}{q_v^{(1/2+s_1+s_2)r_v}\zeta_{F_v}(2+2s_1+2s_2)}.
\end{equation}
\end{itemize}

Multiplying the local calculations, which already include the Gaussian
scalars of \textsection\ref{sec3.1}, and using \eqref{eq4.det} proves
\eqref{fc4.1} for $\Re(s_1),\Re(s_2)\gg1$. The general identity follows
by meromorphic continuation.
\end{proof}

\subsection{Calculation of \texorpdfstring{$Z^{\Diamond}(\mathbf{s},\boldsymbol{\Phi})$}{}}
Recall that
\begin{align*}
Z^{\Diamond}(\mathbf{s},\boldsymbol{\Phi}):=\int_{\mathcal{C}_{E/F}}\Eis(h_1^{\Diamond}(\cdot,s_1)h_2(\cdot,s_2))(\iota(x))d^{\times}x.
\end{align*}

Suppose $\Re(s_1)\gg 1$ and $\Re(s_2)\gg 1$. Unfolding as in \eqref{eq2.4}, we obtain
\begin{align*}
Z^{\Diamond}(\mathbf{s},\boldsymbol{\Phi})=\prod_{v\in \Sigma_F}Z_v^{\Diamond}(\mathbf{s},\Phi_{1,v},\Phi_{2,v}),
\end{align*}
where  
\begin{align*}
Z_v^{\Diamond}(\mathbf{s},\Phi_{1,v},\Phi_{2,v}):=\int_{E_v^{\times}}h_{1,v}^{\Diamond}(\iota(x_v),s_1)\Phi_{2,v}((0,1)\iota(x_v))|\det \iota(x_v)|_{v}^{\frac{1}{2}+s_2}
d^{\times}x_v.
\end{align*}

The normalized intertwining relations are
\begin{equation}\label{eq4.intertwine}
h_1^\Diamond(\cdot,s)=D_E^{-s}h_1(\cdot,-s),\qquad h_2^\Diamond(\cdot,s)=D_E^{-s}N_F(\mathfrak n)^{-2s}h_2(\cdot,-s).
\end{equation}

\begin{prop}\label{prop4.2}
For $\mathbf{s}\in \mathbb{C}^2$, we have the following identity of meromorphic functions:
\begin{align*}
Z^{\Diamond}(\mathbf{s},\Phi_{1},\Phi_{2})=D_E^{-s_1}\mathscr Z(-s_1,s_2).
\end{align*}
\end{prop}

\begin{proof}
Assume first that
\begin{align*}
0<\Re(s_1)<1/2,\quad \Re(s_2-s_1)\gg 1.
\end{align*}
Define
\begin{align*}
\tilde{h}_v(g_v,s_1)
=
|\det g_v|_v^{s_1+1/2}
\int_{F_v^{\times}}
\Phi_v^{\circ}((0,t_v)g_v)|t_v|_v^{1+2s_1}d^{\times}t_v.
\end{align*}
Let $g_{0,v}$ be the matrix used in \eqref{eq4.det}, and define
$b_v=q_v^{d_{E_v}/2}$ at finite places and $b_v=1$ at infinity.
Then, for all $g_v\in G(F_v)$,
\begin{equation}\label{equ4.14}
h_{1,v}(g_v,s_1)
=
b_v|\det g_{0,v}|_v^{-s_1-1/2}
\tilde{h}_v(g_vg_{0,v},s_1).
\end{equation}

For each $v\in \Sigma_F$, let $\tilde{h}_v^{\Diamond}(g_v,s_1)$ denote the image of
$\tilde{h}_v(g_v,s_1)$ under the intertwining operator. Then
$\tilde{h}_v^{\Diamond}(g_v,s_1)$ is a spherical section of
$|\cdot|_v^{-s_1}\boxplus |\cdot|_v^{s_1}$. On the other hand,
$\tilde{h}_v(g_v,-s_1)$ is also a spherical section of the same induced
representation. Since the space of spherical vectors is one-dimensional, we have
\begin{equation}\label{equ4.16}
\tilde{h}_v^{\Diamond}(g_v,s_1)
=
\frac{\tilde{h}_v^{\Diamond}(I_2,s_1)}
{\tilde{h}_v(I_2,-s_1)}
\tilde{h}_v(g_v,-s_1),
\quad g_v\in G(F_v).
\end{equation}

\begin{itemize}
\item Suppose $v\in \Sigma_{F,\infty}$. Then
\begin{align*}
\tilde{h}_v^{\Diamond}(I_2,s_1)&=
\begin{cases}
\pi^{-1}2^{-1/2-s_1}\Gamma_{\mathbb R}(2s_1),&v\text{ real non-split},\\
\Gamma_{\mathbb R}(2s_1),&v\text{ real split},\\
\pi^{-1}\Gamma_{\mathbb C}(2s_1),&v\text{ complex},
\end{cases}
\\[3pt]
\tilde{h}_v(I_2,-s_1)&=
\begin{cases}
\pi^{-1}2^{-1/2+s_1}\Gamma_{\mathbb R}(1-2s_1),&v\text{ real non-split},\\
\Gamma_{\mathbb R}(1-2s_1),&v\text{ real split},\\
\pi^{-1}\Gamma_{\mathbb C}(1-2s_1),&v\text{ complex}.
\end{cases}
\end{align*}
Thus \eqref{equ4.16} becomes
\begin{align*}
\tilde{h}_v^{\Diamond}(g_v,s_1)
=
2^{-2s_1\mathbf1_{v\text{ real non-split}}}\zeta_{F_v}(2s_1)\zeta_{F_v}(1-2s_1)^{-1}
\tilde{h}_v(g_v,-s_1).
\end{align*}
Together with \eqref{equ4.14}, this implies
\begin{align*}
h_{1,v}^{\Diamond}(g_v,s_1)
=
|\det g_{0,v}|_v^{-2s_1}2^{-2s_1\mathbf1_{v\text{ real non-split}}}
\zeta_{F_v}(2s_1)\zeta_{F_v}(1-2s_1)^{-1}
h_{1,v}(g_v,-s_1).
\end{align*}
The local zeta integrals then take the form
\begin{multline}\label{eq4.16}
Z_v^{\Diamond}(\mathbf{s},\Phi_{1,v},\Phi_{2,v})
=
|\det g_{0,v}|_v^{-2s_1}2^{-2s_1\mathbf1_{v\text{ real non-split}}}\\
\zeta_{F_v}(2s_1)\zeta_{F_v}(1-2s_1)^{-1}
Z_v((-s_1,s_2),\Phi_{1,v},\Phi_{2,v}).
\end{multline}

Combining \eqref{eq4.3}, \eqref{eq4.4}, \eqref{eq4.8}, and \eqref{eq4.16} gives
\begin{equation}\label{4.19}
\prod_{v\in \Sigma_{F,\infty}}
Z_v^{\Diamond}(\mathbf{s},\Phi_{1,v},\Phi_{2,v})
=
\frac{
\pi^{-r_{\mathrm{ns}}-r_2}2^{-2ds_1}
\Lambda_{F_{\infty}}(2s_1)
\Lambda_{F_{\infty}}(1+2s_2)
\Lambda_{E_{\infty}}(1-s_1+s_2)
}{
|\tau_{\infty}|_{\infty}^{s_1}
\Lambda_{F_{\infty}}(2-2s_1+2s_2)
}.
\end{equation}

\item Suppose $v\in \Sigma_{F,\fin}$. Then
\begin{align*}
\tilde{h}_v^{\Diamond}(I_2,s_1)
=
q_v^{-d_{F_v}}\zeta_{F_v}(2s_1),
\quad
\tilde{h}_v(I_2,-s_1)
=
q_v^{-d_{F_v}/2}\zeta_{F_v}(1-2s_1).
\end{align*}
With these identities, \eqref{equ4.16} becomes
\begin{align*}
\tilde{h}_v^{\Diamond}(g_v,s_1)
=
q_v^{-d_{F_v}/2}
\zeta_{F_v}(2s_1)\zeta_{F_v}(1-2s_1)^{-1}
\tilde{h}_v(g_v,-s_1).
\end{align*}
Applying \eqref{equ4.14}, we obtain
\begin{align*}
h_{1,v}^{\Diamond}(g_v,s_1)
=
|\det g_{0,v}|_v^{-2s_1}q_v^{-d_{F_v}/2}
\zeta_{F_v}(2s_1)\zeta_{F_v}(1-2s_1)^{-1}
h_{1,v}(g_v,-s_1).
\end{align*}
Therefore
\begin{equation}\label{eq4.17}
Z_v^{\Diamond}(\mathbf{s},\Phi_{1,v},\Phi_{2,v})
=
|\det g_{0,v}|_v^{-2s_1}q_v^{-d_{F_v}/2}
\frac{\zeta_{F_v}(2s_1)}{\zeta_{F_v}(1-2s_1)}
Z_v((-s_1,s_2),\Phi_{1,v},\Phi_{2,v}).
\end{equation}
It follows from \eqref{4.6}, \eqref{e4.10}, \eqref{e4.12}, and
\eqref{eq4.17} that
\begin{align*}
\prod_{v\in \Sigma_{F,\fin}}
Z_v^{\Diamond}(\mathbf{s},\Phi_{1,v},\Phi_{2,v})
\end{align*}
has a holomorphic continuation to the region
\begin{align*}
\Re(s_1)\gg 1,\quad \Re(s_2-s_1)\gg 1,
\end{align*}
and in this region is equal to
\begin{equation}\label{e4.19}
\frac{
2^{\omega_F(\mathfrak{n})+2ds_1}N_F(\mathfrak{a})^{2s_1}D_E^{1/2}
\zeta_F(2s_1)\zeta_F(1+2s_2)\zeta_E(1-s_1+s_2)
}{
N_F(\mathfrak{n})^{1/2-s_1+s_2}
D_F
\zeta_F(2-2s_1+2s_2)
}
H_{\mathfrak{n}}(-s_1,s_2).
\end{equation}
\end{itemize}

By definition, $|\tau_{\infty}|_{\infty}=D_{E/F}N_F(\mathfrak{a})^2$. 
Combining \eqref{4.19} and \eqref{e4.19},  and then using the meromorphic
continuation of the Dedekind zeta functions, gives
\begin{align*}
Z^{\Diamond}(\mathbf{s},\Phi_{1},\Phi_{2})
=
\frac{
\pi^{-r_{\mathrm{ns}}-r_2}D_E^{1/2}
\Lambda_F(2s_1)\Lambda_F(1+2s_2)\Lambda_E(1-s_1+s_2)
}{
2^{-\omega_F(\mathfrak{n})}
D_{E/F}^{s_1}
N_F(\mathfrak{n})^{1/2-s_1+s_2}
D_F
\Lambda_F(2-2s_1+2s_2)
}
H_{\mathfrak{n}}(-s_1,s_2).
\end{align*} 
Therefore, Proposition \ref{prop4.2} follows from the functional equation,
the relation $D_E=D_{E/F}D_F^2$, and the normalized identity
\eqref{eq4.intertwine}. 
\end{proof}

\begin{remark}
Uniqueness of the spherical vector gives \eqref{eq4.16} and
\eqref{eq4.17}, reducing the computation of $Z_v^\Diamond$ to that of $Z_v$.
\end{remark}

\section{Estimates of \texorpdfstring{$\mathcal{P}_{\cusp}(\mathbf{s},\boldsymbol{\Phi})$}{} and \texorpdfstring{$\mathcal{P}_{\Eis}(\mathbf{s},\boldsymbol{\Phi})$}{}}\label{sec5}

\subsection{Estimate of \texorpdfstring{$\mathcal{P}_{\cusp}(\mathbf{s},\boldsymbol{\Phi})$}{}}\label{sec5.1}
 Recall that
\begin{equation}\label{eq5.1}
\mathcal{P}_{\cusp}(\mathbf{s},\boldsymbol{\Phi})
=
\sum_{\substack{\pi\in \mathcal{A}_0([G],\mathbf{1})\\
\phi\in\mathfrak{B}(\pi)}}
\widetilde{\Psi}
\bigl(\overline{W_{\phi}},W_1(\cdot,s_1),h_2(\cdot,s_2)\bigr)
\int_{\mathcal{C}_{E/F}}\phi(x)d^{\times}x .
\end{equation}

The main result in this section is the following. 
\begin{prop}\label{prop5.1}
Let $|\Re(s_1)|+|\Re(s_2)|\leq \varepsilon$ and $0<\varepsilon<10^{-3}$. Then
\begin{multline}\label{eq5.2}
\mathcal{P}_{\cusp}(\mathbf{s},\boldsymbol{\Phi})\ll D_E^{\frac{1}{4}+10\varepsilon}\sum_{\mathfrak{n}\subset \mathfrak{m}}N_F(\mathfrak{m})^{-\frac{1}{2}+1000\varepsilon}\sum_{\pi\in \mathcal{F}(\mathfrak{m},\mathbf{1})}\mathcal V_\infty(\pi)L(1/2,\pi_E)^{\frac{1}{2}}\\
|L_F(1/2+s_1+s_2,\pi)L_F(1/2-s_1+s_2,\pi)|,
\end{multline}	
where the implied constant depends on $F$, $\varepsilon$, $s_1$, and $s_2$.

Here $\mathcal V_\infty(\pi)=\exp(-c\sum_{v\mid\infty}|\nu_{\pi_v}|)$ for a sufficiently small fixed $c>0$, and $\nu_{\pi_v}$ is the spherical Archimedean parameter. The $L$-functions in this bound are finite. The family $\mathcal{F}(\mathfrak{m},\mathbf{1})$ is defined in \textsection\ref{sec3.1.3}.
\end{prop}

\begin{remark}
The implied constant in \eqref{eq5.2} is independent of $\mathfrak{a}$
(see \eqref{e5.4} in \textsection\ref{5.1.3}), and hence independent of the
choice of $\tau$ used in the parametrization $E=F(\sqrt{\tau})$.
\end{remark}

\subsubsection{Waldspurger's Formula}\label{5.1.1}
We may take $\mathfrak{B}(\pi)\simeq \otimes_{v\in \Sigma_F}\mathfrak{B}_{\pi_v}$, where $\mathfrak{B}_{\pi_v}$ is an orthonormal basis of vectors in $\pi_v$. Thus an automorphic form $\phi\in\mathfrak{B}(\pi)$ corresponds, up to our mild abuse of notation, to a pure tensor
\begin{align*}
\otimes_{v\in \Sigma_F}\phi_v\in \pi=\otimes_{v\in \Sigma_F}'\pi_v.
\end{align*}
Let
\begin{align*}
\langle \pi(g)\phi,\phi\rangle=\prod_{v\in \Sigma_F}\langle \pi_v(g_v)\phi_v,\phi_v\rangle_v,\ \ \ g=\otimes_{v\in \Sigma_F}'g_v\in G(\mathbb{A}_F)
\end{align*}
be the associated matrix coefficient.

For a unit pure tensor, the squared Waldspurger formula, with the measures of \textsection\ref{1.1.1}, gives the bound
\begin{equation}\label{e5.2}
\Big|\int_{\mathcal{C}_{E/F}}\phi(\iota(x))d^{\times}x\Big|^2\ll_{F,\varepsilon}D_E^\varepsilon\frac{\Lambda_F(2)\Lambda(1/2,\pi_E)}{\Lambda(1,\pi,\Ad)}\prod_{v\in \Sigma_F}\mathcal{I}_v^{\sharp}(\phi_v)^2,
\end{equation}
where the fixed Haar scalar and the quadratic Hecke $L$-value factors arising from the quotient measure are absorbed into $D_E^\varepsilon$. No change of local measure is made. Here $\pi_E$ denotes the base change of $\pi$, and
\begin{equation}\label{fc5.2}
\mathcal I_v^{\sharp}(\phi_v)
:=\sqrt{\frac{L_v(1,\pi_v,\mathrm{Ad})L_v(1,\eta_v)}
{\zeta_{F_v}(2)L_v(1/2,\pi_{E,v})}
\int_{F_v^{\times}\backslash E_v^{\times}}
\frac{\langle \pi_v(\iota(x_v))\phi_v,\phi_v\rangle_v}
{\langle \phi_v,\phi_v\rangle_v}
d^{\times}x_v}.
\end{equation}
At a good unramified spherical place, $\mathcal I_v^{\sharp}(\phi_v)=1$; the factor $L_v(1,\eta_v)$ cancels $1+q_v^{-1}$ in the inert case. See \cite[\textsection 5--\textsection 6]{II10}. Choose normalized linear toric functionals $\ell_v$ with $|\ell_v(\phi_v)|=\mathcal I_v^{\sharp}(\phi_v)$ on unit vectors. Then the global period factors as $B_\pi\prod_v\ell_v(\phi_v)$, with $|B_\pi|$ bounded by \eqref{e5.2}; the linear functionals account for the phases of all basis vectors.

\subsubsection{Rankin--Selberg Integrals}

Recall the adelic change-of-basis matrix $g_0$ from \textsection\ref{sec4}. By \eqref{eq4.det},
\begin{equation}\label{f5.5}
|\det g_0|=2^{-r_{\mathrm{ns}}}D_{E/F}^{1/2}.
\end{equation}

At every untwisted finite place, \eqref{eq3.1} reads
\begin{align*}
\Phi_v^*(t_1,t_2)
=
\Phi_v^\circ((t_1,t_2)g_{0,v}),
\quad t_1,t_2\in F_v.
\end{align*}

At a twisting prime the split change of basis in \eqref{eq3.3} lies in
$K_v$ and gives the same spherical maximal-order test. Changing variables
$x\mapsto x(g_0)^{-1}$ in the Rankin--Selberg integral therefore gives
\begin{equation}\label{e5.1}
|\det g_0|^{1+s_1+s_2}
\widetilde{\Psi}
\bigl(\overline{W_{\phi^{\dag}}},W_1(\cdot,s_1),h_2(\cdot,s_2)\bigr)
=
\widetilde{\Psi}
\bigl(\overline{W_{\phi}},W_1^{\dag}(\cdot,s_1),h_2^{\dag}(\cdot,s_2)\bigr),
\end{equation}
where $\phi^{\dag}:=\pi(g_0)\phi$. Here
$W_1^{\dag}(\cdot,s_1)$ denotes the Whittaker function associated, via \eqref{2.2}, with
\begin{align*}
\Phi_1^{\dag}:=\bigotimes_{v\in \Sigma_{F,\infty}}\Phi_v^{\circ}\bigotimes_{v\in \Sigma_{F,\fin}}q_v^{d_{E_v}/2}\Phi_v^{\circ}, 
\end{align*}
and $h_2^{\dag}(\cdot,s_2)$ denotes the section associated, via \eqref{e2.2} and \eqref{3.2}, with
\begin{align*}
\Phi_2^{\dag}
:=
\bigotimes_{v\in \Sigma_{F,\infty}}\Phi_v^{\circ}
\bigotimes_{v\in \Sigma_{F,\fin}:\ e_v(\mathfrak{n})=0} q_v^{d_{E_v}/2}\Phi_v^{\circ}\bigotimes_{v\in \Sigma_{F,\fin}:\ e_v(\mathfrak{n})=r_v>0}\Phi_v^{\dag},
\end{align*}
where, for $v\in \Sigma_{F,\fin}$ with $e_v(\mathfrak{n})=r_v$, we set
\begin{align*}
\Phi_v^{\dag}(t_1,t_2):=q_v^{\frac{r_v}{2}}\Phi_v^{\circ}\left((t_1,t_2)\begin{pmatrix}
\varpi_v^{-r_v}\\
& 1
\end{pmatrix}\right)+q_v^{\frac{r_v}{2}}\Phi_v^{\circ}\left((t_1,t_2)\begin{pmatrix}
1\\
& \varpi_v^{-r_v}
\end{pmatrix}\right).
\end{align*}

For $v\in \Sigma_F$, initially in the region $\Re(s_1)\gg 1$ and
$\Re(s_2-s_1)\gg 1$, let
$\Psi_v
\bigl(\overline{W_{\phi,v}},W_{1,v}^{\dag}(\cdot,s_1),h_{2,v}^{\dag}(\cdot,s_2)\bigr)$
denote the local component of the absolutely convergent integral 
$\Psi
\bigl(\overline{W_{\phi}},W_1^{\dag}(\cdot,s_1),h_2^{\dag}(\cdot,s_2)\bigr)$,
and let
$\widetilde{\Psi}_v
\bigl(\overline{W_{\phi,v}},W_{1,v}^{\dag}(\cdot,s_1),h_{2,v}^{\dag}(\cdot,s_2)\bigr)$
be its meromorphic continuation to $(s_1,s_2)\in \mathbb{C}^2$. Define
\begin{align*}
\mathcal{P}_v^{\sharp}(\mathbf{s},\phi_v):=\frac{\widetilde{\Psi}_v
\bigl(\overline{W_{\phi,v}},W_{1,v}^{\dag}(\cdot,s_1),h_{2,v}^{\dag}(\cdot,s_2)\bigr)}{L_v(1/2+s_1+s_2,\widetilde{\pi}_v)L_v(1/2-s_1+s_2,\widetilde{\pi}_v)}.
\end{align*}
We fix the local Whittaker embeddings as follows. At finite places, choose $\delta_v$ with $\psi_v(t)=\psi_{0,v}(\delta_vt)$, where $\psi_{0,v}$ has conductor $\mathcal O_v$, and normalize the unit newvector by $W_v^\circ(a(\delta_v^{-1}))=1$, with $a(y)=\diag(y,1)$. At a real place its spherical Whittaker function is $|y|^{1/2}K_{\nu_{\pi_v}}(2\pi|y|)$, and at a complex place it is $|y|_v^{1/2}K_{2\nu_{\pi_v}}(4\pi|y|_v^{1/2})$. Let $A_\pi$ be the resulting global Fourier coefficient, so that $W_\phi=A_\pi\prod_vW_{\phi,v}$. No division by the value of an arbitrary oldvector at $I_2$ is made. By Rankin--Selberg theory,
\begin{align*}
\mathcal{P}_v^{\sharp}(\mathbf{s},\phi_v)\equiv 1
\end{align*} 
for all but finitely many places $v\in \Sigma_F$; see Lemma \ref{lem5.1} below.

\subsubsection{Factorization of $\mathcal{P}_{\cusp}(\mathbf{s},\boldsymbol{\Phi})$}\label{5.1.3}
Replacing $\phi$ by $\pi(g_0)\phi$ in \eqref{eq5.1}, and using
\eqref{f5.5} and \eqref{e5.1}, we obtain
\begin{multline}\label{e5.4}
\mathcal{P}_{\cusp}(\mathbf{s},\boldsymbol{\Phi})
=|\det g_0|^{-1-s_1-s_2}\sum_{\pi\in \mathcal{A}_0([G],\mathbf{1})}\overline{A_\pi}B_\pi\\
\Lambda_F(1/2+s_1+s_2,\widetilde{\pi})\Lambda_F(1/2-s_1+s_2,\widetilde{\pi})\prod_{v\in \Sigma_F}\sum_{\phi_v\in\mathfrak{B}_{\pi_v}}
\mathcal{P}_v^{\sharp}(\mathbf{s},\phi_v)\ell_v(\pi_v(g_v^{\dag})\phi_v).
\end{multline}

\subsubsection{Archimedean Integrals}\label{5.1.4}
\begin{lemma}\label{lem5.2}
We have the following identity of meromorphic functions:
\begin{align*}
\prod_{v\mid\infty}\mathcal{P}_v^{\sharp}(\mathbf{s},\phi_v)
=J_{\infty}(s_1,s_2)
\mathbf{1}_{\phi\ \text{is }K_{\infty}\text{-invariant}}
,
\end{align*}
where
\begin{equation}\label{eq5.8}
J_{\infty}(s_1,s_2)
:=
2^{-r_1-(1+s_1+s_2)r_{\mathrm{ns}}-2r_2}
\pi^{-2r_{\mathrm{ns}}-r_2},
\end{equation}
with the fixed Whittaker embeddings above; $K_\alpha$ denotes the $K$-Bessel function.
\end{lemma}
\begin{proof}
Assume first that $\Re(s_2)$ is sufficiently large relative to $s_1$ and the spherical parameter. Iwasawa decomposition and the Mellin integral for two $K$-Bessel functions give
\begin{multline}\label{5.6}
\widetilde{\Psi}_v
\bigl(\overline{W_{\phi,v}},W_{1,v}^{\dag}(\cdot,s_1),h_{2,v}^{\dag}(\cdot,s_2)\bigr)
=
\mathbf{1}_{\phi\ \text{is }K_v\text{-invariant}}
J_v^{\mathrm{arch}}(s_1,s_2)\\
L_v(1/2+s_1+s_2,\widetilde{\pi}_v)
L_v(1/2-s_1+s_2,\widetilde{\pi}_v).
\end{multline}
Here $J_v^{\mathrm{arch}}$ equals $2^{-2-s_1-s_2}\pi^{-2}$ at a real nonsplit place, $1/2$ at a real split place, and $1/(4\pi)$ at a complex place. To verify the scalars, \eqref{e2.2} and \eqref{2.2} give the following Gaussian section and Whittaker function.
\begin{itemize}
\item Let $v\in \Sigma_{F,\infty}^{\mathrm{r,ns}}$. Substituting
$\Phi_v^{\circ}(t_1,t_2)=\pi^{-1}e^{-2\pi(t_1^2+t_2^2)}$ into the above formulas gives
\begin{align*}
&h_{2,v}^{\dag}(I_2,s_2)=
\pi^{-1}\int_{\mathbb{R}^{\times}}
e^{-2\pi t^2}|t|^{1+2s_2}d^{\times}t=\pi^{-1}2^{-s_2-1/2}\Gamma_{\mathbb{R}}(1+2s_2),\\
&W_{1,v}^{\dag}(a(y),s_1)=\pi^{-1}2^{1/2-s_1}|y|^{1/2}K_{s_1}(2\pi|y|).
\end{align*}

\item Let $v\in \Sigma_{F,\infty}^{\mathrm{r,sp}}$. Here
$\Phi_v^{\circ}(t_1,t_2)=e^{-\pi(t_1^2+t_2^2)}$, and hence
\begin{align*}
&h_{2,v}^{\dag}(I_2,s_2)=
\int_{\mathbb{R}^{\times}}
e^{-\pi t^2}|t|^{1+2s_2}d^{\times}t=\Gamma_{\mathbb{R}}(1+2s_2),\\
&W_{1,v}^{\dag}(a(y),s_1)=2|y|^{1/2}K_{s_1}(2\pi|y|).
\end{align*}

\item Let $v\in \Sigma_{F,\infty}^{\mathrm{c}}$. In this case
$\Phi_v^{\circ}((z_1, z_2))=\pi^{-2}e^{-2\pi(z_1\overline{z_1}+z_2\overline{z_2})}$, so
\begin{align*}
&h_{2,v}^{\dag}(I_2,s_2)=
\pi^{-2}\int_{\mathbb{C}^{\times}}e^{-2\pi |t|_v}
|t|_v^{1+2s_2}d^{\times}t=\pi^{-1}\Gamma_{\mathbb C}(1+2s_2),\\
&W_{1,v}^{\dag}(a(y),s_1)=4\pi^{-1}|y|_v^{1/2}K_{2s_1}(4\pi|y|_v^{1/2}).
\end{align*}
\end{itemize}

The identity
\begin{align*}
\int_0^\infty K_\alpha(t)K_\beta(t)t^{z-1}dt=\frac{2^{z-3}}{\Gamma(z)}\prod_{\epsilon,\delta=\pm1}\Gamma\left(\frac{z+\epsilon\alpha+\delta\beta}{2}\right),
\end{align*}
valid for $\Re z>|\Re\alpha|+|\Re\beta|$, follows by inserting the integral representations of the two Bessel functions and applying the Gamma and Beta integrals. Use $z=1/2+s_2$ at a real place and $z=1+2s_2$ at a complex place. The real multiplicative integral includes both signs, and at a complex place $dy=2\,d\Re(y)\,d\Im(y)$. This proves \eqref{5.6}; multiplying the local scalars gives \eqref{eq5.8}. Meromorphic continuation proves the general identity. In particular it is the Bessel coefficient, not the Whittaker value at $I_2$, that enters this calculation.
\end{proof}

\subsubsection{Non-Archimedean Integrals: $v\nmid \mathfrak{n}$}\label{sec5.1.5}
\begin{lemma}\label{lem5.1}
Let $v\in \Sigma_{F,\fin}$ with $e_v(\mathfrak{n})=0$. Then 
\begin{equation}\label{5.4}
\mathcal{P}_v^{\sharp}(\mathbf{s},\phi_v)
=J_v(s_1,s_2)
\mathbf{1}_{\phi\ \text{is }K_v\text{-invariant}}
,
\end{equation}
where 
\begin{equation}\label{equ5.11}
J_v(s_1,s_2):=q_v^{d_{E_v}-2d_{F_v}+d_{F_v}(s_2-s_1)}.
\end{equation}
\end{lemma}
\begin{proof}
In the region $\Re(s_1)\gg 1$ and $\Re(s_2-s_1)\gg 1$, the definitions of
$W_{1,v}^{\dag}(\cdot,s_1)$ and $h_{2,v}^{\dag}(\cdot,s_2)$, together with the
Casselman--Shalika formula, give
\begin{multline}\label{e5.5}
\widetilde{\Psi}_v
\bigl(\overline{W_{\phi,v}},W_{1,v}^{\dag}(\cdot,s_1),h_{2,v}^{\dag}(\cdot,s_2)\bigr)
=
\mathbf{1}_{\phi\ \text{is }K_v\text{-invariant}}
J_v(s_1,s_2)\\
L_v(1/2+s_1+s_2,\widetilde{\pi}_v)
L_v(1/2-s_1+s_2,\widetilde{\pi}_v).
\end{multline}

By \eqref{e2.2}, we have
\begin{align*}
h_{2,v}^{\dag}(I_2,s_2)=
\int_{F_v^{\times}}
q_v^{d_{E_v}/2}\Phi_v^\circ((0,t))|t|_v^{1+2s_2}d^{\times}t=q_v^{\frac{d_{E_v}-d_{F_v}}{2}}\zeta_{F_v}(1+2s_2).
\end{align*}
Similarly, by \eqref{2.2},
\begin{align*}
W_{1,v}^{\dag}(a(y),s_1)=q_v^{d_{E_v}/2-d_{F_v}/2-d_{F_v}s_1}W_{s_1,0}(a(\delta_vy)),
\end{align*}
where $W_{s_1,0}$ is the spherical $\psi_{0,v}$-Whittaker function of $|\cdot|^{s_1}\boxplus|\cdot|^{-s_1}$ with value one at $I_2$. Substitute $x=\delta_vy$ in the Iwasawa integral. Its factor $q_v^{d_{F_v}(s_2-1/2)}$, the multiplicative-measure factor $q_v^{-d_{F_v}/2}$, and the two displayed Gaussian factors give $J_v(s_1,s_2)$. The spherical generating series is the product of the two local $L$-factors divided by $\zeta_{F_v}(1+2s_2)$, which cancels the section value. This proves \eqref{e5.5} and \eqref{5.4}; meromorphic continuation gives the general identity.
\end{proof}

\subsubsection{Non-Archimedean Integrals: $v\mid \mathfrak{n}$}\label{sec5.1.6}

\begin{lemma}\label{lemma5.3}
Let $v\mid\mathfrak n$. Transport the normalized toric functional to the diagonal torus by $g_v^\dag$, and denote the resulting functional by $\ell_v^{\mathrm{diag}}$. Then
\begin{equation}\label{eq5.10}
\ell_v(\pi_v(g_v^\dag)\phi_v)=\ell_v^{\mathrm{diag}}(\phi_v).
\end{equation}
\end{lemma}
\begin{proof}
Writing $x_v=a+b\theta_v$, the split change of basis in \eqref{eq3.3} gives
\begin{align*}
(g_v^\dag)^{-1}\iota(x_v)g_v^\dag=\diag(a+b\theta_v,a-b\theta_v).
\end{align*}
Under $F_v^\times\backslash E_v^\times\simeq F_v^\times$, use the quotient measure pushed forward by this map. Unitarity gives
\begin{align*}
\int_{F_v^\times\backslash E_v^\times}\langle\pi_v(\iota(x_v))\pi_v(g_v^\dag)\phi_v,\pi_v(g_v^\dag)\phi_v\rangle_vd^\times x_v=\int_{F_v^\times}\langle\pi_v(a(t))\phi_v,\phi_v\rangle_vd^\times t.
\end{align*}
Choosing the phase of $\ell_v^{\mathrm{diag}}$ by transport gives \eqref{eq5.10}. This identifies the translated original-torus functional with the diagonal-torus functional, not with the original-torus functional on $\phi_v$.
\end{proof}

For $1\leq j\leq 2$, define
\begin{align*}
&h_{2,v}^{(1)}(g_v,s_2):=q_v^{\frac{r_v}{2}}|\det g_v|_v^{\frac{1}{2}+s_2}
\int_{F_v^{\times}}
\Phi_v^{\circ}\left((0,t)g_v\begin{pmatrix}
\varpi_v^{-r_v}\\
& 1
\end{pmatrix}\right)|t|_v^{1+2s_2}d^{\times}t,\\
&h_{2,v}^{(2)}(g_v,s_2):=q_v^{\frac{r_v}{2}}|\det g_v|_v^{\frac{1}{2}+s_2}
\int_{F_v^{\times}}
\Phi_v^{\circ}\left((0,t)g_v\begin{pmatrix}
1\\
& \varpi_v^{-r_v}
\end{pmatrix}\right)|t|_v^{1+2s_2}d^{\times}t.
\end{align*}

By the definition \eqref{3.2}, we have the decomposition
\begin{equation}\label{e5.8}
\mathcal{P}_v^{\sharp}(\mathbf{s},\phi_v)=\sum_{j=1}^{2}\mathcal{P}_v^{(j)}(\mathbf{s},\phi_v),
\end{equation}
where 
\begin{align*}
\mathcal{P}_v^{(j)}(\mathbf{s},\phi_v):=\frac{\widetilde{\Psi}_v
\bigl(\overline{W_{\phi,v}},W_{1,v}^{\dag}(\cdot,s_1),h_{2,v}^{(j)}(\cdot,s_2)\bigr)}{L_v(1/2+s_1+s_2,\widetilde{\pi}_v)L_v(1/2-s_1+s_2,\widetilde{\pi}_v)}.
\end{align*}

Therefore,
\begin{equation}\label{5.18}
\sum_{\phi_v\in\mathfrak{B}_{\pi_v}}
\mathcal{P}_v^{\sharp}(\mathbf{s},\phi_v)\ell_v^{\mathrm{diag}}(\phi_v)=\sum_{j=1}^{2}\mathcal{J}_v^{(j)}(\mathbf{s},\pi_v),
\end{equation}
where 
\begin{align*}
\mathcal{J}_v^{(j)}(\mathbf{s},\pi_v):=\sum_{\phi_v\in\mathfrak{B}_{\pi_v}}
\mathcal{P}_v^{(j)}(\mathbf{s},\phi_v)\ell_v^{\mathrm{diag}}(\phi_v).
\end{align*}

The normalized diagonal-torus functional is unique up to phase. The Weyl element $w=\left(\begin{smallmatrix}0&-1\\1&0\end{smallmatrix}\right)$ normalizes this torus and acts on its functional by a scalar of modulus one. Changing the orthonormal basis by $\phi_v\mapsto\pi_v(w)\phi_v$ therefore gives
\begin{align*}
|\mathcal{J}_v^{(j)}(\mathbf{s},\pi_v)|=\left|\sum_{\phi_v\in\mathfrak{B}_{\pi_v}}\mathcal{P}_v^{(j)}(\mathbf{s},\pi_v(w)\phi_v)\ell_v^{\mathrm{diag}}(\phi_v)\right|.
\end{align*}

Since $W_{1,v}^{\dag}(\cdot,s_1)$ is spherical, a change of variable gives
\begin{align*}
\mathcal{P}_v^{(1)}(\mathbf{s},\pi_v(w)\phi_v)=\mathcal{P}_v^{(2)}(\mathbf{s},\phi_v).
\end{align*}

Consequently, we obtain
\begin{equation}\label{5.19}
|\mathcal{J}_v^{(1)}(\mathbf{s},\pi_v)|=|\mathcal{J}_v^{(2)}(\mathbf{s},\pi_v)|.
\end{equation}

Let $r_{\pi_v}$ be the exponent of the arithmetic conductor of $\pi_v$. For $l\geq 0$, let $\mathfrak{B}_{\pi_v}^{K_v[l]}$ be the subset consisting of right-$K_v[l]$-invariant vectors in $\mathfrak{B}_{\pi_v}$. Here $K_v[l]$ is defined by \eqref{3.5}.

\begin{lemma}\label{lem5.5}
We have 
\begin{equation}\label{eq5.20}
\mathcal{J}_v^{(1)}(\mathbf{s},\pi_v)=\sum_{\phi_v\in\mathfrak{B}_{\pi_v}^{K_v[r_v]}}
\mathcal{P}_v^{(1)}(\mathbf{s},\phi_v)\ell_v^{\mathrm{diag}}(\phi_v).
\end{equation}
In particular, $\mathcal{J}_v^{(1)}(\mathbf{s},\pi_v)$ vanishes unless $r_{\pi_v}\leq r_v$. 
\end{lemma}
\begin{proof}
By the definition, for all
$g_v\in G(F_v)$ and $k_v\in K_v[r_v]$, we have
\begin{equation}\label{e5.20}
h_{2,v}^{(1)}(g_vk_v,s_2)=h_{2,v}^{(1)}(g_v,s_2).
\end{equation}

Since $W_{1,v}^{\dag}(\cdot,s_1)$ is spherical, it follows from \eqref{e5.20} that $\mathcal{P}_v^{(1)}(\mathbf{s},\phi_v)$
vanishes unless $\pi_v$ has arithmetic conductor dividing $\mathfrak{p}_v^{r_v}$. Therefore, \eqref{eq5.20} follows. 
\end{proof}

Let $W_v^{\circ}$ denote a unit newvector in the Whittaker model of $\pi_v$, and let $\phi_v^{\circ}\in \pi_v$ be the corresponding vector. 

For $i,j\in\mathbb Z_{\geq0}$ and $t\in F_v^\times$, define the diagonal-torus matrix coefficient
\begin{align*}
C_{i,j}(t):=\langle \pi_v(a(t))\pi_v(\diag(1,\varpi_v^i))W_v^{\circ},
\pi_v(\diag(1,\varpi_v^j))W_v^{\circ}\rangle_v. 
\end{align*}

\begin{lemma}\label{lem5.6}
Let $i, j\in \mathbb{Z}_{\geq 0}$. Then
\begin{equation}\label{e5.22}
\int_{F_v^{\times}}C_{i,j}(t)d^{\times}t=\frac{|L_v(1/2,\pi_v)|^2}{\displaystyle\int_{F_v^\times}|W_v^\circ(a(y))|^2d^\times y}.
\end{equation}
\end{lemma}
\begin{proof}
The invariant inner product in the Kirillov model is the $L^2(F_v^\times)$ pairing divided by the displayed norm of the newvector. Changing variables in the two Mellin integrals removes both translations and gives
\begin{align*}
\int_{F_v^\times}C_{i,j}(t)d^\times t=\frac{\left|\int_{F_v^\times}W_v^\circ(a(y))d^\times y\right|^2}{\displaystyle\int_{F_v^\times}|W_v^\circ(a(y))|^2d^\times y}.
\end{align*}
Thus \eqref{e5.22} follows from the Casselman--Shalika formula.
\end{proof}

\begin{lemma}\label{lem5.7}
Let $0<\varepsilon<10^{-3}$. Suppose $|\Re(s_1)|+|\Re(s_2)|\leq \varepsilon$. Then 
\begin{equation}\label{e5.23}
\mathcal{J}_v^{(1)}(\mathbf{s},\pi_v)\ll r_v^6\cdot
q_v^{(-\frac12+100\varepsilon)r_v}\mathbf{1}_{r_{\pi_v}\leq r_v},
\end{equation}
where the implied constant depends only on $\varepsilon$.  
\end{lemma}
\begin{proof}
By Lemma \ref{lem5.5}, the quantity $\mathcal{J}_v^{(1)}(\mathbf{s},\pi_v)$ vanishes unless $r_{\pi_v}\leq r_v$. 

By \cite[Lemma 9]{BM15}, the subspace of right-$K_v[r_v]$-invariant vectors in the Whittaker model of $\pi_v$ has an orthonormal basis $\{W_{v,n_v}:\ 0\leq n_v \leq r_v-r_{\pi_v}\}$, where
\begin{equation}\label{e7.5}
W_{v,n_v}=\sum_{i=0}^{n_v}\xi_{\pi_v}(\mathfrak{p}_v^{i},\mathfrak{p}_v^{n_v})\pi_v\left(\begin{pmatrix}
1\\
&\varpi_v^{i}
\end{pmatrix}\right)W_v^{\circ}
\end{equation}
for some explicit coefficients $\xi_{\pi_v}(\mathfrak{p}_v^{i},\mathfrak{p}_v^{n_v})\ll 1$. 

The newvector Kirillov norm in Lemma~\ref{lem5.6} is at least one, since $W_v^\circ(I_2)=1$ and $v\mid\mathfrak n$ is good. The local adjoint and quadratic normalizing factors are uniformly bounded; in the spherical case their ratio with this norm is exactly one. Hence, by \eqref{fc5.2}, Lemma~\ref{lem5.6}, and the bounded coefficients in \eqref{e7.5}, we have
\begin{equation}\label{5.26}
\mathcal{I}_v^{\sharp}(\pi_v(g_v^{\dag})\phi_v)\ll r_v.
\end{equation}

Moreover, the proof of \cite[Lemma 7.4]{Yan26a} is an algebraic estimate in
$q_v^{\pm s_1}$ and $q_v^{\pm s_2}$.  On the region
$|\Re(s_1)|+|\Re(s_2)|\leq\varepsilon$, every monomial occurring there differs
from its value at $\mathbf s=\mathbf0$ by at most
$q_v^{O(\varepsilon r_v)}$.  Enlarging the harmless numerical coefficient of
$\varepsilon$ therefore gives, uniformly on compact subsets of this region,
\begin{equation}\label{5.25}
\mathcal{P}_v^{(1)}(\mathbf{s},\phi_v)\ll
r_v^4\cdot q_v^{(-\frac12+100\varepsilon)r_v}.
\end{equation}

There are at most $r_v+1$ basis vectors in \eqref{eq5.20}. Summing the bounds \eqref{5.25} and \eqref{5.26} therefore gives \eqref{e5.23}.
\end{proof}

\begin{cor}\label{cor5.8}
Let $0<\varepsilon<10^{-3}$ and $|\Re(s_1)|+|\Re(s_2)|\leq \varepsilon$. Then 
\begin{equation}\label{5.28}
\sum_{\phi_v\in\mathfrak{B}_{\pi_v}}
\mathcal{P}_v^{\sharp}(\mathbf{s},\phi_v)\ell_v^{\mathrm{diag}}(\phi_v)
\ll r_v^6q_v^{(-\frac12+100\varepsilon)r_v}
\mathbf{1}_{r_{\pi_v}\leq r_v}.
\end{equation}	
\end{cor}

\begin{proof}
By \eqref{5.18}, \eqref{5.19}, and the triangle inequality we obtain
\begin{equation}\label{5.29}
\sum_{\phi_v\in\mathfrak{B}_{\pi_v}}
\mathcal{P}_v^{\sharp}(\mathbf{s},\phi_v)\ell_v^{\mathrm{diag}}(\phi_v)\ll \big|\mathcal{J}_v^{(1)}(\mathbf{s},\pi_v)\big|.
\end{equation}

Therefore, \eqref{5.28} follows from \eqref{5.29} and Lemma \ref{lem5.7}.
\end{proof}

\subsubsection{Proof of Proposition \ref{prop5.1}}
Apply the Rankin--Selberg norm formula \cite[\textsection 2.2.2]{MV10} to $A_\pi$ and the squared Waldspurger bound \eqref{e5.2} to $B_\pi$ in \eqref{e5.4}. The reciprocal adjoint factors and the quadratic Hecke $L$-value factors cost $(D_EN_F(\mathfrak n))^\varepsilon$; see also \cite[Theorem~1.1]{FMP17} for the corresponding test-vector normalization. At an untwisted finite place the translated integer-ring embedding puts the unit torus in $K_v$; at a ramified place there are at most two uniformizer cosets. The quotient volume is $q_v^{-(d_{E_v}-d_{F_v})/2}$. Thus the product of the square-root toric factors contributes $D_E^{-1/4+\varepsilon}$ up to an $F$-factor. On the other hand, \eqref{equ5.11} gives $\prod_{v<\infty}J_v\ll_{F,\mathbf s}D_E$, while \eqref{f5.5} contributes $D_E^{-1/2+O(\varepsilon)}$. Their product is $D_E^{1/4+O(\varepsilon)}$. At the twisting primes Corollary~\ref{cor5.8} supplies $N_F(\mathfrak n)^{-1/2+O(\varepsilon)}$, with the polynomial factors in $r_v$ absorbed by the divisor bound. The conductor divides $\mathfrak n$, and $N_F(\mathfrak n)^{-1/2}\leq N_F(\mathfrak m)^{-1/2}$ for $\mathfrak m\mid\mathfrak n$. Finally, the Bessel integrals above, the Archimedean norm factors, and Stirling's formula leave a rapidly decreasing weight bounded by $\mathcal V_\infty(\pi)$ after absorbing polynomial spectral factors and decreasing $c$. This proves \eqref{eq5.2}, locally uniformly in the shifts.

\subsection{Estimate of \texorpdfstring{$\mathcal{P}_{\Eis}(\mathbf{s},\boldsymbol{\Phi})$}{}}\label{sec5.2}
Recall that
\begin{align*}
\mathcal{P}_{\Eis}(\mathbf{s},\boldsymbol{\Phi})
=
\sum_{\mu\in \widehat{C}_F}
\frac{1}{4\pi i\rho_F}
\int_{i\mathbb{R}}
\sum_{h\in \mathfrak{B}(\mu,\overline{\mu})}
\widetilde{\Psi}(\lambda,h;\mathbf{s})
\int_{\mathcal{C}_{E/F}}E(\iota(x),h,\lambda)d^{\times}xd\lambda.
\end{align*}

\subsubsection{Notation}\label{sec5.2.1}

For $\Re(s_1)\gg 1$ and $\Re(s_2-s_1)\gg 1$, write
$\mu_E=\mu\circ N_{E/F}$ and define
\begin{align*}
\mathbf{L}_v(\mathbf{s},\lambda,\mu_v):=\prod_{\epsilon\in \{\pm 1\}}L_v(1/2+\epsilon\lambda+s_1+s_2,\mu_v^{\epsilon})L_v(1/2+\epsilon\lambda-s_1+s_2,\mu_v^{\epsilon}).
\end{align*}
Define $\mathbf{L}(\mathbf{s},\lambda,\mu)=\prod_{v\in \Sigma_F}\mathbf{L}_v(\mathbf{s},\lambda,\mu_v)$ and 
\begin{align*}
&\mathcal{P}_v^{\sharp}(\mathbf{s},h_v,\lambda):=L_v(1-2\lambda,\mu_v^{-2})\mathbf{L}_v(\mathbf{s},\lambda,\mu_v)^{-1}\widetilde{\Psi}_v
\bigl(
\overline{W_{E(\cdot,h,-\overline{\lambda}),v}},
W_{1,v}^{\dag}(\cdot,s_1),h_{2,v}^{\dag}(\cdot,s_2)
\bigr),\\
&\mathcal I_v^{\sharp}(h_v,\lambda):=L_v(1+2\lambda,\mu_v^2)L_{E_v}(1/2+\lambda,\mu_{E,v})^{-1}\int_{F_v^{\times}\backslash E_v^{\times}}S(h_v)(\iota(x_v)g_v^{\dag},\lambda)d^{\times}x_v.
\end{align*}
We identify $\mathcal{P}_v^{\sharp}(\mathbf{s},h_v,\lambda)$, $\mathbf{L}(\mathbf{s},\lambda,\mu)$ and $\mathcal I_v^{\sharp}(h_v,\lambda)$ with their meromorphic continuations.

\subsubsection{Simplification of $\mathcal{P}_{\Eis}(\mathbf{s},\boldsymbol{\Phi})$}
Unfolding the definition \eqref{e2.5} as in \eqref{eq2.4}, we obtain 
\begin{align*}
\int_{\mathcal{C}_{E/F}}E(\iota(x)g_0,h,\lambda)d^{\times}x=\int_{\mathbb{A}_F^{\times}\backslash\mathbb{A}_E^{\times}}S(h)(\iota(x)g_0,\lambda)d^{\times}x,\ \ \Re(\lambda)\gg 1. 
\end{align*}
Consequently, by meromorphic continuation, 
\begin{equation}\label{5.61}
\int_{\mathcal{C}_{E/F}}E(\iota(x)g_0,h,\lambda)d^{\times}x=\Lambda_E(1/2+\lambda,\mu_E)\Lambda(1+2\lambda,\mu^2)^{-1}\prod_{v\in \Sigma_F}\mathcal I_v^{\sharp}(h_v,\lambda).
\end{equation}
Here $\Lambda(1+2\lambda,\mu^2)$ refers to the complete Hecke $L$-function. 

We will show that $\mathcal I_v^{\sharp}(h_v,\lambda)\equiv 1$ for all but finitely many places $v\in \Sigma_F$, so that the product on the right-hand side of \eqref{5.61} is well-defined. 

Similarly to \eqref{e5.4}, we have
\begin{multline*}
\mathcal{P}_{\Eis}(\mathbf{s},\boldsymbol{\Phi})
=|\det g_0|^{-1-s_1-s_2}
\sum_{\mu\in \widehat{C}_F}\frac{1}{4\pi i\rho_F}
\int_{i\mathbb{R}}\mathbf{L}(\mathbf{s},\lambda,\mu)\\
\Lambda_E(1/2+\lambda,\mu_E)\Lambda(1+2\lambda,\mu^2)^{-1}\\
\Lambda(1-2\lambda,\mu^{-2})^{-1}\prod_{v\in \Sigma_F}\sum_{h_v\in \mathfrak{B}_v(\mu_v,\overline{\mu}_v)}\mathcal{P}_v^{\sharp}(\mathbf{s},h_v,\lambda)\mathcal I_v^{\sharp}(h_v,\lambda)d\lambda,
\end{multline*}
where $\mathfrak{B}_v(\mu_v,\overline{\mu}_v)$ is an orthonormal basis of the local representation such that $\mathfrak{B}(\mu,\overline{\mu})=\otimes_{v\in \Sigma_F}'\mathfrak{B}_v(\mu_v,\overline{\mu}_v)$. 

When $\mu_v$ is unramified, let $h_v^{\circ}$ be the spherical vector with $h_v^{\circ}(I_2)=1$. At a ramified twisting place use the unit Casselman newvector, without imposing a value at $I_2$. By definition, $\Phi_{1,v}^{\dag}$ is everywhere unramified, and $\Phi_{2,v}^{\dag}$ is right-$K_v$-invariant for $v\nmid\mathfrak n$. Therefore, for each $v\nmid\mathfrak n$, the definitions of $\Phi_{1,v}^{\dag}$ and $\Phi_{2,v}^{\dag}$ imply that
\begin{align*}
\widetilde{\Psi}_v
\bigl(
\overline{W_{E(\cdot,h,-\overline{\lambda}),v}},
W_{1,v}^{\dag}(\cdot,s_1),h_{2,v}^{\dag}(\cdot,s_2)
\bigr)\equiv 0
\end{align*}
unless $h_v=h_v^{\circ}$, which forces $\mu_v$ to be unramified. 


As a consequence, we have
\begin{multline}\label{5.62}
\mathcal{P}_{\Eis}(\mathbf{s},\boldsymbol{\Phi})
=|\det g_0|^{-1-s_1-s_2}
\sum_{\mu\in \widehat{C}_F}\frac{1}{4\pi i\rho_F}
\int_{i\mathbb{R}}\frac{\mathbf{L}(\mathbf{s},\lambda,\mu)\Lambda_E(1/2+\lambda,\mu_E)}{\Lambda(1+2\lambda,\mu^2)\Lambda(1-2\lambda,\mu^{-2})}
\\
\prod_{v\nmid\mathfrak n}\mathcal{P}_v^{\sharp}(\mathbf{s},h_v^{\circ},\lambda)\mathcal I_v^{\sharp}(h_v^{\circ},\lambda)\mathbf{1}_{r_{\mu_v}=0}\\
\times\prod_{v\mid\mathfrak n}\sum_{h_v\in \mathfrak{B}_v(\mu_v,\overline{\mu}_v)}\mathcal{P}_v^{\sharp}(\mathbf{s},h_v,\lambda)\mathcal I_v^{\sharp}(h_v,\lambda)d\lambda,
\end{multline}
where $r_{\mu_v}$ is the arithmetic conductor exponent of $\mu_v$.

Let $v\mid\mathfrak{n}$ and $r_v=e_v(\mathfrak{n})$. Similar to \eqref{e5.8} we have
\begin{align*}
\mathcal{P}_v^{\sharp}(\mathbf{s},h_v,\lambda)=\sum_{j=1}^{2}\mathcal{P}_v^{(j)}(\mathbf{s},h_v,\lambda),
\end{align*}
where
\begin{align*}
\mathcal{P}_v^{(j)}(\mathbf{s},h_v,\lambda):={}&L_v(1-2\lambda,\mu_v^{-2})\mathbf{L}_v(\mathbf{s},\lambda,\mu_v)^{-1}\widetilde{\Psi}_v\bigl(
\overline{W_{E(\cdot,h,-\overline{\lambda}),v}},W_{1,v}^{\dag}(\cdot,s_1),\\
&h_{2,v}^{(j)}(\cdot,s_2)\bigr).
\end{align*}

By \eqref{5.62}, 
\begin{multline}\label{5.63}
\mathcal{P}_{\Eis}(\mathbf{s},\boldsymbol{\Phi})
=|\det g_0|^{-1-s_1-s_2}
\sum_{\mu\in \widehat{C}_F}\frac{1}{4\pi i\rho_F}
\int_{i\mathbb{R}}\frac{\mathbf{L}(\mathbf{s},\lambda,\mu)\Lambda_E(1/2+\lambda,\mu_E)}{\Lambda(1+2\lambda,\mu^2)\Lambda(1-2\lambda,\mu^{-2})}
\\
\prod_{v\nmid\mathfrak n}\mathcal{P}_v^{\sharp}(\mathbf{s},h_v^{\circ},\lambda)\mathcal I_v^{\sharp}(h_v^{\circ},\lambda)\mathbf{1}_{r_{\mu_v}=0}\prod_{v\mid\mathfrak n}\sum_{j=1}^{2}\mathcal{J}_v^{(j)}(\mathbf{s},h_v,\lambda)d\lambda,
\end{multline}
where
\begin{align*}
\mathcal{J}_v^{(j)}(\mathbf{s},h_v,\lambda):=\sum_{h_v\in \mathfrak{B}_v(\mu_v,\overline{\mu}_v)}\mathcal{P}_v^{(j)}(\mathbf{s},h_v,\lambda)\mathcal I_v^{\sharp}(h_v,\lambda).
\end{align*}

Transport the toric functional to the diagonal model as in Lemma~\ref{lemma5.3}. On the unitary axis its Weyl scalar has modulus one, so the two contractions have the same absolute value by the change of
orthonormal basis used in \eqref{5.19}. For $\mu_v=\mathbf1$ the scalar is one, and continuation gives the exact identity used for the residues:
\begin{equation}\label{5.32}
\mathcal{J}_v^{(1)}(\mathbf{s},h_v,\lambda)=\mathcal{J}_v^{(2)}(\mathbf{s},h_v,\lambda),\qquad \mu_v=\mathbf1.
\end{equation}

For $l\geq 0$, let $\mathfrak{B}_v(\mu_v,\overline{\mu}_v)^{K_v[l]}$ be the subset consisting of right-$K_v[l]$-invariant vectors. In parallel to Lemma \ref{lem5.5}, we obtain 
\begin{equation}\label{5.33}
\mathcal{J}_v^{(1)}(\mathbf{s},h_v,\lambda)=\sum_{h_v\in \mathfrak{B}_v(\mu_v,\overline{\mu}_v)^{K_v[r_v]}}\mathcal{P}_v^{(1)}(\mathbf{s},h_v,\lambda)\mathcal I_v^{\sharp}(h_v,\lambda).
\end{equation}
In particular, $\mathcal{J}_v^{(1)}(\mathbf{s},h_v,\lambda)$ vanishes unless $r_{\pi_v}\leq r_v$, which amounts to $r_{\mu_v}\leq r_v/2$.  
 
\subsubsection{Local Integrals at $v\nmid\mathfrak n$}
For $v\nmid\mathfrak n$, the spherical section
$S(h_v^\circ)$ is the Godement section for $\Phi_v^\circ$ divided by
its value at $I_2$; see \cite[p.~6]{JZ87}. Unfolding in the
integral-basis coordinates gives, at a finite place,
\begin{align*}
\mathcal I_v^\sharp(h_v^\circ,\lambda)
=\mu_v(\det g_{0,v})|\det g_{0,v}|_v^{1/2+\lambda}q_v^{(d_{F_v}-d_{E_v})/2}.
\end{align*}

At infinity the corresponding identity is
\[
\mathcal I_v^\sharp(h_v^\circ,\lambda)=
\mu_v(\det g_{0,v})|\det g_{0,v}|_v^{1/2+\lambda}
\begin{cases}
\pi\,2^{1/2+\lambda}\mu_v(2),&v\text{ real non-split},\\
1,&v\text{ real split},\\
\pi,&v\text{ complex}.
\end{cases}
\]
The local calculations in \textsection\ref{5.1.4}--\textsection\ref{sec5.1.5}
apply to $\mu_v|\cdot|_v^\lambda\boxplus\mu_v^{-1}|\cdot|_v^{-\lambda}$.
The scalar
$\upsilon_\mu=\prod_{v\nmid\mathfrak n}
\mu_v(\det g_{0,v})\prod_{v\in\Sigma_{F,\infty}^{\mathrm{r,ns}}}\mu_v(2)\prod_{v<\infty}\mu_v(\delta_v)^{-1}$,
satisfies $|\upsilon_\mu|=1$. Multiplying the local factors, we obtain
\begin{multline}\label{5.60}
\prod_{v\nmid\mathfrak n}
\mathcal P_v^\sharp(\mathbf s,h_v^\circ,\lambda)
\mathcal I_v^\sharp(h_v^\circ,\lambda)
=\upsilon_\mu \pi^{r_{\mathrm{ns}}+r_2}D_F^{1/2+\lambda}D_E^{-1/2}
D_{E/F}^{1/4+\lambda/2}\\
\times 2^{r_1+2r_2}J_\infty(s_1,s_2)\prod_{v\in\Sigma_{F,\fin}}J_v(s_1,s_2).
\end{multline}
Here $J_{\infty}$ and $J_v$ are defined by \eqref{eq5.8} and \eqref{equ5.11}. The extra factors are the normalized Eisenstein Whittaker coefficients: $2/L_v(1-2\lambda,\mu_v^{-2})$ at a real place, $4/L_v(1-2\lambda,\mu_v^{-2})$ at a complex place, and $q_v^{d_{F_v}\lambda}\mu_v(\delta_v)^{-1}/L_v(1-2\lambda,\mu_v^{-2})$ at a finite place. The local scattering denominator cancels in $\mathcal P_v^\sharp$. In particular \eqref{5.60} is an identity of raw basis contractions, not a product of Whittaker values at $I_2$.

\subsubsection{Local Integrals at $v\mid\mathfrak n$}
We now treat the places $v\mid\mathfrak n$. Suppose $|\Re(s_1)|+|\Re(s_2)|\leq \varepsilon$.  

By the Weyl symmetry above, it suffices to bound $j=1$ on the unitary axis. Let $h_v\in \mathfrak{B}_v(\mu_v,\overline{\mu}_v)^{K_v[r_v]}$. Following the arguments in \cite[\textsection 7.2]{Yan26a}, we have
\begin{equation}\label{5.37}
\mathcal{P}_v^{(1)}(\mathbf{s},h_v,\lambda)\ll
r_v^4\cdot q_v^{(-\frac12+100\varepsilon)r_v},
\end{equation}
where the implied constant depends only on $\varepsilon$. 

For $i\leq r_v$, we have 
\begin{align*}
\int_{F_v^{\times}\backslash E_v^{\times}}S(h_v^{\circ})(\iota(x_vg_v^{\dag})\diag(1,\varpi_v^i),\lambda)d^{\times}x_v\ll 1,
\end{align*}
where the implied constant is absolute. Consequently, by the construction of $h_v\in \mathfrak{B}_v(\mu_v,\overline{\mu}_v)^{K_v[r_v]}$ as a linear combination, with bounded coefficients, of right translations of $h_v^{\circ}$ by $\diag(1,\varpi_v^i)$ for some $0\leq i\leq r_v$, we obtain
\begin{equation}\label{5.38}
\mathcal I_v^{\sharp}(h_v,\lambda)\ll r_v+1.
\end{equation}

Combining the Weyl symmetry above with \eqref{5.33}, \eqref{5.37}, and \eqref{5.38} gives
\begin{equation}\label{5.39}
\sum_{j=1}^{2}\mathcal{J}_v^{(j)}(\mathbf{s},h_v,\lambda)\ll
r_v^6\cdot q_v^{(-\frac12+100\varepsilon)r_v}
\mathbf{1}_{r_{\mu_v}\leq r_v/2}.
\end{equation}

Since $\prod_{v<\infty}J_v(s_1,s_2)\ll_{F,\mathbf s}D_E/D_F^2$,
\eqref{f5.5} and \eqref{5.60} contribute
$D_E^{1/4+O(\varepsilon)}$ on $\Re\lambda=0$.
Substitute \eqref{5.60} and \eqref{5.39} into \eqref{5.63} to obtain
the following estimate.
\begin{prop}\label{prop5.9}
Let $|\Re(s_1)|+|\Re(s_2)|\leq \varepsilon$ and $0<\varepsilon<10^{-3}$. Then 
\begin{align*} 
\mathcal{P}_{\Eis}(\mathbf{s},\boldsymbol{\Phi})
\ll N_F(\mathfrak{n})^{-\frac{1}{2}+101\varepsilon}D_E^{1/4+100\varepsilon}
\sum_{\mu\in \widehat{C}_F(\mathfrak{n})}
\int_{i\mathbb{R}}\frac{|\mathbf{L}(\mathbf{s},\lambda,\mu)\Lambda_E(1/2+\lambda,\mu_E)|}{|\Lambda(1+2\lambda,\mu^2)|^2}d\lambda,
\end{align*}	
where $\widehat C_F(\mathfrak n)$ is defined in \textsection\ref{sec3.1.3}.
The implied constant depends on $F$ and $\varepsilon$ and is locally
uniform in the shifts.
\end{prop}

\begin{cor}\label{cor5.central}
At the central point, uniformly for the test levels occurring in the
mollifier and amplifier,
\begin{align*}
\mathcal P_{\cusp}(\mathbf0,\boldsymbol\Phi)+
\mathcal P_{\Eis}(\mathbf0,\boldsymbol\Phi)
\ll_{F,\varepsilon}D_E^{5/12+\varepsilon}N_F(\mathfrak n)^{2/3+\varepsilon}.
\end{align*}
\end{cor}
\begin{proof}
For the auxiliary levels $\mathfrak m\mid\mathfrak n$ occurring here,
which are coprime to $\mathfrak D_{E/F}$, we use the finite-level form
of Nelson's quadratic-twist Weyl bound for
$\pi\in\mathcal F(\mathfrak m,\mathbf1)$:
\begin{equation}\label{eq1.nelson.level}
|L(1/2,\pi\otimes\eta_{E/F})|
\ll_{F,\varepsilon,\pi_\infty}D_E^{1/3+\varepsilon}
N_F(\mathfrak m)^{1/3+\varepsilon}.
\end{equation}
The dependence on $\pi_\infty$ is polynomial in its analytic conductor.
We record the finite-level calculation, since only polynomial dependence
is stated in \cite[Theorem~1.1]{Nel19}.

At $v\mid\mathfrak m$, with $r=e_v(\mathfrak m)\geq1$, use the
pre-Kuznetsov kernel $\mathbf1_{K_0(\mathfrak p_v^r)}$ in
\cite[\textsection 13]{Nel19}. Lemma~13.1 of that paper gives a
nonnegative weight, which is $\gg q_v^{-r}$ on generic representations
of conductor at most $\mathfrak p_v^r$, using any fixed bound towards
Ramanujan. This includes principal series, twists of Steinberg, and
supercuspidals.
Here $\eta_v$ is unramified, so twisting preserves the conductor.
For the dual weight, write $q=q_v$ and $a(\omega)$ for the conductor
exponent of a local unitary character $\omega$. The transform
\cite[(2.30)]{Nel19} satisfies
\begin{align*}
|\widetilde h_v(\omega)|\ll
\begin{cases}
1,&a(\omega)=0,\\
q^{-2},&r=1,\ a(\omega)=1,\\
0,&\text{otherwise}.
\end{cases}
\end{align*}
To check this directly, the two Fourier transforms in
\cite[(2.15) and (2.21)]{Nel19} reduce to
\begin{align*}
V^\sharp(x,y)=\int_{|\xi|\geq q^r|x|}
\big(\mathbf1_{|\xi|\leq1}-q^{-1}\mathbf1_{|\xi|\leq q}\big)
\psi_v(-\xi y)\,d\xi.
\end{align*}
For $r\geq2$, its contribution to (2.31) of that paper is supported on
$e_v(t)\geq r-1$ and equals
$(e_v(t)-r+1)(1-q^{-1})-q^{-1}$. Its Mellin transform vanishes on
ramified characters and is bounded on unramified ones. For $r=1$, the
unit ranges add only characters of conductor exponent one, with weight
$O(q^{-2})$. The same geometric series in the normalized double Hecke
integrals \cite[(15.9)]{Nel19} are bounded by
$(r+1)^{O(1)}q^{O(\alpha r)}$ on the contours of
\cite[Proposition~15.5]{Nel19}.

At $v\mid\mathfrak D_{E/F}$ outside $2\mathfrak D_F$, use the
ramified-character weights of \cite[Theorem~3.1 and Proposition~3.3]{Nel19}.
The cubic weight is $\gg q_v^{-1}$; the dual weight is $O(1)$ on
unramified characters and $O(q_v^{-1})$ on characters of conductor
exponent one, and vanishes at higher conductors. Thus, at every finite
place, the dual weight is at most a fixed multiple of $C(\omega_v)^{-1}$.
Applying \cite[(4.6)]{Nel19} separately at each exact dual conductor
cancels these factors. The divisor bound absorbs the number of conductor
strata and the factors $(r+1)^{O(1)}$. The degenerate terms satisfy the
same bound by \cite[\textsection 14--\textsection 15]{Nel19}, taking
$\alpha$ sufficiently small in terms of $\varepsilon$.
The finitely many local types above $2\mathfrak D_F$ affect only the
$F$-constant. Hence the cubic moment is
$O_{F,\varepsilon}((D_EN_F(\mathfrak m))^\varepsilon)$, with polynomial
Archimedean dependence, while its weight at $\pi\otimes\eta_{E/F}$ is
$\gg_{F,\varepsilon}(D_EN_F(\mathfrak m))^{-1-\varepsilon}$.
Positivity proves \eqref{eq1.nelson.level}.

Apply Proposition~\ref{prop5.1} and factor
$L(1/2,\pi_E)=L(1/2,\pi)L(1/2,\pi\otimes\eta_{E/F})$.
Use \eqref{eq1.nelson.level} only for the twisted factor. For
$M=N_F(\mathfrak m)$, its square root contributes
$D_E^{1/6+\varepsilon}M^{1/6+\varepsilon}$, with polynomial Archimedean
factors absorbed by $\mathcal V_\infty$. H\"older's inequality, the fourth
moment \cite[Theorem~E]{Yan26a}, and the Weyl count bound the remaining
$5/2$-th moment by $M^{1+\varepsilon}$; see \eqref{eq7.level.secondmoment}. The local factor $M^{-1/2+\varepsilon}$ in
\eqref{eq5.2} and the sum over $\mathfrak m\mid\mathfrak n$ give
$N_F(\mathfrak n)^{2/3+\varepsilon}$.
For the continuous spectrum, use Proposition~\ref{prop5.9} and the
Eisenstein fourth moment \cite[Theorem~F]{Yan26a}, with the
scattering denominator included in the spectral measure. The same interpolation
and Stirling decay apply. Multiplication by the common factor
$D_E^{1/4+\varepsilon}$ proves the asserted bound.
\end{proof}

\section{Computation of \texorpdfstring{$\widetilde{\Psi}_{\Res}^{(i)}(\mathbf{s},\boldsymbol{\Phi})$}{}}\label{sec6}
In this section, we compute the residual terms $\widetilde{\Psi}_{\Res}^{(i)}(\mathbf{s},\boldsymbol{\Phi})$, $1\leq i\leq 4$, explicitly. 
\subsection{Singular Spectral Component}
For $s_1, s_2, \lambda\in \mathbb{C}$, we define 
\begin{equation}\label{eq6.1}
\mathcal{I}_{\mathrm{sin}}(\mathbf{s},\boldsymbol{\Phi},\lambda):=(2\rho_F)^{-1}
\sum_{h\in \mathfrak{B}(\mathbf{1},\mathbf{1})}
\widetilde{\Psi}(\lambda,h;\mathbf{s})
\int_{\mathcal{C}_{E/F}}E(\iota(x),h,\lambda)d^{\times}x.
\end{equation}
Here $\widetilde{\Psi}(\lambda,h;\mathbf{s})$ is defined in \textsection\ref{sec2.4.2}. The contribution of $\mu=\mathbf1$ to $\mathcal P_{\Eis}(\mathbf s,\boldsymbol\Phi)$ is
\[
\frac1{2\pi i}\int_{i\mathbb R}\mathcal I_{\mathrm{sin}}(\mathbf s,\boldsymbol\Phi,\lambda)\,d\lambda.
\]

With the convention of Definition~\ref{defn2.1}, the terms with $i=1,3$ are residues of $\mathcal{I}_{\mathrm{sin}}$, while those with $i=2,4$ are their negatives. We have the following explicit expression for $\mathcal{I}_{\mathrm{sin}}(\mathbf{s},\boldsymbol{\Phi},\lambda)$. 

\begin{prop}\label{prop6.1}
Let $s_1, s_2, \lambda\in \mathbb{C}$. We have 
\begin{multline}\label{f6.1}
\mathcal{I}_{\mathrm{sin}}(\mathbf{s},\boldsymbol{\Phi},\lambda)=\frac{2^{\omega_F(\mathfrak n)-1}\pi^{-r_{\mathrm{ns}}}}{\rho_F}D_F^{-1+2s_2}D_E^{1/4-(s_1+s_2)/2+\lambda/2}N_F(\mathfrak n)^{-1/2-s_2}\\
\frac{\mathbf{L}(\mathbf{s},\lambda,\mathbf{1})\Lambda_E(1/2+\lambda)}{\Lambda_F(1+2\lambda)\Lambda_F(1-2\lambda)}R_{\mathfrak{n}}(s_1,s_2;\lambda),
\end{multline}
where $R_{\mathfrak{n}}(s_1,s_2;\lambda)$ is defined in
\eqref{eq3.Rr}.  	
\end{prop}
\begin{proof}
By \eqref{5.32}, \eqref{5.60}, the density
$d\lambda/(4\pi i\rho_F)$, and the translated toric period, we obtain  
\begin{multline}\label{e6.2}
\mathcal{I}_{\mathrm{sin}}(\mathbf{s},\boldsymbol{\Phi},\lambda)=\frac{2^{\omega_F(\mathfrak n)-1}\pi^{-r_{\mathrm{ns}}}}{\rho_F}D_F^{-1+2s_2}D_E^{1/4-(s_1+s_2)/2+\lambda/2}\\
\times\frac{\mathbf{L}(\mathbf{s},\lambda,\mathbf{1})\Lambda_E(1/2+\lambda)}{\Lambda_F(1+2\lambda)\Lambda_F(1-2\lambda)}
\prod_{v\mid\mathfrak n}\frac{\zeta_{F_v}(1+2\lambda)\zeta_{F_v}(1-2\lambda)D_v(\mathbf s;\lambda)}{\mathbf L_v(\mathbf s,\lambda,\mathbf1)\zeta_{E_v}(1/2+\lambda)},
\end{multline}
where $D_v(\mathbf{s};\lambda)$ is given by
\begin{align*}
\sum_{h_v\in \mathfrak{B}_v(\mathbf{1},\mathbf{1})}\widetilde{\Psi}_v
\bigl(
\overline{W_{E(\cdot,h,-\overline{\lambda}),v}},
W_{1,v}^{\dag}(\cdot,s_1),h_{2,v}^{(1)}(\cdot,s_2)
\bigr)\int_{F_v^{\times}\backslash E_v^{\times}}S(h_v)(\iota(x_v)g_v^{\dag},\lambda)d^{\times}x_v.
\end{align*}

At $v\mid\mathfrak n$, the normalized spherical section has toric period
$\zeta_{E_v}(1/2+\lambda)/\zeta_{F_v}(1+2\lambda)$ and Rankin--Selberg coefficient
$\mathbf L_v(\mathbf s,\lambda,\mathbf1)/\zeta_{F_v}(1-2\lambda)$.
The local factor in \eqref{e6.2} replaces this spherical contraction by $D_v$.

By \cite[Lemma 6.3]{Yan26a}, we have
\begin{multline}\label{c6.6}
\widetilde{\Psi}_v
\bigl(
\overline{W_{E(\cdot,h,-\overline{\lambda}),v}},
W_{1,v}^{\dag}(\cdot,s_1),h_{2,v}^{(1)}(\cdot,s_2)
\bigr)=\gamma(1/2-\lambda-s_2+s_1,\mathbf{1},\psi_v)\\
\gamma(1/2-\lambda-s_2-s_1,\mathbf{1},\psi_v)
\widetilde{\Psi}_v(W_{1,v}^{\dag}(\cdot,s_1),W_{2,v}^{(1)}(\cdot,s_2),\overline{S(h_v)(\cdot,-\overline{\lambda})}),
\end{multline}
where $W_{2,v}^{(1)}(\cdot,s_2)$ denotes the Jacquet integral of $h_{2,v}^{(1)}(\cdot,s_2)$ relative to $\overline{\psi}_v$. Here $\gamma(s,\mathbf{1},\psi_v)$ is the gamma factor, explicitly given by
\begin{equation}\label{eq6.2}
\gamma(s,\mathbf{1},\psi_v)=\zeta_{F_v}(s)^{-1}\zeta_{F_v}(1-s).
\end{equation}

Assume that
\begin{equation}\label{eq6.5}
|\Re(s_1)|+|\Re(s_2)|<\varepsilon<10^{-3},\qquad |\Re(\lambda)|<1/2-10\varepsilon.
\end{equation}
In this range, $\widetilde{\Psi}_v(W_{1,v}^{\dag}(\cdot,s_1),W_{2,v}^{(1)}(\cdot,s_2),\overline{S(h_v)(\cdot,-\overline{\lambda})})$ converges absolutely and unfolds to
\begin{equation}\label{e6.3}
\widetilde{\Psi}_v(W_{1,v}^{\dag}(\cdot,s_1),W_{2,v}^{(1)}(\cdot,s_2),\overline{S(h_v)(\cdot,-\overline{\lambda})})=\int_{K_v}\mathfrak{F}_v(k_v,\lambda)\overline{S(h_v)(k_v,-\overline{\lambda})}dk_v,
\end{equation}
where
\begin{align*}
\mathfrak{F}_v(g_v,\lambda):=\int_{F_v^{\times}}W_{1,v}^{\dag}\left(\begin{pmatrix}
y_v\\
& 1
\end{pmatrix}g_v,s_1\right)
W_{2,v}^{(1)}\left(\begin{pmatrix}
y_v\\
& 1
\end{pmatrix}g_v,s_2\right)|y_v|_v^{-\frac{1}{2}-\lambda}d^{\times}y_v. 
\end{align*}

Combining \eqref{c6.6}, \eqref{eq6.2}, and \eqref{e6.3}, we obtain
\begin{equation}\label{e6.4}
D_v(\mathbf{s};\lambda)=\frac{\zeta_{F_v}(1/2+\lambda+s_2-s_1)\zeta_{F_v}(1/2+\lambda+s_2+s_1)}{\zeta_{F_v}(1/2-\lambda-s_2+s_1)\zeta_{F_v}(1/2-\lambda-s_2-s_1)}\cdot \mathfrak{G}_v(\mathbf{s},\lambda),
\end{equation}
where  
\begin{align*}
\mathfrak{G}_v(\mathbf{s},\lambda):=\sum_{h_v\in \mathfrak{B}_v(\mathbf{1},\mathbf{1})}\int_{K_v}\mathfrak{F}_v(k_v,\lambda)\overline{S(h_v)(k_v,-\overline{\lambda})}dk_v\int_{F_v^{\times}\backslash E_v^{\times}}S(h_v)(\iota(x_v)g_v^{\dag},\lambda)d^{\times}x_v. 
\end{align*}

Since $\mathfrak{F}_v(k_v,\lambda)\in |\cdot|_v^{\lambda}\boxplus |\cdot|_v^{-\lambda}$, the orthogonal decomposition is
\begin{equation}\label{e6.5}
\mathfrak{G}_v(\mathbf{s},\lambda)=\int_{F_v^{\times}\backslash E_v^{\times}}\mathfrak{F}_v(\iota(x_v)g_v^{\dag},\lambda)d^{\times}x_v. 
\end{equation}

Let $g_v\in G(F_v)$. By the definition \eqref{2.2}, we have
\begin{equation}\label{eq6.6}
W_{2,v}^{(1)}(g_v,s_2)=q_v^{-r_vs_2}W_{1,v}^{\dag}\left(g_v\begin{pmatrix}
1\\
& \varpi_v^{r_v}
\end{pmatrix},s_2\right). 
\end{equation}

Substituting \eqref{eq6.6} into the definition of $\mathfrak{F}_v(g_v,\lambda)$ yields
\begin{multline*}
\mathfrak{F}_v(g_v,\lambda)=q_v^{-r_vs_2}\int_{F_v^{\times}}W_{1,v}^{\dag}\left(\begin{pmatrix}
y_v\\
& 1
\end{pmatrix}g_v,s_1\right)\\
W_{1,v}^{\dag}\left(\begin{pmatrix}
y_v\\
& 1
\end{pmatrix}g_v\begin{pmatrix}
1\\
& \varpi_v^{r_v}
\end{pmatrix},s_2\right)|y_v|_v^{-\frac{1}{2}-\lambda}d^{\times}y_v. 
\end{multline*}

By Lemma \ref{lem6.2} below, there are coefficients $\alpha_i$ satisfying 
\begin{equation}\label{e6.10}
\mathfrak{F}_v(g_v,\lambda)=q_v^{-r_vs_2}\sum_{i=0}^{r_v}\alpha_ih_{1,v}^{\dag}(g_v\diag(1,\varpi_v^i),\lambda) 
\end{equation}
for all $g_v\in G(F_v)$, such that 
\begin{equation}\label{eq6.11}
\sum_{i=0}^{r_v}\alpha_i=q_v^{-r_v/2}\frac{\mathbf L_v^-(\mathbf s,\lambda)}{\zeta_{F_v}(1+2\lambda)\zeta_{F_v}(1-2\lambda)}\mathscr R_{v,r_v}(s_1,s_2;\lambda),
\end{equation}
where $\mathscr R_{v,r_v}(s_1,s_2;\lambda)$ is defined by
\eqref{eq3.Rr}.

On the other hand, by \eqref{e2.2}, we have 
\begin{multline*}
\int_{F_v^{\times}\backslash E_v^{\times}}h_{1,v}^{\dag}(\iota(x_v)g_v^{\dag}\diag(1,\varpi_v^i),\lambda)d^{\times}x_v=q_v^{-(1/2+\lambda)i}\int_{E_v^{\times}}|N_{E_v/F_v}(x_v)|_v^{\frac{1}{2}+\lambda}\\
\Phi_v^{\circ}((0,1)\iota(x_v)g_v^{\dag}\diag(1,\varpi_v^i))d^{\times}x_v
\end{multline*}

By the definition of $g_v^{\dag}$, as in \eqref{3.14}, and a change of variable, we therefore obtain 
\begin{equation}\label{6.11}
\int_{F_v^{\times}\backslash E_v^{\times}}h_{1,v}^{\dag}(\iota(x_v)g_v^{\dag}\diag(1,\varpi_v^i),\lambda)d^{\times}x_v=\zeta_{E_v}(1/2+\lambda). 
\end{equation}

Using \eqref{e6.5}, \eqref{e6.10}, \eqref{eq6.11}, and \eqref{6.11}, we get
\begin{equation}\label{6.12}
\mathfrak{G}_v(\mathbf{s},\lambda)=q_v^{-r_v(1/2+s_2)}\zeta_{E_v}(1/2+\lambda)\frac{\mathbf L_v^-(\mathbf s,\lambda)}{\zeta_{F_v}(1+2\lambda)\zeta_{F_v}(1-2\lambda)}\mathscr R_{v,r_v}(s_1,s_2;\lambda).
\end{equation}

Recall that 
\begin{align*}
\mathbf{L}_v(\mathbf{s},\lambda,\mathbf{1})=\prod_{\epsilon\in \{\pm 1\}}\zeta_{F_v}(1/2+\epsilon\lambda+s_1+s_2)\zeta_{F_v}(1/2+\epsilon\lambda-s_1+s_2).
\end{align*}

Together with \eqref{e6.4}, identity \eqref{6.12} gives, for $v\mid\mathfrak n$, 
\begin{equation}\label{6.14}
\frac{D_v(\mathbf{s};\lambda)}{\mathbf{L}_v(\mathbf{s},\lambda,\mathbf{1})\zeta_{E_v}(1/2+\lambda)}
=\frac{\mathscr R_{v,r_v}(s_1,s_2;\lambda)}{q_v^{r_v(1/2+s_2)}\zeta_{F_v}(1+2\lambda)\zeta_{F_v}(1-2\lambda)},
\end{equation}
where $\mathbf{L}_v^-(\mathbf{s},\lambda)$ is defined by \eqref{equ3.8}.  

Multiplying \eqref{6.14} over $v\mid\mathfrak n$ supplies
$N_F(\mathfrak n)^{-1/2-s_2}R_{\mathfrak n}(s_1,s_2;\lambda)$
in \eqref{e6.2}, proving \eqref{f6.1} in the region \eqref{eq6.5}.
The general case follows by meromorphic continuation.
\end{proof}

\begin{lemma}\label{lem6.2}
Let $v\mid\mathfrak n$, $r\geq 1$, and $g_v\in G(F_v)$. Define $T(g_v,\lambda)$ by
\begin{align*}
\int_{F_v^{\times}}W_{1,v}^{\dag}\left(\begin{pmatrix}
y_v\\
& 1
\end{pmatrix}g_v,s_1\right)
W_{1,v}^{\dag}\left(\begin{pmatrix}
y_v\\
& 1
\end{pmatrix}g_v\begin{pmatrix}
1\\
&\varpi_v^{r}
\end{pmatrix},s_2\right)|y_v|_v^{-\frac{1}{2}-\lambda}d^{\times}y_v. 
\end{align*}
Then 
\begin{align*}
T(g_v,\lambda)=\sum_{i=0}^{r}\alpha_ih_{1,v}^{\dag}(g_v\diag(1,\varpi_v^i),\lambda),
\end{align*}
with coefficients $\alpha_i$ satisfying  
\begin{equation}\label{eq6.15}
\sum_{i=0}^{r}\alpha_i
=q_v^{-r/2}\frac{\mathbf L_v^-(\mathbf s,\lambda)}{\zeta_{F_v}(1+2\lambda)\zeta_{F_v}(1-2\lambda)}\mathscr R_{v,r}(s_1,s_2;\lambda),
\end{equation}
where $\mathscr R_{v,r}(s_1,s_2;\lambda)$ is defined by \eqref{eq3.Rr}.
\end{lemma}
\begin{proof}
The map $g_v\mapsto T(g_v,\lambda)$ defines a right-$K_v[r]$-invariant section in $|\cdot|_v^{\lambda}\boxplus |\cdot|_v^{-\lambda}$. Its restriction to $K_v$ is determined by its values at the $r+1$ representatives used below. The spherical translates form a basis for generic $\lambda$, since their evaluation matrix is invertible. Thus
\begin{equation}\label{6.15}
T(g_v,\lambda)=\sum_{i=0}^{r}\alpha_i h_{1,v}^{\dag}(g_v\diag(1,\varpi_v^i),\lambda)
\end{equation} 
for some $\alpha_i\in \mathbb{C}$. 

Let $0\leq l\leq r$ and $g_v=\begin{pmatrix}
1\\
\varpi_v^l & 1
\end{pmatrix}$. Put $T_l:=T(g_v,\lambda)$.
Then 
\begin{equation}\label{e6.17}
g_v\begin{pmatrix}
1\\
& \varpi_v^i
\end{pmatrix}=
\begin{cases}
\begin{pmatrix}
1\\
& \varpi_v^i
\end{pmatrix}\begin{pmatrix}
1\\
\varpi_v^{l-i} & 1
\end{pmatrix},\ \quad & \text{if $l\geq i$},\\
\begin{pmatrix}
\varpi_v^{i-l}&1\\
& \varpi_v^{l}
\end{pmatrix}\begin{pmatrix}
& -1\\
1& \varpi_v^{i-l} 
\end{pmatrix},\ \quad & \text{if $l< i$}.
\end{cases}
\end{equation}

By the definition \eqref{e2.2}, we have 
\begin{equation}\label{6.16}
h_{1,v}^{\dag}(g_v\diag(1,\varpi_v^i),\lambda)=\zeta_{F_v}(1+2\lambda)q_v^{(1/2+\lambda)(l-|l-i|)}.
\end{equation}

Write $y:=q_v^{-(1/2+\lambda)}$ and $U_l:=y^{l}T_l/\zeta_{F_v}(1+2\lambda)$. With \eqref{6.16}, identity \eqref{6.15} becomes
\begin{align*}
\sum_{i=0}^{r}\alpha_i y^{|l-i|}=U_l,\ \ \ 0\leq l\leq r. 
\end{align*}

Let $A=(y^{|i-l|})_{0\leq i,l\leq r}$ be the $(r+1)\times (r+1)$ matrix. Suppose $y^2\neq 1$. Then $A$ is invertible and
\begin{align*}
A^{-1}=\frac{1}{1-y^2}
\begin{pmatrix}
1 & -y & 0 & \cdots & 0\\
-y & 1+y^2 & -y & \cdots & 0\\
0 & -y & 1+y^2 & \cdots & 0\\
\vdots & \vdots & \vdots & \ddots & -y\\
0 & 0 & 0 & -y & 1
\end{pmatrix}. 
\end{align*}
The product of this tridiagonal matrix with $A$ has entries zero off
the diagonal and $1-y^2$ on the diagonal.

Therefore,
\begin{align*}
\begin{cases}
\alpha_0
=
\frac{U_0-yU_1}{1-y^2},\\
\alpha_i
=
\frac{-yU_{i-1}+(1+y^2)U_i-yU_{i+1}}{1-y^2},
\quad 1\leq i\leq r-1,\\
\alpha_r
=
\frac{U_r-yU_{r-1}}{1-y^2}.
\end{cases}
\end{align*}

It follows that
\begin{equation}\label{6.17}
\sum_{i=0}^{r}\alpha_i
=
\frac{1}{1+y}
\left(U_0+U_r\right)
+
\frac{1-y}{1+y}\sum_{l=1}^{r-1}U_l.
\end{equation}

Substituting $U_l=q_v^{-(1/2+\lambda)l}T_l/\zeta_{F_v}(1+2\lambda)$ into \eqref{6.17}, we obtain
\begin{equation}\label{e6.20}
\zeta_{F_v}(1+2\lambda)\sum_{i=0}^{r}\alpha_i
=
\frac{q_v^{1/2+\lambda}}{q_v^{1/2+\lambda}+1}
\left(T_0+q_v^{-r(1/2+\lambda)}T_r\right)+
\frac{q_v^{1/2+\lambda}-1}{q_v^{1/2+\lambda}+1}
\sum_{l=1}^{r-1}q_v^{-l(1/2+\lambda)}T_l.
\end{equation}

We now evaluate the terms appearing in \eqref{e6.20}. 
\begin{itemize}
\item Suppose $l=r$. Then 
\begin{align*}
T_r=\int_{F_v^{\times}}W_{1,v}^{\dag}\left(\begin{pmatrix}
y_v\\
& 1
\end{pmatrix},s_1\right)
W_{1,v}^{\dag}\left(\begin{pmatrix}
y_v\\
& \varpi_v^{r}
\end{pmatrix},s_2\right)|y_v|_v^{-\frac{1}{2}-\lambda}d^{\times}y_v. 
\end{align*}
By Macdonald's formula, we have
\begin{multline}\label{6.20}
\mathbf{L}_v^-(\mathbf{s},\lambda)^{-1}T_r
=
q_v^{r\lambda}
\Bigg[
\frac{q_v^{(r+1)s_1}-q_v^{-(r+1)s_1}}{q_v^{s_1}-q_v^{-s_1}}
-
q_v^{\lambda-\frac12}
\left(q_v^{s_2}+q_v^{-s_2}\right)
\frac{q_v^{rs_1}-q_v^{-rs_1}}{q_v^{s_1}-q_v^{-s_1}}\\
+
q_v^{2\lambda-1}
\frac{q_v^{(r-1)s_1}-q_v^{-(r-1)s_1}}{q_v^{s_1}-q_v^{-s_1}}\Bigg],
\end{multline}
with the convention that the last fraction is $0$ when $r=1$. 

\item Suppose $0\leq l<r$. By \eqref{e6.17}, we have  
\begin{align*}
T_l=\int_{F_v^{\times}}W_{1,v}^{\dag}\left(\begin{pmatrix}
y_v\\
& 1
\end{pmatrix},s_1\right)
W_{1,v}^{\dag}\left(\begin{pmatrix}
y_v\\
& 1
\end{pmatrix}\begin{pmatrix}
\varpi_v^{r-l}&1\\
& \varpi_v^{l}
\end{pmatrix},s_2\right)|y_v|_v^{-\frac{1}{2}-\lambda}d^{\times}y_v. 
\end{align*}
 
Therefore,
\begin{align*}
T_l=\int_{F_v^{\times}}W_{1,v}^{\dag}\left(\begin{pmatrix}
y_v\\
& 1
\end{pmatrix},s_1\right)
W_{1,v}^{\dag}\left(\begin{pmatrix}
\varpi_v^{r-2l} y_v\\
& 1
\end{pmatrix},s_2\right)\overline{\psi}_v(\varpi_v^{-l}y_v)|y_v|_v^{-\frac{1}{2}-\lambda}d^{\times}y_v. 
\end{align*}

By Macdonald's formula and Ramanujan sums, we have
\begin{multline}\label{6.21}
T_l=
q_v^{l\lambda+\frac{l-r}{2}}(q_v^{s_1}-q_v^{-s_1})^{-1}(q_v^{s_2}-q_v^{-s_2})^{-1}
\bigg[
\frac{q_v^{(l+1)s_1+(r-l+1)s_2}
}{1-q_v^{-1/2+\lambda+s_1+s_2}}\\
-
\frac{q_v^{(l+1)s_1-(r-l+1)s_2}
}{1-q_v^{-1/2+\lambda+s_1-s_2}}-\frac{q_v^{-(l+1)s_1+(r-l+1)s_2}}{1-q_v^{-1/2+\lambda-s_1+s_2}}+\frac{q_v^{-(l+1)s_1-(r-l+1)s_2}}{
1-q_v^{-1/2+\lambda-s_1-s_2}}
\bigg]
\\
-
\mathbf{1}_{l\geq 1}
\frac{
q_v^{(l-1)\lambda+\frac{l+1-r}{2}}
}{
q_v-1
}
\frac{q_v^{ls_1}-q_v^{-ls_1}}{q_v^{s_1}-q_v^{-s_1}}
\frac{q_v^{(r-l)s_2}-q_v^{-(r-l)s_2}}{q_v^{s_2}-q_v^{-s_2}}.
\end{multline}
\end{itemize}

Therefore, \eqref{eq6.15} follows from \eqref{e6.20}, \eqref{6.20}, and \eqref{6.21}. 
\end{proof}

\subsection{Exact Formula of \texorpdfstring{$\widetilde{\Psi}_{\Res}^{(i)}(\mathbf{s},\boldsymbol{\Phi})$}{}}
\begin{prop}\label{prop6.residues}
Let $|\Re(s_1)|+|\Re(s_2)|<1/2$. Then
\begin{equation}\label{eq6.res12}
\widetilde\Psi_{\mathrm{Res}}^{(1)}(\mathbf s,\boldsymbol\Phi)=\widetilde\Psi_{\mathrm{Res}}^{(2)}(\mathbf s,\boldsymbol\Phi)=-2^{-1}D_E^{-s_1-s_2}N_F(\mathfrak n)^{-2s_2}\mathscr Z(-s_1,-s_2),
\end{equation}
and
\begin{equation}\label{eq6.res34}
\widetilde\Psi_{\mathrm{Res}}^{(3)}(\mathbf s,\boldsymbol\Phi)=\widetilde\Psi_{\mathrm{Res}}^{(4)}(\mathbf s,\boldsymbol\Phi)=-2^{-1}D_E^{-s_2}N_F(\mathfrak n)^{-2s_2}\mathscr Z(s_1,-s_2).
\end{equation}
\end{prop}
\begin{proof}
By Definition \ref{defn2.1}, the four terms are the signed residues of the singular contraction in Proposition \ref{prop6.1} at
\[
\lambda=s_2+s_1-1/2,\quad 1/2-s_2-s_1,\quad s_2-s_1-1/2,\quad 1/2+s_1-s_2,
\]
respectively. By \eqref{eq3.Ridentities} and the functional equation of $\Lambda_E(1/2+\lambda)$, the contraction in \eqref{f6.1} is even in $\lambda$, so the residues at opposite poles differ by a sign. At the two poles near $1/2$, the residue of the completed zeta factor in $\mathbf L(\mathbf s,\lambda,\mathbf1)$ is $\rho_F\pi^{-r_2}$, since $\prod_{v\mid\infty}\zeta_{F_v}(1)=\pi^{-r_2}$. This cancels the factor $\rho_F^{-1}$ already incorporated in \eqref{eq6.1} and \eqref{f6.1}; the remaining $\pi^{-r_2}$ combines with $\pi^{-r_{\mathrm{ns}}}$ to give $\pi^{-r_{\mathrm{ns}}-r_2}$. The factors at $v\mid\mathfrak n$ are specialized only after their common Euler factors have been cancelled meromorphically. The two specialization identities in \eqref{eq3.Ridentities}, together with the functional equation $\Lambda_F(2s_2)=D_F^{1/2-2s_2}\Lambda_F(1-2s_2)$, identify the resulting expressions with \eqref{eq4.Z}. The powers of $D_E$ and $N_F(\mathfrak n)$ are exactly those displayed above. Finally, the contour signs in Definition \ref{defn2.1} and the factor $1/2$ in the normalized contraction give the common coefficient $-1/2$. This proves \eqref{eq6.res12} and \eqref{eq6.res34}.
\end{proof}

The minus sign before the residual sum in \eqref{e1.2} turns each pair of
equal residues in \eqref{eq6.res12}--\eqref{eq6.res34} into one full
contribution: the pair $i=1,2$ supplies the $(-s_1,-s_2)$ branch, and the
pair $i=3,4$ supplies the $(s_1,-s_2)$ branch. Together with the geometric
branches $(s_1,s_2)$ and $(-s_1,s_2)$, these give exactly the four sign
choices in \eqref{eq7.main}.

\begin{lemma}\label{lem6.residual.audit}
In the computation of Proposition \ref{prop6.residues}, every product containing
$R_{\mathfrak n}(s_1,s_2;\lambda)$ is interpreted as a meromorphic product
before $\lambda$ is replaced by the indicated affine function of
$(s_1,s_2)$. With this convention no additional factor
$\mathbf L_v^-(\mathbf s,\lambda)$ occurs. For $0<\varepsilon<10^{-3}$ and
$|\Re(s_1)|+|\Re(s_2)|\leq\varepsilon$, the two positive
specializations satisfy, after meromorphic continuation,
\begin{equation}\label{eq6.residual.audit.bound}
 \max_{\lambda\in\{1/2-s_1-s_2,\,1/2+s_1-s_2\}}
 \left|R_{\mathfrak n}(s_1,s_2;\lambda)\right|
 \ll_{F,\varepsilon}N_F(\mathfrak n)^{C_F\varepsilon},
\end{equation}
uniformly in this small-shift region.
The same estimate holds for every fixed-order derivative of either specialized product, after decreasing the size of the compact set and replacing $C_F$ by a constant depending also on the order of differentiation.
The local factors are also regular at the two negative specializations
$\lambda=s_2\pm s_1-1/2$. The poles in the global main terms come from the
completed zeta factors in $\mathscr Z$. Thus the central main term in
Corollary~\ref{cor7.2} is the joint finite part of the four-term sum in
Theorem~\ref{thm7.1}.
\end{lemma}

\begin{proof}
For $v\mid\mathfrak n$, write $q=q_v$ and $r=e_v(\mathfrak n)$. For $r=1$,
the denominator in \eqref{eq3.9} is nonzero near either central
specialization $\lambda=\pm1/2$. In the unsimplified computation, the
apparent pole at $\lambda=1/2$ cancels by
\begin{align*}
 \lim_{\lambda\to1/2}
 \zeta_{F_v}(1-2\lambda)(1-q^{\lambda-1/2})=1/2.
\end{align*}
For $\lambda=1/2-u$, the cancellation is even visible before taking the
limit:
\[
 \zeta_{F_v}(2u)(1-q^{-u})=(1+q^{-u})^{-1}.
\]
For general $r$, the identities \eqref{eq3.Ridentities} express both
positive specializations in terms of $q^{\pm rs_1}H_v(\pm s_1,-s_2)$.
Replace the geometric quotients in $H_v$ by finite sums. Their derivatives
of any fixed order are $O_{A,\varepsilon}((1+r)^Aq^{Cr\varepsilon})$ on a
sufficiently small neighborhood of the origin. Multiplication over
$v\mid\mathfrak n$ and the divisor bound prove
\eqref{eq6.residual.audit.bound} and its differentiated form. The symmetry
in $\lambda$ gives the same regularity at the negative specializations.
The global zeta poles are handled by summing the four branches before
taking the finite part.
\end{proof}

\section{Twisted Second Moment of Class Group \texorpdfstring{$L$}{L}-functions}\label{sec7}
Using the geometric factor $\mathscr Z$ from \eqref{eq4.Z}, define
\begin{equation}\label{eq7.main}
\mathscr M(\mathbf s;\mathfrak n)=2^{1-\omega_F(\mathfrak n)}N_F(\mathfrak n)^{s_2}
\sum_{\epsilon_1,\epsilon_2=\pm1}\frac{\mathscr Z(\epsilon_1s_1,\epsilon_2s_2)}{D_E^{(1-\epsilon_1)s_1/2+(1-\epsilon_2)s_2/2} N_F(\mathfrak n)^{(1-\epsilon_2)s_2}}.
\end{equation}
The finite part is always taken after summing the four terms; the individual summands need not be regular at the origin.

\begin{thm}\label{thm7.1}
Let $\mathfrak n$ be supported on primes that split in $E$ and be coprime to $2\mathfrak D_{E/F}\mathfrak D_F$. For $|\Re(s_1)|+|\Re(s_2)|<1/2$ one has the meromorphic identity
\begin{equation}\label{7.1}
\sum_{\chi\in\widehat{\mathcal C}_{E/F}^{\mathrm{un}}}\frac{\lambda_\chi^*(\mathfrak n)\Lambda(1/2+s_1,\chi)\Lambda(1/2+s_2,\overline\chi)}{L(1,\eta_{E/F})}=\mathscr M(\mathbf s;\mathfrak n)+\mathcal E_{\mathrm{Spec}}(\mathbf s;\mathfrak n),
\end{equation}
where $\mathcal E_{\mathrm{Spec}}$ is the normalized cuspidal plus continuous contraction of \textsection\ref{sec5}. The identity is locally uniform on compact subsets away from its polar hyperplanes.
\end{thm}
\begin{proof}
Insert Proposition \ref{lem3.1} into Theorem \ref{thm2.3}. Since $\operatorname{vol}(\mathcal C_{E/F})=2L(1,\eta_{E/F})$ and $\lambda_\chi=2^{\omega_F(\mathfrak n)}\lambda_\chi^*$, solving for the moment produces the factor $2^{1-\omega_F(\mathfrak n)}$. Equations \eqref{eq4.Z}, \eqref{eq4.intertwine}, \eqref{eq6.res12}, and \eqref{eq6.res34} give exactly the four sign choices in \eqref{eq7.main}. The remaining terms are precisely the contractions defining $\mathcal E_{\mathrm{Spec}}(\mathbf s;\mathfrak n)$.
\end{proof}

For $\epsilon_1,\epsilon_2\in\{\pm1\}$ set
\begin{equation}\label{eq7.signterm}
\mathscr M_{\epsilon_1,\epsilon_2}(\mathbf s;\mathfrak n):=2^{1-\omega_F(\mathfrak n)}D_E^{-(1-\epsilon_1)s_1/2-(1-\epsilon_2)s_2/2}N_F(\mathfrak n)^{\epsilon_2s_2}\mathscr Z(\epsilon_1s_1,\epsilon_2s_2).
\end{equation}
Then
\begin{equation}\label{eq7.signsum}
\mathscr M(\mathbf s;\mathfrak n)=\sum_{\epsilon_1,\epsilon_2=\pm1}\mathscr M_{\epsilon_1,\epsilon_2}(\mathbf s;\mathfrak n).
\end{equation}
The two constant terms of the first Eisenstein series give the two choices of $\epsilon_1$. Applying \eqref{eq4.intertwine} to the second Eisenstein series changes $s_2$ to $-s_2$ and contributes $D_E^{-s_2}N_F(\mathfrak n)^{-2s_2}$, which is the ratio of the $\epsilon_2=-1$ and $\epsilon_2=1$ prefactors in \eqref{eq7.signterm}. The determinant \eqref{eq4.det} and the Gaussian factor $\pi^{-r_{\mathrm{ns}}-r_2}$ have already been incorporated in $\mathscr Z$. Finally, Proposition \ref{prop6.residues} supplies the contour corrections with the same transformed factors and signs. Thus the apparent diagonal and residual pieces combine before specialization into the symmetric meromorphic sum \eqref{eq7.signsum}. In particular, the central value is the joint finite part of this sum, not the sum of four separately specialized terms.

\begin{cor}\label{cor7.2}
For every $\varepsilon>0$,
\begin{equation}\label{equ7.2}
\sum_{\chi\in\widehat{\mathcal C}_{E/F}^{\mathrm{un}}}\frac{\lambda_\chi^*(\mathfrak n)|\Lambda(1/2,\chi)|^2}{L(1,\eta_{E/F})}=\mathscr M^\heartsuit(\mathfrak n)+O_{F,\varepsilon}(D_E^{5/12+\varepsilon}N_F(\mathfrak n)^{2/3+\varepsilon}),
\end{equation}
where
\begin{equation}\label{eq7.finitepart}
\mathscr M^\heartsuit(\mathfrak n)=-\frac1{4\pi^2}\int_{|s_1|=\delta}\int_{|s_2|=2\delta}\frac{\mathscr M(\mathbf s;\mathfrak n)}{s_1s_2}ds_2ds_1
\end{equation}
for any sufficiently small fixed $\delta>0$, with both circles positively
oriented. The integral extracts the joint constant term after summing the
four sign branches.
\end{cor}
\begin{proof}
Apply Cauchy's formula to \eqref{7.1} to extract the finite part in
\eqref{eq7.finitepart}. The estimates in \textsection\ref{sec5}, locally uniform in the shifts, give the stated error by taking $\delta$ sufficiently small in terms of $\varepsilon$; see also \eqref{eq7.spectral.central} below. The four residual branches have already been incorporated in $\mathscr M$ and produce no additional error term.
\end{proof}

The two geometric branches and two paired residue contributions form the
four-term main expression $\mathscr M$; there is no separate residual error.
Its combined central
value has the natural prefactor $D_E^{1/2}N_F(\mathfrak n)^{-1/2}$, together
with local weights, logarithms, and derivatives of the $L$-factors in
\eqref{eq4.Z}. For $\omega_F(\mathfrak n)\leq2$ and $e_v(\mathfrak n)\leq4$,
Lemma~\ref{lem7.main.bound} bounds it by
$O_{F,\varepsilon}(D_E^{1/2+\varepsilon}N_F(\mathfrak n)^{-1/2+\varepsilon})$.
The remainder consists entirely of the genuine cuspidal and continuous
spectra, with the error bound in \eqref{equ7.2}. These sizes refer to the combined terms, not to
separate central specializations of the polar branches.

For completeness, we record how the spectral estimate used here follows from the two contractions of \textsection\ref{sec5}. At the center their absolute values are bounded by
\begin{multline*}
N_F(\mathfrak n)^{-1/2+\varepsilon}D_E^{1/4+\varepsilon}
\sum_{\mu\in\widehat C_F(\mathfrak n)}\int_{i\mathbb R}
\frac{|\mathbf L(\mathbf0,\lambda,\mu)\Lambda_E(1/2+\lambda,\mu_E)|}
{|\Lambda_F(1+2\lambda,\mu^2)\Lambda_F(1-2\lambda,\mu^{-2})|}\,d\lambda\\
+D_E^{1/4+\varepsilon}\sum_{\mathfrak m\mid\mathfrak n}
\sum_{\pi\in\mathcal F(\mathfrak m,\mathbf1)}
\frac{\mathcal V_\infty(\pi)L(1/2,\pi_E)^{1/2}|L_F(1/2,\pi)|^2}
{N_F(\mathfrak m)^{1/2-\varepsilon}}.
\end{multline*}
Here the first line is the continuous contraction and the second is the cuspidal contraction, with the finite $L$-functions and the rapidly decreasing weight $\mathcal V_\infty$ of Proposition~\ref{prop5.1}. The character $\mu_E=\mu\circ N_{E/F}$ occurs only in the toric numerator, whereas the squares $\mu^{\pm2}$ occur only in the scattering denominators. At an unramified inert place the factor $L_v(1,\eta_v)$ in the normalized Waldspurger quotient cancels $1+q_v^{-1}$; at $v\mid\mathfrak n$ the oldvector basis contraction supplies the displayed factor $N_F(\mathfrak n)^{-1/2+\varepsilon}$ before the conductor sum.

The conductor condition is $\mathfrak m\mid\mathfrak n$, equivalently
$\mathfrak n\subseteq\mathfrak m$, as in Proposition~\ref{prop5.1}.
By \eqref{eq1.nelson.level},
\begin{equation}\label{eq7.weyl.inputs}
|L(1/2,\pi_E)|^{1/2}
\ll_{F,\varepsilon}D_E^{1/6+\varepsilon}
N_F(\mathfrak m)^{1/6+\varepsilon}|L_F(1/2,\pi)|^{1/2},
\end{equation}
with polynomial Archimedean dependence. The remaining moment has exponent
$5/2$. The fourth moment \cite[Theorem~E]{Yan26a} applies to all cuspidal
representations of conductor dividing $\mathfrak m$, without a restriction
on their local representation types. H\"older's inequality gives
\begin{equation}\label{eq7.level.secondmoment}
\begin{aligned}
\sum_{\pi\in\mathcal F(\mathfrak m,\mathbf1)}
\mathcal V_\infty(\pi)|L_F(1/2,\pi)|^{5/2}
&\leq\Big(\sum_\pi\mathcal V_\infty(\pi)|L_F(1/2,\pi)|^4\Big)^{5/8}
\Big(\sum_\pi\mathcal V_\infty(\pi)\Big)^{3/8}\\
&\ll_{F,\varepsilon}N_F(\mathfrak m)^{1+\varepsilon}.
\end{aligned}
\end{equation}
Here both sums on the right are over $\mathcal F(\mathfrak m,\mathbf1)$;
the second is bounded by the standard Weyl count. Oldvector dimensions
are at most $\prod_{v\mid\mathfrak n}(e_v(\mathfrak n)+1)$ and are
absorbed by $N_F(\mathfrak n)^\varepsilon$.
The polynomial Archimedean factor in \eqref{eq7.weyl.inputs} is absorbed
by decreasing the exponential decay rate of $\mathcal V_\infty$.
Summing the finite conductor factors over $\mathfrak m\mid\mathfrak n$
gives the required level exponent.
The continuous spectrum is treated in the same way, using
\cite[Theorem~F]{Yan26a}. Consequently,
\begin{equation}\label{eq7.spectral.central}
\mathcal E_{\mathrm{Spec}}(\mathbf0;\mathfrak n)\ll_{F,\varepsilon}D_E^{5/12+\varepsilon}N_F(\mathfrak n)^{2/3+\varepsilon}.
\end{equation}
The same argument is locally uniform on the small Cauchy contours in \eqref{eq7.finitepart}; this proves the error term in Corollary \ref{cor7.2} without introducing a separate residual error.

\begin{lemma}\label{lem7.main.bound}
If $\omega_F(\mathfrak n)\leq2$ and $e_v(\mathfrak n)\leq4$ for every $v$, then
\begin{equation}\label{eq7.main.bound}
|\mathscr M^\heartsuit(\mathfrak n)|\ll_{F,\varepsilon}D_E^{1/2+\varepsilon}N_F(\mathfrak n)^{-1/2+\varepsilon}.
\end{equation}
\end{lemma}
\begin{proof}
Form the four-term sum \eqref{eq7.signsum} before taking its finite part. At each $v\mid\mathfrak n$, replace every geometric quotient in $H_v$ by its finite sum. For $a+b\leq3$ the derivatives satisfy
\[
\partial_{s_1}^a\partial_{s_2}^bH_v(0,0)\ll_{a,b}(1+e_v(\mathfrak n)\log q_v)^{a+b+1}
\]
at the origin. Since there are at most two such places and each exponent is at most four, their product is $O_{F,\varepsilon}(N_F(\mathfrak n)^\varepsilon)$. The joint finite part uses derivatives only through order three. The remaining derivatives are fixed-order derivatives at $1$ of the completed Dedekind and quadratic Hecke factors occurring in \eqref{eq4.Z}; Cauchy's formula and convexity bound them by $O_{F,\varepsilon}(D_E^\varepsilon)$ after the polar terms cancel in \eqref{eq7.signsum}. Finally, the explicit prefactor in \eqref{eq4.Z} is $D_E^{1/2}N_F(\mathfrak n)^{-1/2}$ up to these derivative losses. This proves \eqref{eq7.main.bound}.
\end{proof}

\section{Twisted First Moment of Class Group \texorpdfstring{$L$}{L}-functions}\label{sec8}
The main result of this section is the following first moment formula.
\begin{thm}\label{thm8.1}
Let $\mathfrak{n}$ be an integral ideal whose prime divisors split in $E$ and which is coprime to $2\mathfrak{D}_{E/F}\mathfrak{D}_F$. Let $m\geq 10$ and $|\Re(s_2)|\leq 1$. We have
\begin{equation}\label{e8.1}
\sum_{\chi\in \widehat{\mathcal{C}}_{E/F}^{\text{un}}}\frac{\lambda_{\chi}^*(\mathfrak{n})\Lambda(1/2+s_2,\chi)}{L(1,\eta_{E/F})}=\mathcal{M}(s_2;\mathfrak{n})+\mathcal{M}^{\Diamond}(s_2;\mathfrak{n})+O(D_E^{-m/2}N_F(\mathfrak{n})^{m+5}),
\end{equation}
as an identity of meromorphic functions in $s_2$, where the implied constant depends only on $F$ and $m$. Here,
\begin{align*}
\mathcal M(s_2;\mathfrak n)&:=2\pi^{-r_{\mathrm{ns}}-r_2}2^{-d(1/2+s_2)}D_E^{1/2}D_F^{-1/2}N_F(\mathfrak n)^{-1/2-s_2}\Lambda_F(1+2s_2),\\
\mathcal M^\Diamond(s_2;\mathfrak n)&:=2\pi^{-r_{\mathrm{ns}}-r_2}2^{-d(1/2-s_2)}D_E^{1/2-s_2}D_F^{-1/2}N_F(\mathfrak n)^{-1/2+s_2}\Lambda_F(1-2s_2).
\end{align*}
\end{thm}

\begin{remark}\label{rem8.1-normalization}
The local formulas \eqref{8.15}--\eqref{8.17} include the Gaussian
scalars of \textsection\ref{sec3.1}. In particular, the coefficients of the poles of
$\mathcal M(s_2;\mathfrak n)$ and
$\mathcal M^{\Diamond}(s_2;\mathfrak n)$ at $s_2=0$ are opposite. Their sum is
therefore holomorphic there, as required by the left-hand side of
\eqref{e8.1}.
\end{remark}

\begin{remark}
When $F=\mathbb{Q}$ and $E$ is imaginary quadratic, the asymptotic formula \eqref{e8.1} was obtained in \cite[\textsection 3]{Blo04} with an exponential error term, and was later generalized to the case of general CM fields in \cite{Yan20}. Both proofs rely on the structure of CM points, which is not available for the general quadratic extension $E/F$ considered here. The argument below specializes the coarse Fourier spectral reciprocity of \cite[Proposition~3.1]{Yan26c} to the present spherical data. The local computations make the twisting weights and the uniform error explicit; the latter suffices for our applications.
\end{remark}

Taking $s_2=0$ in Theorem \ref{thm8.1} leads to the following. 
\begin{cor}\label{cor8.3}
Let $m\geq 10$ and $0<\varepsilon<10^{-3}$. We have
\begin{align*}
\sum_{\chi\in \widehat{\mathcal{C}}_{E/F}^{\text{un}}}\frac{\lambda_{\chi}^*(\mathfrak{n})\Lambda(1/2,\chi)}{L(1,\eta_{E/F})}=\mathcal{M}(\mathfrak{n})+\mathcal{M}^{\Diamond}(\mathfrak{n})+O(D_E^{-m/2}N_F(\mathfrak{n})^{m+5}),
\end{align*}
where the implied constant depends only on $F$ and $m$. Here, 
\begin{align*}
\mathcal{M}(\mathfrak{n}):=\frac{1}{2\pi i}\int_{|s_2|=\varepsilon}\frac{\mathcal{M}(s_2;\mathfrak{n})}{s_2}ds_2,\quad \mathcal{M}^{\Diamond}(\mathfrak{n}):=\frac{1}{2\pi i}\int_{|s_2|=\varepsilon}\frac{\mathcal{M}^{\Diamond}(s_2;\mathfrak{n})}{s_2}ds_2.
\end{align*}	
\end{cor}
\begin{proof}
Apply Cauchy's formula to \eqref{e8.1} on a sufficiently small circle around $s_2=0$. The left-hand side is holomorphic there. On the right-hand side, the residues of $\mathcal M(s_2;\mathfrak n)$ and $\mathcal M^\Diamond(s_2;\mathfrak n)$ cancel: this follows either from the functional equation of $\Lambda_F$ in the displayed formulas of Theorem \ref{thm8.1}, or directly from the paired constant terms in \eqref{8.4}. Hence the contour integral is the finite part of their sum. The estimate in Lemma \ref{lem8.4} is uniform on the circle, so its contour integral has the same bound. This proves the corollary.
\end{proof}

When $F$ is totally real and $E/F$ is CM, taking the finite part at $s_2=0$
in the two main terms gives
\begin{equation}\label{e8.3-cm}
\mathcal{M}(\mathfrak n)+\mathcal{M}^{\Diamond}(\mathfrak n)
=A_F\sqrt{\frac{D_E}{N_F(\mathfrak n)}}
\left(\log\frac{D_E^{1/2}}{N_F(\mathfrak n)}+B_F\right),
\end{equation}
where $A_F>0$ and $B_F\in\mathbb R$ are explicit constants depending only on
the fixed Haar and Gamma-factor conventions over $F$. Formula \eqref{e8.3-cm}
is the form used in the mollification argument below. For $F=\mathbb Q$, after
passing from completed to finite $L$-functions, it is the first-moment formula
in \cite[(3.1)]{Blo04}.

\subsection{Spectral Expansion}
Let $x\in \mathbb{A}_E^{\times}$. By the toric Fourier expansion along $\widehat{\mathcal{C}}_{E/F}$, together with \eqref{e2.8}, we obtain 
\begin{equation}\label{8.1}
E_2(\iota(x),s_2)=\frac1{2L(1,\eta_{E/F})}\sum_{\chi\in \widehat{\mathcal{C}}_{E/F}}\overline{\chi}(x)Z(s_2,\Phi_2,\chi),
\end{equation}
where $Z(s_2,\Phi_2,\chi)$ is defined by \eqref{2.8}.

By Proposition \ref{lem3.1}, we obtain from \eqref{8.1}, with $x=1$, that 
\begin{equation}\label{8.2}
E_2(I_2,s_2)=\frac{N_F(\mathfrak{n})^{-s_2}}{2L(1,\eta_{E/F})}\sum_{\chi\in \widehat{\mathcal{C}}_{E/F}^{\text{un}}}\lambda_{\chi}(\mathfrak{n})\Lambda(1/2+s_2,\chi).
\end{equation}

\subsection{Fourier Expansion}
By the Fourier expansion  \eqref{eq2.6} we have 
\begin{equation}\label{8.4}
E_2(I_2,s_2)=h_2(I_2,s_2)+h_2^{\Diamond}(I_2,s_2)+\mathcal{G}_2(I_2,s_2).
\end{equation}

\begin{lemma}\label{lem8.2}
The following identity holds:
\begin{equation}\label{e8.5}
h_2(I_2,s_2)=\pi^{-r_{\mathrm{ns}}-r_2}2^{\omega_F(\mathfrak{n})-(1/2+s_2)d}D_E^{1/2}D_F^{-1/2}N_F(\mathfrak{n})^{-(1/2+2s_2)}\Lambda_F(1+2s_2).
\end{equation}
\end{lemma}
\begin{proof}
Suppose $\Re(s_2)\gg 1$. By definition, 
\begin{align*}
h_2(I_2,s_2)=
\int_{\mathbb{A}_F^{\times}}
\Phi((0,t))|t|^{1+2s_2}d^{\times}t.
\end{align*}

For later use, note that for $v\in \Sigma_{F,\infty}$, 
\begin{equation}\label{8.6}
\mathbf{B}=\begin{pmatrix}
1 & -1\\
1 & 1
\end{pmatrix}=\sqrt{2}\cdot \begin{pmatrix}
1/\sqrt{2} & -1/\sqrt{2}\\
1/\sqrt{2} & 1/\sqrt{2}
\end{pmatrix},\quad \begin{pmatrix}
1/\sqrt{2} & -1/\sqrt{2}\\
1/\sqrt{2} & 1/\sqrt{2}
\end{pmatrix}\in K_v.
\end{equation}
Moreover,
\begin{equation}\label{8.7}
\int_{\mathbb{R}^{\times}}e^{-2\pi t^2} |t|^{1+2s_2}d^{\times}t
=
2^{-s_2-1/2}\Gamma_{\mathbb{R}}(1+2s_2).
\end{equation}

Using \eqref{8.6}, \eqref{8.7}, and the definition of $\Phi_v$ in \textsection\ref{sec3.1}, a change of variables gives 
\begin{equation}\label{8.8}
\int_{F_{\infty}^{\times}}
\Phi_{\infty}((0,t_{\infty}))|t_{\infty}|_{\infty}^{1+2s_2}d^{\times}t_{\infty}=\pi^{-r_{\mathrm{ns}}-r_2}2^{-(1/2+s_2)d}\zeta_{F_{\infty}}(1+2s_2).
\end{equation}

From the definition of $\Phi_v$ in \textsection\ref{sec3.2}, we have
\begin{equation}\label{8.9}
\prod_{v<\infty,\ v\nmid\mathfrak n}\zeta_{F_v}(1+2s_2)^{-1}\int_{F_v^{\times}}
\Phi_v((0,t_v))|t_v|_{v}^{1+2s_2}d^{\times}t_v=D_E^{1/2}D_F^{-1/2}.	
\end{equation}

Let $v\mid\mathfrak n$. Since $\mathbf{B}\in K_v$, the definition \eqref{3.2} gives
\begin{multline}\label{8.10}
\int_{F_v^{\times}}
\Phi_v((0,t_v))|t_v|_{v}^{1+2s_2}d^{\times}t_v=q_v^{r_v/2}\int_{F_v^{\times}}
\big[\Phi_v^{\circ}((\varpi_v^{-r_v}t_v,t_v))\\
+\Phi_v^{\circ}((t_v,\varpi_v^{-r_v}t_v))\big]|t_v|_{v}^{1+2s_2}d^{\times}t_v.
\end{multline}

By a change of variable in \eqref{8.10}, we obtain 
\begin{equation}\label{8.11}
\int_{F_v^{\times}}
\Phi_v((0,t_v))|t_v|_{v}^{1+2s_2}d^{\times}t_v=2q_v^{-(1/2+2s_2)r_v}\zeta_{F_v}(1+2s_2).
\end{equation}

Multiplying \eqref{8.8}, \eqref{8.9}, and \eqref{8.11} proves
\eqref{e8.5} for $\Re(s_2)\gg1$. The remaining cases follow by
meromorphic continuation.
\end{proof}

\begin{lemma}\label{lem8.3}
The following identity holds:
\begin{align*}
h_2^{\Diamond}(I_2,s_2)=\pi^{-r_{\mathrm{ns}}-r_2}2^{\omega_F(\mathfrak n)-d(1/2-s_2)}D_E^{1/2-s_2}D_F^{-1/2}N_F(\mathfrak n)^{-1/2}\Lambda_F(1-2s_2).
\end{align*}
\end{lemma}
\begin{proof}
Inserting \eqref{e2.2} into the definition of $h_2^{\Diamond}(I_2,s_2)$, and making a change of variables, we obtain 
\begin{align*}
h_2^{\Diamond}(I_2,s_2)=\prod_{v\in \Sigma_F}h_{2,v}^{\Diamond}(I_2,s_2),
\end{align*}
where
\begin{align*}
h_{2,v}^{\Diamond}(I_2,s_2):=\int_{F_v}\int_{F_v^{\times}}
\Phi_v((t,c))|t|^{2s_2}d^{\times}tdc.
\end{align*}

We consider the following local cases. 
\begin{itemize}
\item Suppose $v\in \Sigma_{F,\infty}^{\mathrm{r,ns}}$. Using \eqref{8.6}, the definition of $\Phi_v$, and a change of variables, we obtain 
\begin{equation}\label{8.15}
h_{2,v}^{\Diamond}(I_2,s_2)=\pi^{-1}|\tau_v|_v^{-s_2}
\int_{\mathbb{R}}e^{-2\pi c^2}dc\int_{\mathbb{R}^{\times}}
e^{-2\pi t^2}|t|^{2s_2}d^{\times}t
=\frac{|\tau_v|_v^{-s_2}}{\pi\,2^{1/2+s_2}}\Gamma_{\mathbb{R}}(2s_2).
\end{equation}

\item Suppose $v\in \Sigma_{F,\infty}^{\mathrm{r,sp}}$. We obtain
\begin{equation}\label{8.16}
h_{2,v}^{\Diamond}(I_2,s_2)=\frac{|\tau_v|_v^{-s_2}}{2^{1/2+s_2}}\int_{\mathbb{R}}e^{-\pi c^2}dc\int_{\mathbb{R}^{\times}}
e^{-\pi t^2}|t|^{2s_2}d^{\times}t=\frac{|\tau_v|_v^{-s_2}}{2^{1/2+s_2}}\Gamma_{\mathbb{R}}(2s_2).	
\end{equation}

\item Suppose $v\in \Sigma_{F,\infty}^{\mathrm{c}}$. We obtain
\begin{equation}\label{8.17}
h_{2,v}^{\Diamond}(I_2,s_2)=\pi^{-2}|\tau_v|_v^{-s_2}\int_{\mathbb C}e^{-4\pi|c|^2}dc\int_{\mathbb C^\times}
e^{-4\pi|t|^2}|t|^{4s_2}d^\times t
=\frac{|\tau_v|_v^{-s_2}\Gamma_{\mathbb C}(2s_2)}{\pi\,2^{1+2s_2}}.
\end{equation}

\item Suppose $v<\infty$ with $v\nmid\mathfrak n$. Then
\begin{align*}
\begin{aligned}
h_{2,v}^{\Diamond}(I_2,s_2)&=q_v^{d_{E_v}/2}\int_{F_v}\int_{F_v^{\times}}
\Phi_v^\circ(t/u_v,c-u_{0,v}t/u_v)|t|_v^{2s_2}d^\times t\,dc
\\&=|u_v|_v^{2s_2}q_v^{d_{E_v}/2-d_{F_v}}\zeta_{F_v}(2s_2).
\end{aligned}
\end{align*}

\item Suppose $v\mid\mathfrak n$. Then
\begin{align*}
h_{2,v}^{\Diamond}(I_2,s_2)=q_v^{r_v/2}h_{2,v}^{\Diamond,(1)}(I_2,s_2)+q_v^{r_v/2}h_{2,v}^{\Diamond,(2)}(I_2,s_2),\end{align*}
where 
\begin{align*}
&h_{2,v}^{\Diamond,(1)}(I_2,s_2):=\int_{F_v}\int_{F_v^{\times}}
\Phi_v^{\circ}\left((t,c)\begin{pmatrix}
1 & -1\\
1 & 1
\end{pmatrix}\begin{pmatrix}
\varpi_v^{-r_v}\\
& 1
\end{pmatrix}\right)|t|_v^{2s_2}d^{\times}tdc,\\
&h_{2,v}^{\Diamond,(2)}(I_2,s_2):=\int_{F_v}\int_{F_v^{\times}}
\Phi_v^{\circ}\left((t,c)\begin{pmatrix}
1 & -1\\
1 & 1
\end{pmatrix}\begin{pmatrix}
1\\
& \varpi_v^{-r_v}
\end{pmatrix}\right)|t|_v^{2s_2}d^{\times}tdc.
\end{align*}

By a change of variables, 
\begin{equation}\label{8.20}
h_{2,v}^{\Diamond,(1)}(I_2,s_2)=\int_{\mathcal{O}_v}\int_{F_v^{\times}}
\mathbf{1}_{\mathcal{O}_v}((c+t)\varpi_v^{-r_v})|t|_v^{2s_2}d^{\times}tdc.
\end{equation}

The condition $(c+t)\varpi_v^{-r_v}\in \mathcal{O}_v$ is equivalent to
\begin{equation}\label{8.21}
t\in \mathcal{O}_v-\{0\},\quad c\in -t+\varpi_v^{r_v}\mathcal{O}_v.
\end{equation}

Swapping the order of integration in \eqref{8.20}, and using \eqref{8.21}, we obtain 
\begin{align*}
h_{2,v}^{\Diamond,(1)}(I_2,s_2)=q_v^{-r_v}\zeta_{F_v}(2s_2).
\end{align*}
The same computation applies to the second term:
\begin{align*}
h_{2,v}^{\Diamond,(2)}(I_2,s_2)=q_v^{-r_v}\zeta_{F_v}(2s_2).	
\end{align*}
\end{itemize}

The integral-basis discriminant satisfies
$\prod_{v<\infty}|u_v|_v=2^dN_F(\mathfrak a)$; the twisting primes
have $|u_v|_v=1$. Multiplying the local identities, and using
$|\tau_\infty|_\infty=D_{E/F}N_F(\mathfrak a)^2$, gives
\[
h_2^\Diamond(I_2,s_2)=
\pi^{-r_{\mathrm{ns}}-r_2}2^{\omega_F(\mathfrak n)-d/2+ds_2}
D_E^{1/2}D_F^{-1}D_{E/F}^{-s_2}
N_F(\mathfrak n)^{-1/2}\Lambda_F(2s_2).
\]
Lemma~\ref{lem8.3} follows from the functional equation of $\Lambda_F$.
The identity extends by meromorphic continuation.
\end{proof}

\begin{lemma}\label{lem8.4}
Let $|\Re(s_2)|\leq 1$ and $m\geq 10$. We have 
\begin{equation}\label{8.24}
\mathcal{G}_2(I_2,s_2)\ll N_F(\mathfrak{n})^{10}\left(\frac{\sqrt{D_E}}{N_F(\mathfrak{n})}\right)^{-m},
\end{equation}	
where the implied constant depends on $F$ and $m$.
\end{lemma}
\begin{proof}
Let $r>1+|\Re(s_2)|$ be a large constant. Consider 
\begin{align*}
F(y):=|y|^{r}\sum_{\alpha\in F^{\times}}
W_2\left(\begin{pmatrix}
\alpha y\\
& 1
\end{pmatrix},s_2\right),\quad y\in \mathbb{A}_F^{\times}.
\end{align*}

For $r$ sufficiently large, we have $F(y)\in L^1(F^{\times}\backslash\mathbb{A}_F^{\times})\cap L^2(F^{\times}\backslash\mathbb{A}_F^{\times})$. Expanding $F(\cdot)$ along the spectrum of $F^{\times}\backslash\mathbb{A}_F^{\times}$ gives
\begin{equation}\label{8.25}
F(y)=\frac1{2\pi\rho_F}\sum_{\xi\in \widehat{C}_F}\int_{\mathbb{R}}\int_{F^{\times}\backslash\mathbb{A}_F^{\times}}F(a)\xi(a)|a|^{it}d^{\times}a\, \overline{\xi}(y)|y|^{-it} dt.
\end{equation}

For $r$ sufficiently large, 
\begin{align*}
\int_{F^{\times}\backslash\mathbb{A}_F^{\times}}F(a)\xi(a)|a|^{it}d^{\times}a=\Psi(W_2(\cdot,s_2),\xi|\cdot|^{r+it}),
\end{align*}
where the Hecke integral on the right-hand side is defined in \textsection\ref{sec2.3.1}. Taking $y=1$ in \eqref{8.25}, we obtain
\begin{equation}\label{8.26}
\mathcal{G}_2(I_2,s_2)=F(1)=\frac1{2\pi\rho_F}\sum_{\xi\in \widehat{C}_F}\int_{\mathbb{R}}\Psi(W_2(\cdot,s_2),\xi|\cdot|^{r+it})dt.
\end{equation}

By definition, 
\begin{align*}
\Psi(W_2(\cdot,s_2),\xi|\cdot|^s)=\prod_{v\in \Sigma_F}\Psi_v(W_{2,v}(\cdot,s_2),\xi_v|\cdot|_v^s),
\end{align*}
where 
\begin{align*}
\Psi_v(W_{2,v}(\cdot,s_2),\xi_v|\cdot|_v^s):=\int_{F_v^{\times}}
W_{2,v}\left(\begin{pmatrix}
a_v\\
& 1
\end{pmatrix},s_2\right)\xi_v(a_v)|a_v|_v^sd^{\times}a_v.
\end{align*}

We record the required local computations. For the Hecke integrals below,
take $\theta_v=\sqrt{|\tau_v|}$ at a real non-split place; thus every
argument of $\xi_v$ belongs to $F_v^\times$.
\begin{itemize}
\item Suppose $v\in \Sigma_{F,\infty}^{\mathrm{r,ns}}$. By \eqref{2.2} and a change of variables, we have
\begin{multline}\label{e8.27}
\Psi_v(W_{2,v}(\cdot,s_2),\xi_v|\cdot|_v^s)=\int_{F_v^{\times}}
\int_{F_v^{\times}}\int_{F_v}
\Phi_{v}((a_v,c_v))\overline{\psi}_v(c_vt_v)dc_v\\ 
\xi_v(t_v)|t_v|_v^{1/2-s_2+s}d^{\times}t_v\xi_v(a_v)|a_v|_v^{1/2+s_2+s}d^{\times}a_v.
\end{multline}
By definition,  
\begin{equation}\label{8.28}
\Phi_{v}((a_v,c_v))=\Phi_v^{\circ}((|\tau_v|_v^{1/2}a_v,c_v))=\pi^{-1}e^{-2\pi|\tau_v|_va_v^2}e^{-2\pi c_v^2}.
\end{equation}
After substituting \eqref{8.28} into \eqref{e8.27} and changing
variables, we obtain
\begin{multline}\label{8.29}
\Psi_v(W_{2,v}(\cdot,s_2),\xi_v|\cdot|_v^s)=\pi^{-1}\xi_v(\theta_v)^{-1}|\tau_v|_v^{-\frac{1}{4}-\frac{s_2+s}{2}}\int_{F_v^{\times}}\int_{F_v}e^{-2\pi c_v^2}\overline{\psi}_v(c_vt_v)dc_v\\
\xi_v(t_v)|t_v|_v^{1/2-s_2+s}d^{\times}t_v\int_{F_v^{\times}}
e^{-2\pi a_v^2}\xi_v(a_v)|a_v|_v^{1/2+s_2+s}d^{\times}a_v\mathbf{1}_{\text{$\xi_v$  unram}}.
\end{multline}
Here $\mathbf{1}_{\text{$\xi_v$  unram}}$ denotes the indicator that $\xi_v$ is unramified. 

Computing the local integrals in \eqref{8.29}, we obtain 
\begin{equation}\label{8.30}
\Psi_v(W_{2,v}(\cdot,s_2),\xi_v|\cdot|_v^s)=\frac{L_v(1/2+s+s_2,\xi_v)L_v(1/2+s-s_2,\xi_v)}{\pi\,2^{1/2+s_2}\xi_v(\theta_v)|\tau_v|_v^{1/4+(s_2+s)/2}}\mathbf{1}_{\text{$\xi_v$ unram}}.
\end{equation}

\item Suppose $v\in \Sigma_{F,\infty}^{\mathrm{r,sp}}\bigsqcup \Sigma_{F,\infty}^{\mathrm{c}}$. By \eqref{8.6} and a change of variables,  
\begin{equation}\label{8.27}
\Psi_v(W_{2,v}(\cdot,s_2),\xi_v|\cdot|_v^s)=\frac{L_v(1/2+s+s_2,\xi_v)L_v(1/2+s-s_2,\xi_v)}{2^{[F_v:\mathbb R](1/2+s_2)}\xi_v(\theta_v)|\tau_v|_v^{1/4+(s_2+s)/2}}\mathbf{1}_{\text{$\xi_v$ unram}}.
\end{equation}

\item Now suppose $v<\infty$ and $v\nmid\mathfrak n$. Put
$S_F=\{v<\infty:v\mid2\mathfrak D_F\}$ and
$a_\pm=1/2+s\pm s_2$. Using \eqref{eq3.1} in the local Hecke integral
and integrating first in $c_v$ gives
\begin{multline*}
\Psi_v=q_v^{(d_{E_v}-d_{F_v})/2}|u_v|_v^{a_+}\xi_v(u_v)
\\\int_{\mathcal O_v\cap F_v^\times}\int_{\varpi_v^{-d_{F_v}}\mathcal O_v\cap F_v^\times}
\overline\psi_v(u_{0,v}bt)\xi_v(bt)|b|_v^{a_+}|t|_v^{a_-}
d^\times t\,d^\times b.
\end{multline*}
If $v\notin S_F$, then $d_{F_v}=0$ and $u_{0,v}\in\mathcal O_v$.
The additive character in this integral is trivial, and hence
\begin{equation}\label{8.32}
\Psi_v(W_{2,v}(\cdot,s_2),\xi_v|\cdot|_v^s)=q_v^{(d_{E_v}-d_{F_v})/2}|u_v|_v^{a_+}\xi_v(u_v)\prod_{\epsilon=\pm1}L_v(1/2+s+\epsilon s_2,\xi_v)\mathbf1_{\xi_v\text{ unram}}.
\end{equation}
At $v\in S_F$, use the displayed double integral. Its normalized quotient
by the two local $L$-factors is uniformly bounded on $\Re s=r$ for
$r$ sufficiently large, and vanishes unless the conductor exponent of
$\xi_v$ is bounded in terms of $F$: $\mathfrak a$ is prime to
$S_F$, and the integral-basis discriminant bounds the valuations of
$u_v$ and $u_{0,v}$ there independently of $E$. These finitely many
quotients are included in $\mathcal T_{\mathfrak n}$ below.

\item Finally suppose $v\mid\mathfrak n$. Then 
\begin{equation}\label{e8.33}
W_{2,v}(g_v,s_2)=W_{2,v}^{(1)}(g_v,s_2)+W_{2,v}^{(2)}(g_v,s_2), 
\end{equation}
where 
\begin{align*}
&W_{2,v}^{(1)}(g_v,s_2):=q_v^{-r_vs_2}W_v^{\circ}\left(g_vg_v^{\dag}\begin{pmatrix}
\varpi_v^{-r_v}\\
& 1
\end{pmatrix},s_2\right),\\
&W_{2,v}^{(2)}(g_v,s_2):=q_v^{-r_vs_2}W_v^{\circ}\left(g_vg_v^{\dag}\begin{pmatrix}
1\\
& \varpi_v^{-r_v}
\end{pmatrix},s_2\right).
\end{align*}
Since $2\theta_v\in\mathcal O_v^\times$, we have
\begin{align*}
g_v^{\dag}\begin{pmatrix}
\varpi_v^{-r_v}\\
& 1
\end{pmatrix}\in \begin{pmatrix}
1 & \theta_v\\
& 1
\end{pmatrix}\begin{pmatrix}
1\\
& \varpi_v^{-r_v}
\end{pmatrix}K_v.
\end{align*}
After a change of variables,
\begin{equation}\label{8.33}
\Psi_v(W_{2,v}^{(1)}(\cdot,s_2),\xi_v|\cdot|_v^s)=q_v^{r_v(s-s_2)}\mathcal{T}_v^{(1)}(s_2;s,\xi_v)\prod_{\epsilon\in \{\pm 1\}}L_v(1/2+s+\epsilon s_2,\xi_v),
\end{equation}
where 
\begin{multline*}
\mathcal{T}_v^{(1)}(s_2;s,\xi_v):=\overline{\xi}_v(\theta_v\varpi_v^{r_v})\prod_{\epsilon\in \{\pm 1\}}L_v(1/2+s+\epsilon s_2,\xi_v)^{-1}\\
\sum_{n\geq 0}\frac{\xi_v(\varpi_v)^n
}{q_v^{ns}}W_v^{\circ}\left(\begin{pmatrix}
\varpi_v^n\\
& 1
\end{pmatrix},s_2\right)
G_v(\xi_v,\psi_v;n-r_v).
\end{multline*}
Here
\begin{align*}
G_v(\xi_v,\psi_v;n-r_v):=\int_{\mathcal{O}_v^{\times}}\psi_v(\varpi_v^{n-r_v}\beta)\xi_v(\beta)d^{\times}\beta.
\end{align*}
Similarly,
\begin{equation}\label{8.34}
\Psi_v(W_{2,v}^{(2)}(\cdot,s_2),\xi_v|\cdot|_v^s)=q_v^{r_v(s-s_2)}\mathcal{T}_v^{(2)}(s_2;s,\xi_v)\prod_{\epsilon\in \{\pm 1\}}L_v(1/2+s+\epsilon s_2,\xi_v),
\end{equation}
Define $\mathcal{T}_v^{(2)}(s_2;s,\xi_v)$ analogously to $\mathcal{T}_v^{(1)}(s_2;s,\xi_v)$. The second translate has unipotent coordinate $-\theta_v$, so replace $G_v(\xi_v,\psi_v;n-r_v)$ by $G_v(\xi_v,\overline{\psi}_v;n-r_v)$.

Let $r_{\xi_v}$ be the conductor exponent of $\xi_v$. By the basic property of Gauss sums, $G_v(\xi_v,\psi_v;n-r_v)\equiv 0$ unless $r_{\xi_v}\leq \max\{r_v-n,0\}$. Substituting \eqref{8.33} and \eqref{8.34} into \eqref{e8.33}, we obtain
\begin{equation}\label{8.35}
\frac{\Psi_v(W_{2,v}(\cdot,s_2),\xi_v|\cdot|_v^s)}{\prod_{\epsilon\in \{\pm 1\}}L_v(1/2+s+\epsilon s_2,\xi_v)}=q_v^{r_v(s-s_2)}\mathbf{1}_{r_{\xi_v}\leq r_v}\sum_{j=1}^2\mathcal{T}_v^{(j)}(s_2;s,\xi_v).
\end{equation}
\end{itemize}

Let $u_{\mathrm{fin}}$ have component $u_v$ at $v\nmid\mathfrak n$ and $1$
at $v\mid\mathfrak n$; the twisting-place phases are already included in
$\mathcal T_v^{(j)}$. Substituting
\eqref{8.30}, \eqref{8.27}, \eqref{8.32}, and \eqref{8.35} into
\eqref{8.26}, and using $\prod_{v<\infty}|u_v|_v=2^dN_F(\mathfrak a)$,
yields
\begin{multline}\label{8.37}
\mathcal{G}_2(I_2,s_2)=\frac{\pi^{-r_{\mathrm{ns}}}}{2\pi i\rho_F}D_{E/F}^{-\frac{1+2s_2}{4}}D_E^{1/2}D_F^{-1/2}N_F(\mathfrak{n})^{-s_2}\sum_{\xi\in \widehat{C}_F(\mathfrak{n}^2;S_F)}\xi_{\mathrm{fin}}(u_{\mathrm{fin}})\overline{\xi_\infty(\theta_\infty)}\\
\int_{(r)}\Lambda(1/2+s+s_2,\xi)\Lambda(1/2+s-s_2,\xi)
2^{ds}D_{E/F}^{-\frac{s}{2}}N_F(\mathfrak{n})^{s}\mathcal{T}_{\mathfrak{n}}(s_2;s,\xi)ds,
\end{multline}
where
\begin{align*}
\mathcal T_{\mathfrak n}(s_2;s,\xi):=\mathcal E_{S_F}(s_2;s,\xi)
\prod_{v\mid\mathfrak n}\sum_{j=1}^2\mathcal T_v^{(j)}(s_2;s,\xi_v).
\end{align*}
The exceptional factor removes the common factors already extracted in
\eqref{8.37}:
\[
\mathcal E_{S_F}(s_2;s,\xi):=\prod_{v\in S_F}
\frac{\Psi_v(W_{2,v}(\cdot,s_2),\xi_v|\cdot|_v^s)}
{q_v^{(d_{E_v}-d_{F_v})/2}|u_v|_v^{a_+}\xi_v(u_v)L_v(a_+,\xi_v)L_v(a_-,\xi_v)},
\]
Thus it is the product of the normalized
double integrals above, without repeating their prefactors.
The family $\widehat C_F(\mathfrak n^2;S_F)$ consists of
characters unramified outside $S_F\cup\{v:v\mid\mathfrak n\}$ and at
infinity, with $r_{\xi_v}\leq e_v(\mathfrak n)$ at $v\mid\mathfrak n$
and a fixed conductor bound depending only on $F$ at $v\in S_F$.

We have the trivial bound 
\begin{equation}\label{8.38}
\mathcal{T}_{\mathfrak{n}}(s_2;s,\xi)\ll N_F(\mathfrak{n})^{\varepsilon},
\end{equation}
where the implied constant depends on $F$, $r$, and $\varepsilon$. By the triangle inequality and \eqref{8.38}, we obtain from \eqref{8.37}
\begin{equation}\label{8.39}
\mathcal G_2(I_2,s_2)\ll_{F,r,\varepsilon}
D_E^{(1-2\Re(s_2))/4}D_{E/F}^{-r/2}
N_F(\mathfrak n)^{r-\Re(s_2)+\varepsilon}\mathcal S_{\mathfrak n}(s_2),
\end{equation}
where
\begin{align*}
\mathcal{S}_{\mathfrak{n}}(s_2):=\sum_{\xi\in \widehat{C}_F(\mathfrak{n}^2;S_F)}\int_{(r)}\big|\Lambda(1/2+s+s_2,\xi)\Lambda(1/2+s-s_2,\xi)\big|\,|ds|.
\end{align*}

By Stirling's formula,
\begin{equation}\label{8.40}
\mathcal{S}_{\mathfrak{n}}(s_2)\ll N_F(\mathfrak{n})^{2},	
\end{equation}
where the implied constant depends on $r$ and $F$. Taking $r=m+2$ in
\eqref{8.39} proves \eqref{8.24}, with room in the stated power of
$N_F(\mathfrak n)$.
\end{proof}

\subsection{Proof of Theorem \ref{thm8.1}}
Combine \eqref{8.2} and \eqref{8.4} with
Lemmas~\ref{lem8.2}--\ref{lem8.4}. The conversion from $E_2(I_2,s_2)$
to the normalized moment contributes the factor
$2^{1-\omega_F(\mathfrak n)}N_F(\mathfrak n)^{s_2}$.
In \eqref{8.39}--\eqref{8.40}, take $r=m+2$. Including the factor
$N_F(\mathfrak n)^{-\Re(s_2)}$ when forming the moment gives
$O_{F,m,\varepsilon}(D_E^{-m/2}N_F(\mathfrak n)^{m+4+\varepsilon})$,
which implies the stated bound.

\section{Nonvanishing of Class Group \texorpdfstring{$L$}{L}-functions}\label{sec9}
Throughout this section, $F$ is a fixed totally real field and $E/F$ is
CM. Write
\[
h_{E/F}:=|\mathrm{Cl}_{E/F}|,
\qquad
c_{E/F}:=\prod_{\mathfrak p\mid\mathfrak D_{E/F}}
\left(1-N_F(\mathfrak p)^{-1}\right).
\]
By the relative class number formula,
\begin{equation}\label{eq9.nv.classnumber}
h_{E/F}\asymp_F D_E^{1/2}L(1,\eta_{E/F});
\end{equation}
the Hasse unit index is absorbed in the $F$-constants. We multiply the
split-prime mollifier by inverse Euler factors at small ramified primes.
The unit-group tests below realize this modification in the two moment
formulas.

\subsection{Moments with Ramified Euler Factors Removed}
We first bound the ramified Euler product.

\begin{lemma}\label{lem9.ramified.relative}
For every $\varepsilon>0$,
\begin{equation}\label{eq9.ramified.relative}
 c_{E/F}^{-1}\ll_{F,\varepsilon}D_E^{\varepsilon}.
\end{equation}
\end{lemma}

\begin{proof}
For \eqref{eq9.ramified.relative}, split the product at
$N_F(\mathfrak p)\leq D_E^{\varepsilon/2}$ and use the Mertens estimate for
$F$ on the first part.  On the complementary part, the number of prime-ideal
divisors of $\mathfrak D_{E/F}$ is at most
$2\varepsilon^{-1}$, up to a constant depending on $F$, and each factor is
$1+O(D_E^{-\varepsilon/2})$.  This proves \eqref{eq9.ramified.relative}.
\end{proof}

For each prime $\mathfrak p\mid\mathfrak D_{E/F}$, let
$\mathfrak P_{\mathfrak p}$ be the unique prime of $E$ above $\mathfrak p$,
write $q_{\mathfrak p}=N_F(\mathfrak p)$ and
$\epsilon_{\mathfrak p}(\chi)=\chi(\mathfrak P_{\mathfrak p})\in\{\pm1\}$.

We have the following twisted moments with ramified Euler factors removed.
\begin{thm} \label{thm9.unit.group}
Let $E/F$ be a CM quadratic extension of a fixed totally real field $F$.
Let $\mathfrak n$ be an integral ideal supported on primes that split
in $E$ and coprime to $2\mathfrak D_{E/F}\mathfrak D_F$.
Let $S$ be any subset of the prime ideals dividing
$\mathfrak D_{E/F}$ that are coprime to $6\mathfrak D_F$. Put
\[
R_S:=\prod_{\mathfrak p\in S}q_{\mathfrak p},\qquad
c_S:=\prod_{\mathfrak p\in S}(1-q_{\mathfrak p}^{-1}),\qquad
g_S:=\prod_{\mathfrak p\in S}(1+q_{\mathfrak p}^{-1}),
\]
and
\[
\mathcal B_S(s,\chi):=\prod_{\mathfrak p\in S}
\big(1-\epsilon_{\mathfrak p}(\chi)q_{\mathfrak p}^{-1/2-s}\big).
\]
For $m\geq10$ and $\varepsilon>0$, the central first moment is
\begin{multline}\label{eq9.unit.first.center}
\sum_{\chi\in\widehat{\mathcal C}_{E/F}^{\mathrm{un}}}
\frac{\lambda_\chi^*(\mathfrak n)\mathcal B_S(0,\chi)
\Lambda(1/2,\chi)}{L(1,\eta_{E/F})}
=c_S A_F\sqrt{\frac{D_E}{N_F(\mathfrak n)}}
\Big(\log\frac{D_E^{1/2}}{N_F(\mathfrak n)}+B_F+
\sum_{\mathfrak p\in S}\frac{\log q_{\mathfrak p}}{q_{\mathfrak p}-1}\Big)
\\+O_{F,m,\varepsilon}\big(D_E^{-m/2-3/4+\varepsilon}
N_F(\mathfrak n)^{m+4+\varepsilon}R_S^{m+3/2+\varepsilon}\big),
\end{multline}
where $A_F$ and $B_F$ are as in \eqref{e8.3-cm}.
The central second moment is
\begin{equation}\label{eq9.unit.second}
\sum_{\chi\in\widehat{\mathcal C}_{E/F}^{\mathrm{un}}}
\frac{\lambda_\chi^*(\mathfrak n)\mathcal B_S(0,\chi)^2
|\Lambda(1/2,\chi)|^2}{L(1,\eta_{E/F})}
=\mathscr M_S^\heartsuit(\mathfrak n)+O_{F,\varepsilon}\big(D_E^{5/12+\varepsilon}
N_F(\mathfrak n)^{2/3+\varepsilon}R_S^\varepsilon\big).
\end{equation}
Here $\mathscr M_S^\heartsuit(\mathfrak n)$ is the finite part at
$\mathbf s=\mathbf0$ of the main term in \eqref{eq9.unit.main},
taken as in \eqref{eq7.finitepart}.
\end{thm}

\begin{proof}
In the coordinates of \textsection\ref{sec3.2}, replace both local tests
at $v\in S$ by
\begin{align*}
\Phi_v^\times=q_v^{d_{E_v}/2}
(\mathbf1_{\mathcal O_{E_v}}-\mathbf1_{\mathfrak P_v})
=q_v^{d_{E_v}/2}\mathbf1_{\mathcal O_{E_v}^\times}.
\end{align*}
For an unramified character $\chi_v$ trivial on $F_v^\times$,
\begin{align*}
Z_v(s,\Phi_v^\times,\chi_v)
=(1-\epsilon_v(\chi_v)q_v^{-1/2-s})L_v(1/2+s,\chi_v)=1.
\end{align*}

Fix $v\in S$ and write $q=q_v$.
Let $\Phi_{0,v}=q^{d_{E_v}/2}\mathbf1_{\mathcal O_{E_v}}$ and
$\Phi_{1,v}=q^{d_{E_v}/2}\mathbf1_{\mathfrak P_v}$, with corresponding
sections $h_{0,v}$ and $h_{1,v}$. Since
$\mathfrak P_v\cap F_v=\mathfrak p_v$, one has
\[
h_{1,v}(I_2,s)=q^{-1-2s}h_{0,v}(I_2,s).
\]
By the local intertwining calculation preceding \eqref{eq4.17},
\[
h_{0,v}^{\Diamond}(g,s)=c_v(s)h_{0,v}(g,-s),\qquad
c_v(s)=|\det g_{0,v}|_v^{-2s}q_v^{-d_{F_v}/2}
\frac{\zeta_{F_v}(2s)}{\zeta_{F_v}(1-2s)},
\]
where $g_{0,v}$ is the change-of-basis matrix used in \eqref{equ4.14}.
The section $h_{1,v}(s)$ is $q^{-1/2-s}$ times a translate of
$h_{0,v}(s)$ by multiplication by an $E_v$-uniformizer inverse.
The intertwiner commutes with this translation, whence
\[
h_{1,v}^{\Diamond}(s)=c_v(s)q^{-2s}h_{1,v}(-s).
\]
Thus, with $h_v^\times=h_{0,v}-h_{1,v}$,
\begin{equation}\label{eq9.unit.matrix}
(h_v^\times)^{\Diamond}(s)
=c_v(s)\left(h_v^\times(-s)+(1-q^{-2s})h_{1,v}(-s)\right).
\end{equation}
Consequently the two first-moment constant terms are $\prod_{\mathfrak p\in S}(1-q_{\mathfrak p}^{-1-2s})
\mathcal M(s;\mathfrak n)$ and $c_S\mathcal M^\Diamond(s;\mathfrak n)$. Their finite part at $s=0$ is the main term in \eqref{eq9.unit.first.center}.

For the first moment, at $v\in S$ choose the trace-zero integral
basis $1,u_v\theta$. Then $a\theta+c$ is a unit precisely when
$a\in u_v\mathcal O_v$ and $c\in\mathcal O_v^\times$.
Since $\psi_v$ has conductor $\mathcal O_v$,
\[
\int_{\mathcal O_v^\times}\overline\psi_v(ct)\,dc
=\mathbf1_{\mathcal O_v}(t)-q_v^{-1}\mathbf1_{\mathfrak p_v^{-1}}(t).
\]
Substitution in the Hecke integral \eqref{e8.27} multiplies
\eqref{8.32}, for unramified $\xi_v$, by
$1-\xi_v(\varpi_v)^{-1}q_v^{s-s_2-1/2}$.
The integral still vanishes for ramified $\xi_v$, so the character
family in \eqref{8.37} does not change. At $s_2=0$ and
$\Re s=m+2$, the product of these extra factors is
$O_{F,m,\varepsilon}(R_S^{m+3/2+\varepsilon})$.
Keeping the full discriminant power in \eqref{8.39} and using
\eqref{8.40} proves the error estimate in \eqref{eq9.unit.first.center}, since
\begin{align*}
D_E^{1/4}D_{E/F}^{-(m+2)/2}\ll_{F,m}D_E^{-m/2-3/4}. 
\end{align*}

To compute the second-moment main term, write $Z_{ab,v}(\mathbf s):=Z_v(\mathbf s,\Phi_{a,v},\Phi_{b,v})$,
$a,b\in\{0,1\}$, with $Z_v$ as in Proposition~\ref{prop4.1}.
Substituting the defining integral for $h_{a,v}$ into $Z_{ab,v}$,
let $x\in E_v^\times$ and $t\in F_v^\times$ denote the outer and
inner integration variables, respectively. The support conditions are
$x\in\mathfrak P_v^b$ and $tx\in\mathfrak P_v^a$.
Writing $j=e_{E_v}(x)$ and $k=e_v(t)$, these become
$j\geq b$ and $j+2k\geq a$, since $E_v/F_v$ is ramified quadratic. With
$u_i=q^{-1-2s_i}$ and $r=q^{-1-s_1-s_2}$, we can sum the geometric series explicitly:
\begin{align*}
\begin{pmatrix}Z_{00,v}&Z_{01,v}\\Z_{10,v}&Z_{11,v}\end{pmatrix}
=\frac{q^{(d_{E_v}-d_{F_v})/2}}{(1-u_1)(1-u_2)}
\begin{pmatrix}1+r&u_2+r\\u_1+r&r(1+r)\end{pmatrix}.
\end{align*}
For $\epsilon=(\epsilon_1,\epsilon_2)\in\{\pm1\}^2$, write
\[
t_i=q^{-(1-\epsilon_i)s_i},\qquad
u_{i,\epsilon}=q^{-1-2\epsilon_i s_i},\qquad
r_\epsilon=q^{-1-\epsilon_1s_1-\epsilon_2s_2},
\]
and
\begin{equation}\label{eq9.unit.branch}
Q_{v,\epsilon}(\mathbf s)
=1-t_1\frac{u_{1,\epsilon}+r_\epsilon}{1+r_\epsilon}
-t_2\frac{u_{2,\epsilon}+r_\epsilon}{1+r_\epsilon}
+t_1t_2r_\epsilon.
\end{equation}
The full modified main term is
\begin{equation}\label{eq9.unit.main}
\mathscr M_S(\mathbf s;\mathfrak n)
=\sum_{\epsilon_1,\epsilon_2=\pm1}
\mathscr M_{\epsilon_1,\epsilon_2}(\mathbf s;\mathfrak n)
\prod_{v\in S}Q_{v,\epsilon}(\mathbf s),
\end{equation}
and its finite part is taken as in \eqref{eq7.finitepart}.
For a negative sign in variable $i$, \eqref{eq9.unit.matrix}
replaces the lattice difference by
$\Phi_{0,v}-q^{-2s_i}\Phi_{1,v}$ at shift $-s_i$.
For the residual branches, the unfolding in
\textsection\ref{sec2.4.2}, applied to the constant-term monomial
$|a|^{1/2+\lambda_0}f(k)$ in \eqref{eq2.6}, evaluates the residue as
\[
\underset{\lambda=\lambda_0}{\Res}
\sum_h\widetilde\Psi(\lambda,h;\mathbf s)E(g,h,\lambda)
=\rho_F\Eis(f)(g).
\]
The norm integral has residue one and the norm-one idele
classes have volume $\rho_F$. Here $(f,\lambda_0)$ is
$(h_1^\Diamond h_2^\Diamond,1/2-s_1-s_2)$ or
$(h_1h_2^\Diamond,1/2+s_1-s_2)$.
After applying the functional equation and the signed-residue
normalization, the contribution at each paired pole is $-\Eis(f)/2$,
as in the proof of
Proposition~\ref{prop6.residues}. Thus the minus sign in \eqref{2.3}
makes each pair contribute $\Eis(f)$ with coefficient one.
This holds for every lattice pair
$(\Phi_{a,v},\Phi_{b,v})$, $(a,b)\in\{(0,0),(0,1),(1,0),(1,1)\}$,
initially at distinct poles and then by meromorphic continuation. Dividing the resulting local
combination by $Z_{00,v}(\epsilon_1s_1,\epsilon_2s_2)$ proves
\eqref{eq9.unit.branch}, and hence \eqref{eq9.unit.main}.

In integral coordinates, $q^{-d_{E_v}/2}\Phi_v^\times
=\mathbf1_{\mathcal O_v\times\mathcal O_v^\times}$.
Let $I_v^-\subset K_v=\mathrm{GL}_2(\mathcal O_v)$ consist of matrices
with upper-right entry in $\mathfrak p_v$, and take
$\operatorname{vol}(K_v)=1$. The corresponding central section is one
on $n(u)I_v^-$, $u\in\mathcal O_v/\mathfrak p_v$, and zero on $wI_v^-$.
The two lattice sections have Whittaker functions $W_0(j)$ and
$q^{-1/2}W_0(j+1)$, where $W_0(j)=(j+1)q^{-j/2}$ for $j\geq0$ and
is zero otherwise. Their difference is $W^\times(j)=W^\times(a(\varpi_v^j),0)$, so
\[
W^\times(j)=
\begin{cases}
q^{-j/2}\bigl((j+1)-(j+2)/q\bigr),&j\geq0,\\
-q^{-1/2},&j=-1,\\
0,&j<-1.
\end{cases}
\]
Iwahori invariance allows only unramified principal series and
unramified twists of Steinberg; ramified principal series and
supercuspidals have no $I_v^-$-fixed vectors. In the spherical case, both normalized local contractions
are $O(1)$ by the standard Whittaker formulas: the oldvector Gram
matrix and its inverse, the normalized torus average, and the local
$L$-factors are uniformly bounded, using a fixed bound towards Ramanujan.

For $\pi_v=\mathrm{St}\otimes\mu_v$, write $a=\mu_v(\varpi_v)\in\{\pm1\}$.
The standard newvector has Whittaker values $a^j q^{-j}$ for $j\geq0$.
Its $w$-translate, fixed by $I_v^-$, has values
$-a^{j+2}q^{-j-1}$ for $j\geq-1$. Its norm is bounded above and below
by absolute constants. Integration over $K_v/I_v^-$ contributes
$q/(q+1)$. Thus the local Rankin--Selberg integral, with the common
lattice scalar removed, equals
\[
\frac{q}{q+1}\left\{\frac aq-
\frac1q\sum_{j\geq0}\left(\frac aq\right)^j
\left(j+1-\frac{j+2}{q}\right)\right\}.
\]
After division by $L_v(1/2,\pi_v)^2=(1-a/q)^{-2}$, the value is
$(a-1)/(q+1)$. Hence the full normalized contraction,
including the bounded toric functional, is $O(q^{-1})$.
For small shifts, the same geometric series are
$O(q^{-1+O(\varepsilon)})$. Consequently, if the conductor is $\mathfrak m\mathfrak d$,
where $\mathfrak m\mid\mathfrak n$ and $\mathfrak d\mid\prod_{v\in S}\mathfrak p_v$,
the additional local weight is
$C^{|S|}N_F(\mathfrak d)^{-1+\varepsilon}$, for a fixed $C$.

At $v\in S$, the two possibilities for $\pi_v$ give
\begin{align*}
\pi_v\otimes\eta_v=
\begin{cases}
(\mu_v\eta_v)\boxplus(\mu_v^{-1}\eta_v),&\pi_v\text{ spherical},\\
\mathrm{St}\otimes(\mu_v\eta_v),&\pi_v=\mathrm{St}\otimes\mu_v.
\end{cases}
\end{align*}
Here $\mu_v$ is unramified, and in the second case $\mu_v^2=1$.
Both twists therefore lie in $\Sigma_{F_v}(\eta_v)$ as defined in
\cite[\textsection 3.1]{Nel19}; both have conductor exponent two.
The same ramified-character weight is $\gg q_v^{-1}$ on both, by
\cite[Theorem~3.1]{Nel19}, and its dual and degenerate bounds are the
same by \cite[Proposition~3.3 and Lemma~15.6]{Nel19}.
Thus replacing a spherical component by Steinberg at these places does
not add an auxiliary-level factor to the Weyl estimate. The calculation
proving \eqref{eq1.nelson.level} gives
\begin{align*}
|L(1/2,\pi\otimes\eta_{E/F})|
\ll_{F,\varepsilon,\pi_\infty}D_E^{1/3+\varepsilon}
N_F(\mathfrak m)^{1/3+\varepsilon},
\end{align*}
with polynomial Archimedean dependence, even when the conductor of
$\pi$ is $\mathfrak m\mathfrak d$ as above. Apply it only to the twisted
factor in the base-change factorization.
For $M=N_F(\mathfrak m)$ and $d=N_F(\mathfrak d)$, the
fourth-moment estimate and Weyl count used in
\eqref{eq7.level.secondmoment}, now at level $\mathfrak m\mathfrak d$,
bound the spectral sum by
\[
\sum_{\pi\in\mathcal F(\mathfrak m\mathfrak d,\mathbf1)}
\mathcal V_\infty(\pi)|L(1/2,\pi)|^{5/2}
\ll_{F,\varepsilon}(Md)^{1+\varepsilon}.
\]
The factors from \textsection\ref{sec5}, the Weyl bound, and the
local calculation above therefore contribute
\[
M^{-1/2+\varepsilon}M^{1/6+\varepsilon}
 d^{-1+\varepsilon}(Md)^{1+\varepsilon}
=M^{2/3+O(\varepsilon)}d^{O(\varepsilon)}.
\]
Polynomial Archimedean factors are absorbed by $\mathcal V_\infty$.
Summing over $\mathfrak m\mid\mathfrak n$ and
$\mathfrak d\mid\prod_{v\in S}\mathfrak p_v$, and using
$(2C)^{|S|}\ll_{F,\varepsilon}R_S^\varepsilon$, proves the error estimate
in \eqref{eq9.unit.second} after renaming $\varepsilon$.
In the continuous spectrum, $I_v^-$-invariance
forces the inducing characters to be unramified at $v\in S$, so no
continuous conductor stratum is added there. Use
\cite[Theorem~F]{Yan26a} with its scattering denominator as the measure,
and apply the same H\"older inequality with the corresponding weighted
Weyl count. The finite conductor restriction in that theorem is
$c(\mu\boxplus\mu^{-1})\mid\mathfrak m$, exactly the one required here. This proves \eqref{eq9.unit.second}.
\end{proof}

\begin{lemma}\label{lem9.fixed.primes}
The moment formulas \eqref{eq9.unit.first.center} and \eqref{eq9.unit.second},
with main term \eqref{eq9.unit.main}, also hold for any subset $S$ of the
primes dividing $\mathfrak D_{E/F}$ and any integral ideal $\mathfrak n$
supported on split primes, provided $e_v(\mathfrak n)\leq2$ at
$v\mid6\mathfrak D_F$. The error estimates are unchanged.
\end{lemma}
\begin{proof}
Only the fixed set $T_F=\{v:v\mid6\mathfrak D_F\}$ needs consideration.
At a ramified $v\in S\cap T_F$, the matrix for $Z_{ab,v}$ above is valid
also in residue characteristic two: its derivation uses only
$j\geq b$ and $j+2k\geq a$. All dependence on the differents is in the
common scalar $q_v^{(d_{E_v}-d_{F_v})/2}$, which cancels upon division
by $Z_{00,v}$. Likewise, $c_v(s)$ is common to the two lattices in
\eqref{eq9.unit.matrix}. The Tate integral of the normalized unit-group
test is $1$, and the ratios for the two first-moment constant terms are
$1-q_v^{-1-2s}$ and $1-q_v^{-1}$. Thus neither the first-moment factor
nor $Q_{v,\epsilon}$ acquires an extra constant at these places.

At a split $v\in T_F$, write $q=q_v$ and, in $E_v=F_v\times F_v$
coordinates, let $\Phi_{0,v}=q^{d_{E_v}/2}\mathbf1_{\mathcal O_v^2}$ and
\begin{align*}
\Phi_{r,v}=q^{d_{E_v}/2+r/2}
\big(\mathbf1_{\mathfrak p_v^r\times\mathcal O_v}
+\mathbf1_{\mathcal O_v\times\mathfrak p_v^r}\big),\qquad r\in\{1,2\}.
\end{align*}
Let $h_{r,v}$ be the corresponding section. Since
$\operatorname{vol}(\mathcal O_{E_v}^{\times})=q^{-d_{E_v}/2}$, the
Tate integrals satisfy
\begin{align*}
Z_v(s,\Phi_{0,v},\chi_v)&=L_v(1/2+s,\chi_v),\\
Z_v(s,\Phi_{r,v},\chi_v)&=2q^{-rs}\lambda_\chi^*(\mathfrak p_v^r)
L_v(1/2+s,\chi_v).
\end{align*}
Integration on the diagonal $F_v^\times$, and then in the complementary
coordinate for the intertwined section, gives
\begin{align*}
\frac{h_{r,v}(I_2,s)}{h_{0,v}(I_2,s)}=2q^{-r/2-2rs},\qquad
\frac{h_{r,v}^\Diamond(I_2,s)}{h_{0,v}^\Diamond(I_2,s)}=2q^{-r/2}.
\end{align*}
For the second moment, summing the regions $e_v(x_1)\geq e_v(x_2)$ and
$e_v(x_1)<e_v(x_2)$ gives
\begin{align*}
\frac{Z_v(s_1,s_2;\Phi_{0,v},\Phi_{r,v})}
{Z_v(s_1,s_2;\Phi_{0,v},\Phi_{0,v})}
=2q^{-r/2-r(s_1+s_2)}H_{\mathfrak p_v^r}(s_1,s_2).
\end{align*}
These are precisely the ratios in \textsection\ref{sec3.6} and
\textsection\ref{sec8}. Moreover,
\begin{align*}
h_{r,v}^\Diamond(s)=c_v(s)q^{-2rs}h_{r,v}(-s),
\end{align*}
so all four sign branches have the stated normalization. Hence the main
terms agree exactly.

For the errors, $F_v^\times/F_v^{\times2}$ is finite, so there are only
finitely many quadratic $F_v$-algebras. In integral coordinates the tests
above have support, denominators, and invariance level bounded in terms
of $F$, since $r\leq2$. In the first moment, the double integral preceding
\eqref{8.32}, after extracting its displayed discriminant factor, is
bounded on the shifted contour and vanishes unless the conductor of
$\xi_v$ is bounded in terms of $F$. Enlarge $S_F$ to $T_F$ in
\eqref{8.37}; these changes affect only its fixed local factors and
conductor bound.
In the second moment, the local contractions are bounded in terms of
$F$, and the spectral conductor acquires only a fixed factor depending
on $F$. Applying the same spectral moment estimates at that level changes
only the implied constants. The estimates outside $T_F$ are exactly
those in Theorem~\ref{thm9.unit.group}, proving the assertion.
\end{proof}

\subsection{Choice of the Mollifier}
\begin{lemma}\label{lem9.unit.cutoff}
Put $\ell=\log D_E$ and $z=\ell/20$, and let
\[
S=\big\{\mathfrak p\mid\mathfrak D_{E/F}:N_F(\mathfrak p)\leq z\big\}.
\]
The set $S$ may be empty; in that case $R_S=c_S=g_S:=1$.
For sufficiently large $D_E$,
\begin{equation}\label{eq9.unit.cutoff}
R_S\leq D_E^{1/16},\qquad
\prod_{\substack{\mathfrak p\mid\mathfrak D_{E/F}\\q_{\mathfrak p}>z}}
(1+q_{\mathfrak p}^{-1})
=1+O_F\!\left(\frac1{\log\ell}\right).
\end{equation}
\end{lemma}
\begin{proof}
For the first inequality, by the prime-ideal theorem,
\begin{align*}
\log R_S\leq\sum_{N_F(\mathfrak p)\leq z}\log N_F(\mathfrak p)
=(1+o_F(1))z\leq\frac{\ell}{16}.
\end{align*}
For the tail estimate, since $q\log q$ is increasing for $q>1$,
\begin{align*}
\sum_{\substack{\mathfrak p\mid\mathfrak D_{E/F}\\q_{\mathfrak p}>z}}
\frac1{q_{\mathfrak p}}
\leq\frac1{z\log z}
\sum_{\mathfrak p\mid\mathfrak D_{E/F}}\log q_{\mathfrak p}
\leq\frac{\ell}{z\log z}
=\frac{20}{\log z}.
\end{align*}
The assertion follows from $\log(1+x)\leq x$ and
$\log z=\log\ell-\log 20$.
\end{proof}

Use this choice of $S$ throughout the nonvanishing argument. Fix
$0<\beta<1/28$ and write $Y=D_E^\beta$.
Let $\mathcal Q(Y)$ be the set of squarefree
integral ideals $\mathfrak q$ with $N_F(\mathfrak q)\leq Y$, supported on
primes that split in $E$. Lemma~\ref{lem9.fixed.primes} applies to all
moment twists below, since their prime exponents are at most two.
For $\mathfrak q_1,\mathfrak q_2\in\mathcal Q(Y)$,
the local Hecke
relation is
\begin{equation}\label{eq9.nv.hecke}
\lambda_\chi^*(\mathfrak q_1)\lambda_\chi^*(\mathfrak q_2)
=2^{-\omega_F(\mathfrak d)}\sum_{\mathfrak r\mid\mathfrak d}
\lambda_\chi^*(\mathfrak q\mathfrak r^2),\qquad \mathfrak d=(\mathfrak q_1,\mathfrak q_2),\quad \mathfrak q=\mathfrak q_1\mathfrak q_2\mathfrak d^{-2}.
\end{equation}

In the following local formulas, write $q=N_F(\mathfrak p)$.
Let $b$ be the multiplicative function supported on squarefree ideals with
$b(\mathfrak p)=2q/(q-1)$ if $\mathfrak p$ splits in $E$, and
$b(\mathfrak p)=0$ otherwise. Define
\begin{align*}
S(X)=\sum_{N_F(\mathfrak a)\leq X}
\frac{b(\mathfrak a)}{N_F(\mathfrak a)}.
\end{align*}

\begin{lemma}\label{lem9.shortlong}
For fixed $0<\beta<\sigma$, as $D_E\to\infty$,
\begin{equation}\label{eq9.shortlong}
S(D_E^\beta)\geq
\left(\frac{e^{-2\gamma}}2+o_F(1)\right)
\frac{\beta^2}{\sigma^2}S(D_E^\sigma),
\end{equation}
where $\gamma$ is Euler's constant and the error is uniform in the
split-prime set.
\end{lemma}
\begin{proof}
Independently select each permitted prime ideal of
norm $q\leq Y$ with probability $2/(q+1)$. The ratio of $S(Y)$ to
its Euler product is the probability that the selected product has
norm at most $Y$. Adding all missing prime ideals can only decrease
this probability. For the product over all prime ideals, the prime-ideal
theorem identifies the limiting Laplace transform of its logarithmic
norm divided by $\log Y$ as
\begin{align*}
\exp\left(2\int_0^1\frac{e^{-tu}-1}{u}\,du\right).
\end{align*}
This is the generalized Dickman distribution with parameter $2$,
whose density on $(0,1)$ is $e^{-2\gamma}x$; see
\cite[\textsection 1, (1.1)]{Pin17}. Thus the limiting probability is
$e^{-2\gamma}/2$. Comparing the Euler products at
$D_E^\beta$ and $D_E^\sigma$ by Mertens' formula, using
$1+2/(q-1)=(1-q^{-2})(1-q^{-1})^{-2}$, proves \eqref{eq9.shortlong}.
\end{proof}

We choose the mollifier by the Selberg diagonalization of
\cite[\textsection 4]{Blo04}. Write $\mu_F$ and $\tau_F$ for the
ideal M\"obius and divisor functions. For a split prime $\mathfrak p$, set
\[
\alpha_0(\mathfrak p):=\frac{q}{q+1},
\qquad
\alpha_1(\mathfrak p):=\frac{q-1}{q+1},
\qquad
h(\mathfrak p):=
\frac{(q-1)^3}{2q(q+1)^2},
\]
and extend these functions multiplicatively to squarefree ideals.
For $\mathfrak a\in\mathcal Q(Y)$, write
$V(\mathfrak a)=\ell/2+\log N_F(\mathfrak a)+B_F\asymp_F\ell$,
where $B_F$ is as in \eqref{e8.3-cm}, and choose
\begin{align*}
\xi_0(\mathfrak a):=
\frac{\mu_F(\mathfrak a)\alpha_1(\mathfrak a)}
{2\ell^2h(\mathfrak a)N_F(\mathfrak a)^{1/2}}.
\end{align*}
Since $\alpha_1(\mathfrak a)^2/h(\mathfrak a)=b(\mathfrak a)$
for $\mathfrak a\in\mathcal Q(Y)$, we have
\begin{equation}\label{eq9.nv.normalization}
\sum_{\mathfrak a\in\mathcal Q(Y)}h(\mathfrak a)\xi_0(\mathfrak a)^2
=\frac{S(Y)}{4\ell^4}.
\end{equation}
Define the mollifier 
\begin{align*}
\mathcal A^\times(\chi)=\mathcal B_S(0,\chi)
\sum_{\mathfrak q\in\mathcal Q(Y)}a(\mathfrak q)\lambda_\chi^*(\mathfrak q),
\end{align*}
where 
\begin{align*}
a(\mathfrak q):=
\sum_{\substack{\mathfrak f:\ \mathfrak f\mathfrak q\in\mathcal Q(Y)}}
\frac{\mu_F(\mathfrak f)\alpha_0(\mathfrak f)\tau_F(\mathfrak f)}
{N_F(\mathfrak f)^{1/2}}\xi_0(\mathfrak f\mathfrak q),
\end{align*}
and extend $a$ and $\xi_0$ by zero outside $\mathcal Q(Y)$.
We suppress the dependence on the fixed set $S$.

The definition of $\xi_0$ bounds each term in $a(\mathfrak q)$, so
\begin{align*}
|a(\mathfrak q)|
&\leq |\xi_0(\mathfrak q)|
\sum_{\substack{\mathfrak f:\ \mathfrak f\mathfrak q\in\mathcal Q(Y)}}
\prod_{\mathfrak p\mid\mathfrak f}
\frac{4N_F(\mathfrak p)}{(N_F(\mathfrak p)-1)^2}\\
&\leq |\xi_0(\mathfrak q)|
\prod_{N_F(\mathfrak p)\leq Y}
\left(1+\frac{4N_F(\mathfrak p)}{(N_F(\mathfrak p)-1)^2}\right)
\ll_F |\xi_0(\mathfrak q)|(\log(2Y))^4.
\end{align*}
Since $|\xi_0(\mathfrak q)|\ll_{F,\varepsilon}
N_F(\mathfrak q)^{-1/2+\varepsilon}$, we obtain
\begin{equation}\label{eq9.nv.coeff}
a(\mathfrak q)\ll_{F,\varepsilon}D_E^\varepsilon
N_F(\mathfrak q)^{-1/2+\varepsilon}.
\end{equation}

For $D_E$ sufficiently large in terms of $F$,
the roots of unity in $E$ are $\{\pm1\}$ and $E/F$ is ramified at
a finite prime not dividing $2$, since the exceptions form a finite set.
By \cite[Theorem~1(i).1]{Lem95}, the Hasse unit index is $1$ and
$\mathrm{Cl}_F\to\mathrm{Cl}_E$ is injective. Thus
$R_E=2^{d-1}R_F$, so the analytic class number formula reads
\begin{equation}\label{eq9.classnumber.exact}
h_{E/F}=\frac{h_E}{h_F}
=\frac{2^{1-d}\sqrt{D_E}\,L(1,\eta_{E/F})}{\pi^d\sqrt{D_F}}.
\end{equation}

Normalize the moments by
\begin{align*}
\mathscr L_1:=\frac1{h_{E/F}c_S}\sum_\chi
\mathcal A^\times(\chi)\Lambda(1/2,\chi),\qquad
\mathscr L_2:=\frac{g_S}{h_{E/F}c_S^2}\sum_\chi
|\mathcal A^\times(\chi)\Lambda(1/2,\chi)|^2.
\end{align*}

\subsection{Mollified Moments}
\begin{lemma}[First moment]\label{lem9.mollified.first}
For fixed $0<\beta<1/28$ and $\varepsilon>0$,
\begin{multline}\label{eq9.mollified.first}
\mathscr L_1=
\frac{2^{d/2-1}\rho_F}{\ell^2}
\sum_{\mathfrak a\in\mathcal Q(Y)}\frac{b(\mathfrak a)}{N_F(\mathfrak a)}
\Big(V(\mathfrak a)+
\sum_{\mathfrak p\mid\mathfrak a}\frac{2\log q_{\mathfrak p}}{q_{\mathfrak p}-1}
+\sum_{\mathfrak p\in S}\frac{\log q_{\mathfrak p}}{q_{\mathfrak p}-1}\Big)
\\+O_{F,\varepsilon}\big(D_E^{-11/2+(31/2)\beta+\varepsilon}\big).
\end{multline}
In particular, as $D_E\to\infty$,
\begin{equation}\label{eq9.explicit.first}
\mathscr L_1\geq\bigl(2^{d/2-2}\rho_F+o_F(1)\bigr)\frac{S(Y)}\ell.
\end{equation}
\end{lemma}
\begin{proof}
Taking $m=10$ in \eqref{eq9.unit.first.center} and using
\eqref{eq9.nv.coeff} and $R_S^{23/2}\leq D_E^{23/32}\leq D_E^{3/4}$
bounds the summed error by
$O_{F,\varepsilon}(D_E^{-5+\varepsilon}R_S^\varepsilon Y^{31/2+\varepsilon})$.
Since $L(1,\eta_{E/F})/h_{E/F}\asymp_F D_E^{-1/2}$,
\eqref{eq9.ramified.relative} and \eqref{eq9.unit.cutoff} absorb the
remaining factors and establish the error in \eqref{eq9.mollified.first}.
For fixed $\beta<1/28$, it is $o_F(S(Y)/\ell)$, since $S(Y)\geq1$.

Substitute the definition of $a(\mathfrak q)$
and write $\mathfrak a=\mathfrak f\mathfrak q$. For $\mathfrak a\in\mathcal Q(Y)$ and $j\in\{0,1\}$, 
\begin{equation}\label{eq9.nv.divisorsum}
\sum_{\mathfrak f\mid\mathfrak a}
\mu_F(\mathfrak f)\alpha_0(\mathfrak f)\tau_F(\mathfrak f)
(\log N_F(\mathfrak f))^j
=\mu_F(\mathfrak a)\alpha_1(\mathfrak a)
\Big(2\log N_F(\mathfrak a)+
\sum_{\mathfrak p\mid\mathfrak a}
\frac{2\log q_{\mathfrak p}}{q_{\mathfrak p}-1}\Big)^j.
\end{equation}
Since $V(\mathfrak q)-2\log N_F(\mathfrak q)
=V(\mathfrak a)-2\log N_F(\mathfrak a)+\log N_F(\mathfrak f)$,
the two divisor sums combine to give
\begin{equation}\label{eq9.nv.firsttransform}
\sum_{\mathfrak q}\frac{a(\mathfrak q)(V(\mathfrak q)-2\log N_F(\mathfrak q))}{N_F(\mathfrak q)^{1/2}}=\sum_{\mathfrak a}\frac{\mu_F(\mathfrak a)\alpha_1(\mathfrak a)\xi_0(\mathfrak a)}{N_F(\mathfrak a)^{1/2}}\Big(V(\mathfrak a)+\sum_{\mathfrak p\mid\mathfrak a}\frac{2\log N_F(\mathfrak p)}{N_F(\mathfrak p)-1}\Big).
\end{equation}

The first-moment constant in Theorem~\ref{thm8.1} is
\begin{align*}
A_F=\frac{2^{1-d/2}\rho_F}{\pi^d\sqrt{D_F}}.
\end{align*}
For $j=0$, summing \eqref{eq9.nv.divisorsum} against
$\xi_0(\mathfrak a)/N_F(\mathfrak a)^{1/2}$ shows that
\begin{align*}
\sum_{\mathfrak q}\frac{a(\mathfrak q)}{N_F(\mathfrak q)^{1/2}}
=\frac{S(Y)}{2\ell^2}.
\end{align*}
Substituting this and \eqref{eq9.nv.firsttransform} into
\eqref{eq9.unit.first.center}, using
\begin{align*}
\frac{\mu_F(\mathfrak a)\alpha_1(\mathfrak a)\xi_0(\mathfrak a)}
{N_F(\mathfrak a)^{1/2}}
=\frac{b(\mathfrak a)}{2\ell^2N_F(\mathfrak a)}
\end{align*}
and \eqref{eq9.classnumber.exact}, proves \eqref{eq9.mollified.first}.
Its main term has nonnegative weights. Discarding the two prime sums
and using $V(\mathfrak a)\geq\ell/2+O_F(1)$ proves
\eqref{eq9.explicit.first}.
\end{proof}

Recall that $H_{\mathfrak{n}}(s_1,s_2)$ is defined by \eqref{e3.6} in \textsection\ref{sec3.6}.  Define
\begin{align*}
\widetilde H_{\mathfrak n}(z_1,z_2)
=N_F(\mathfrak n)^{-1/2-z_1}H_{\mathfrak n}(z_1,z_2).
\end{align*}

Let $b_{\mathfrak p}=\widetilde H_{\mathfrak p}$ and
$h_{\mathfrak p}=(1+\widetilde H_{\mathfrak p^2})/2-b_{\mathfrak p}^2$,
and extend multiplicatively to squarefree ideals.
For a split prime $\mathfrak p$, the local formula in \textsection\ref{3.2.1} specializes to
$H_{\mathfrak p^r}(0,0)=1+r(q-1)/(q+1)$, so
\begin{align*}
b_{\mathfrak p}(0,0)
=\frac{2\sqrt q}{q+1}=\frac{2\alpha_0(\mathfrak p)}{\sqrt q},\qquad
h_{\mathfrak p}(0,0)=\frac{(q-1)^3}{2q(q+1)^2}=h(\mathfrak p).
\end{align*}
For the shifted mollifier, define
\begin{align*}
\xi_{\mathfrak a}(z_1,z_2)
=\sum_{\substack{\mathfrak f:\ \mathfrak a\mathfrak f\in\mathcal Q(Y)}}
a(\mathfrak a\mathfrak f)b_{\mathfrak f}(z_1,z_2).
\end{align*}
and
\begin{equation}\label{eq9.nv.Q}
Q(x,y)=\frac{4\ell^4}{S(Y)}\sum_{\mathfrak a}h(\mathfrak a)
\xi_{\mathfrak a}(x/\ell,y/\ell)^2,\qquad
Q_{ij}=\partial_x^i\partial_y^jQ(0,0).
\end{equation}

\begin{lemma}\label{lem9.mollifier.derivatives}
For fixed $0<\beta<1/28$, the function $Q$ in \eqref{eq9.nv.Q}
satisfies $Q(0,0)=1$, $Q(x,y)=Q(y,x)$, and, as $D_E\to\infty$,
\begin{equation}\label{eq9.explicit.Qbounds}
\begin{cases}
0\leq Q_{10}\leq4\beta+o_F(1),\\
0\leq Q_{11}\leq12\beta^2+o_F(1),\\
-Q_{30}/2+3Q_{21}/2\leq 80\beta^3/3+o_F(1),
\end{cases}
\end{equation}
where the lower bounds allow an $o_F(1)$ error.
\end{lemma}
\begin{proof}
Substituting the definition of $a(\mathfrak a\mathfrak f)$ and collecting
divisors rewrites $\xi_{\mathfrak a}(z_1,z_2)$ in terms of $\xi_0$:
\begin{equation}\label{eq9.nv.shiftedxi}
\xi_{\mathfrak a}(z_1,z_2)
=\sum_{(\mathfrak l,\mathfrak a)=1}\xi_0(\mathfrak a\mathfrak l)
\prod_{\mathfrak p\mid\mathfrak l}
(b_{\mathfrak p}(z_1,z_2)-b_{\mathfrak p}(0,0)).
\end{equation}
In particular, $\xi_{\mathfrak a}(0,0)=\xi_0(\mathfrak a)$, so
$Q(0,0)=1$ by \eqref{eq9.nv.normalization}. The local factors are symmetric
in the two variables, hence so is $Q$.
For $\mathfrak a\in\mathcal Q(Y)$, write
$f(x,y)=\xi_{\mathfrak a}(x/\ell,y/\ell)/\xi_0(\mathfrak a)$.
For $\mathfrak p\nmid\mathfrak a$ with $\mathfrak a\mathfrak p\in\mathcal Q(Y)$,
the factor associated with $\mathfrak p$ in \eqref{eq9.nv.shiftedxi}, after division by $\xi_0(\mathfrak a)$, is
\begin{align*}
\frac{\xi_0(\mathfrak a\mathfrak p)}{\xi_0(\mathfrak a)}
\bigl(b_{\mathfrak p}(x/\ell,y/\ell)-b_{\mathfrak p}(0,0)\bigr)
=-\frac{2\sqrt q(q+1)}{(q-1)^2}
\bigl(b_{\mathfrak p}(x/\ell,y/\ell)-b_{\mathfrak p}(0,0)\bigr).
\end{align*}
By \eqref{eq3.Ridentities} and \eqref{eq3.9},
\begin{align*}
b_{\mathfrak p}(z_1,z_2)
=\frac{q^{-1/2}(q^{-z_1}+q^{-z_2})}{1+q^{-1-z_1-z_2}}.
\end{align*}
Thus, when taking any derivative of total order $j$, $1\leq j\leq3$, at $(0,0)$,
we may replace the preceding factor by
$-2q^{-1}(e^{-x\log q/\ell}+e^{-y\log q/\ell}-2)$,
with an error $O_j(q^{-2}(\log q/\ell)^j)$. These errors satisfy
\begin{align*}
\sum_{\mathfrak p}\frac{(\log q)^j}{q^2\ell^j}
\ll_{F,j}\ell^{-j}=o_F(1),\qquad 1\leq j\leq3.
\end{align*}
Each factor vanishes at $(0,0)$, so a derivative of order at most $3$
contains at most three prime factors. The prime-ideal Mertens estimate
bounds the sums over the other differentiated factors by $O_{F,\beta}(1)$.
Hence the replacement changes these derivatives of $f$ by $o_F(1)$,
uniformly in $\mathfrak a$.

For a product of $k\leq3$ distinct primes, the norm is at most $Y=D_E^\beta$, so
the prime-ideal theorem bounds the contribution from $k$ distinct
primes, each differentiated once, by
\begin{align*}
2^k\int_{\substack{t_i\geq0\\\sum t_i\leq\beta}}dt_1\cdots dt_k
=\frac{(2\beta)^k}{k!},\qquad 1\leq k\leq3,
\end{align*}
up to $o_F(1)$. At $(0,0)$, this bounds $f_x$ and $f_{xy}$ by
$2\beta+o_F(1)$ and $2\beta^2+o_F(1)$, respectively; both are
nonnegative up to $o_F(1)$. In $-f_{xxx}+3f_{xxy}$, the one-prime
terms are nonpositive, the two-prime terms cancel, and the three-prime
terms are at most $8\beta^3/3+o_F(1)$.

The weights $4\ell^4h(\mathfrak a)\xi_0(\mathfrak a)^2/S(Y)$ are nonnegative
and sum to $1$ by \eqref{eq9.nv.normalization}. Since $Q$ is the corresponding
average of $f^2$, the product rule proves the first two bounds. For the third, use
$f(0,0)=1$ and $f_x(0,0)=f_y(0,0)$ to write
\begin{align*}
-\tfrac12(f^2)_{xxx}+\tfrac32(f^2)_{xxy}
=-f_{xxx}+3f_{xxy}+6f_xf_{xy}
\leq\tfrac{80}{3}\beta^3+o_F(1)
\end{align*}
at $(0,0)$, and take the same average.
\end{proof}

For the second moment, write
\begin{align*}
P(\beta)=\frac34+12\beta+36\beta^2,\qquad
J(\beta)=\frac{80}3\beta^3,\qquad
\sigma=\sigma(\beta)=\frac12+\frac{J(\beta)}{P(\beta)}.
\end{align*}
Let $u=\ell^{-1}L(1,\eta_{E/F})^{-1}L'(1,\eta_{E/F}) $. Choose a fixed smooth function $0\leq\mathcal U\leq1$, equal to $1$
on $(0,1/2]$ and to $0$ on $[1,\infty)$.

\begin{lemma}[Second moment]\label{lem9.mollified.second}
For fixed $0<\beta<1/28$ and $\varepsilon>0$,
\begin{multline}\label{eq9.mollified.second}
\mathscr L_2=\frac{g_SL(1,\eta_{E/F})}{h_{E/F}c_S^2}
\sum_{\mathfrak q_1,\mathfrak q_2\in\mathcal Q(Y)}
\frac{a(\mathfrak q_1)a(\mathfrak q_2)}{2^{\omega_F(\mathfrak d)}}
\sum_{\mathfrak r\mid\mathfrak d}
\mathscr M_S^\heartsuit
(\mathfrak q_1\mathfrak q_2\mathfrak d^{-2}\mathfrak r^2)
\\+O_{F,\varepsilon}\big(D_E^{-1/12+(7/3)\beta+\varepsilon}\big),
\end{multline}
where $\mathfrak d=(\mathfrak q_1,\mathfrak q_2)$. Moreover, as $D_E\to\infty$,
\begin{equation}\label{eq9.nv.decorated}
\mathscr L_2\leq\frac{2^{d-2}\rho_F^3}{3\zeta_F(2)}\cdot 
\frac{L(1,\eta_{E/F})S(Y)}{\ell}\cdot 
\Big[P(\beta)(1/2+u)+J(\beta)+o_F(1+|u|)\Big].
\end{equation}
\end{lemma}
\begin{proof}
Apply \eqref{eq9.nv.hecke} and \eqref{eq9.unit.second}.
By \eqref{eq9.nv.coeff}, the summed error before normalization is
$O_{F,\varepsilon}(D_E^{5/12+\varepsilon}R_S^\varepsilon Y^{7/3+\varepsilon})$.
Using \eqref{eq9.classnumber.exact}, \eqref{eq9.ramified.relative}, and
\eqref{eq9.unit.cutoff} completes the proof of \eqref{eq9.mollified.second}.
Choose $\varepsilon>0$ with $2\varepsilon<1/12-7\beta/3$.
By $S(Y)\geq1$ and Siegel's lower bound
$L(1,\eta_{E/F})\gg_{F,\varepsilon}D_E^{-\varepsilon}$, the error divided
by $L(1,\eta_{E/F})S(Y)/\ell$ is
\begin{align*}
O_{F,\varepsilon}\big(\ell D_E^{-1/12+(7/3)\beta+2\varepsilon}\big)=o_F(1).
\end{align*}

We now evaluate the main term in \eqref{eq9.mollified.second}.

For $\epsilon=(\epsilon_1,\epsilon_2)\in\{\pm1\}^2$, write
$z_i=\epsilon_i s_i$ for $i=1,2$.
After substituting \eqref{eq4.Z}, each term in \eqref{eq7.signterm} takes the form
\begin{multline}\label{eq9.mollified.branch}
\mathscr M_{\epsilon_1,\epsilon_2}(\mathbf s;\mathfrak n)
=\frac{2\sqrt{D_E}}{\pi^d\sqrt{D_F}}
D_E^{-(1-\epsilon_1)s_1/2-(1-\epsilon_2)s_2/2}
\\\times
\frac{\Lambda_F(1+2z_1)\Lambda_F(1+2z_2)\Lambda_E(1+z_1+z_2)}
{\Lambda_F(2+2z_1+2z_2)}
\widetilde H_{\mathfrak n}(z_1,z_2).
\end{multline}
Thus all dependence on $\mathfrak n$ is in $\widetilde H_{\mathfrak n}(z_1,z_2)$;
the ramified factors $\prod_{v\in S}Q_{v,\epsilon}(\mathbf s)$ in
\eqref{eq9.unit.main} are independent of $\mathfrak n$.

Expand $\sum_{\mathfrak a}h_{\mathfrak a}\xi_{\mathfrak a}^2$ using the
definition of $\xi_{\mathfrak a}$. For each ordered pair $(\mathfrak q_1,\mathfrak q_2)$, summing
over $\mathfrak a\mid\mathfrak d=(\mathfrak q_1,\mathfrak q_2)$ produces the factor
\begin{align*}
b_{\mathfrak q_1\mathfrak q_2\mathfrak d^{-2}}(z_1,z_2)
\prod_{\mathfrak p\mid\mathfrak d}
\bigl(b_{\mathfrak p}(z_1,z_2)^2+h_{\mathfrak p}(z_1,z_2)\bigr).
\end{align*}
Since $b_{\mathfrak p}^2+h_{\mathfrak p}
=(1+\widetilde H_{\mathfrak p^2})/2$, expanding this product proves
\begin{equation}\label{eq9.nv.shifteddiagonal}
\sum_{\mathfrak a}h_{\mathfrak a}(z_1,z_2)
\xi_{\mathfrak a}(z_1,z_2)^2=\sum_{\mathfrak q_1,\mathfrak q_2\in\mathcal Q(Y)}
\frac{a(\mathfrak q_1)a(\mathfrak q_2)}{2^{\omega_F(\mathfrak d)}}
\sum_{\mathfrak r\mid\mathfrak d}
\widetilde H_{\mathfrak q_1\mathfrak q_2\mathfrak d^{-2}\mathfrak r^2}(z_1,z_2),
\end{equation}
where $\mathfrak d=(\mathfrak q_1,\mathfrak q_2)$.

Use \eqref{eq9.nv.shifteddiagonal} to sum \eqref{eq9.mollified.branch}
against the mollifier coefficients in \eqref{eq9.mollified.second}.
Multiplying by the ramified factors from \eqref{eq9.unit.main} and
normalizing by \eqref{eq9.classnumber.exact}, the main term in
\eqref{eq9.mollified.second} is the joint constant term at
$(s_1,s_2)=(0,0)$ of the following expression, with the four sign
choices summed first, as in \eqref{eq7.finitepart}:
\begin{multline}\label{eq9.mollified.main}
\frac{2^dg_S}{c_S^2}
\sum_{\epsilon_1,\epsilon_2=\pm1}
D_E^{-\frac{(1-\epsilon_1)s_1+(1-\epsilon_2)s_2}{2}}\cdot 
\frac{\Lambda_F(1+2z_1)\Lambda_F(1+2z_2)\Lambda_E(1+z_1+z_2)}
{\Lambda_F(2+2z_1+2z_2)}
\\\times\prod_{v\in S}Q_{v,\epsilon}(\mathbf s)
\sum_{\mathfrak a}h_{\mathfrak a}(z_1,z_2)\xi_{\mathfrak a}(z_1,z_2)^2,
\end{multline}
with $z_i=\epsilon_i s_i$.

Each factor in the product in \eqref{eq9.nv.shiftedxi} vanishes at $(0,0)$. Thus a
derivative of order $j$ at the center receives contributions only from
$\omega_F(\mathfrak l)\leq j$. Write $\partial^j$ for any derivative of total order $j$ in $z_1,z_2$.
For $1\leq j\leq3$, the local formula in \textsection\ref{3.2.1}
implies the derivative bounds
\begin{align*}
|\partial^j b_{\mathfrak p}(0,0)|\ll_j b_{\mathfrak p}(0,0)(\log q)^j,\qquad \partial^j\log h_{\mathfrak p}(0,0)=O_j((\log q)^j/q).
\end{align*}
Together with the definition of $\xi_0$, the prime-ideal Mertens
estimate bounds the derivatives by
\begin{align*}
|\partial^j\xi_{\mathfrak a}(0,0)|\ll_{F,j}(\log Y)^j|\xi_0(\mathfrak a)|.
\end{align*}

The derivatives of $h_{\mathfrak a}/h(\mathfrak a)$ are bounded by fixed
powers of $\log\log(3N_F(\mathfrak a))$. At a ramified prime, the
explicit factor in \eqref{eq9.unit.branch} satisfies, with $q=q_v$ and
$1\leq j\leq3$,
\begin{align*}
Q_{v,\epsilon}(\mathbf0)=\frac{(1-q^{-1})^2}{1+q^{-1}},\qquad
\partial^j Q_{v,\epsilon}(\mathbf0)
\ll_j\frac{(\log q)^j}q.
\end{align*}
Thus, for every sign branch and $1\leq j\leq3$,
\begin{align*}
\prod_{v\in S}Q_{v,\epsilon}(\mathbf0)=\frac{c_S^2}{g_S},\qquad
\Big|\partial^j\prod_{v\in S}Q_{v,\epsilon}(\mathbf0)\Big|
\ll_{F,j}\frac{c_S^2}{g_S}\cdot (\log z)^j.
\end{align*}

When extracting the constant term in \eqref{eq9.mollified.main},
we may replace $h_{\mathfrak a}(z_1,z_2)$ by
$h(\mathfrak a)$ and the ramified product by $c_S^2/g_S$.
The omitted derivatives have size $o_F(\ell^j)$ at order $j$ and,
using the bounds for the derivatives of $L(s,\eta_{E/F})$ below, contribute
$o_F(L(1,\eta_{E/F})S(Y)(1+|u|)/\ell)$.
The ramified scalar cancels $g_S/c_S^2$ in \eqref{eq9.mollified.main}.

We next bound the derivatives $L_j=L^{(j)}(1,\eta_{E/F})$,
$0\leq j\leq3$. For $\Re s>1$, write
\begin{align*}
\zeta_F(s)L(s,\eta_{E/F})
=\sum_{\mathfrak a}\frac{r(\mathfrak a)}{N_F(\mathfrak a)^s},
\qquad
r(\mathfrak a):=\sum_{\mathfrak c\mid\mathfrak a}
\eta_{E/F}(\mathfrak c)\geq0.
\end{align*}
The coefficients are nonnegative since the product is $\zeta_E(s)$.
Define
\begin{align*}
\mathscr S_j(x)=\sum_{\mathfrak a}\frac{r(\mathfrak a)}{N_F(\mathfrak a)}\mathcal U(N_F(\mathfrak a)/x)(\log(x/N_F(\mathfrak a)))^j,\qquad 0\leq j\leq2.
\end{align*}
With $t=\log x$, the nonnegative summands satisfy
$0\leq \mathscr S_{j+1}(x)\leq t\mathscr S_j(x)$
for $j=0,1$.
In Mellin inversion, shifting the contour to
$\Re s=-1/2+\varepsilon$ leaves an error of size
$O_{F,\varepsilon}(D_E^{1/4+\varepsilon}x^{-1/2+\varepsilon})$.
Up to the common positive factor $\operatorname{Res}_{s=1}\zeta_F(s)$,
the leading residue polynomials for $\mathscr S_0,\mathscr S_1,\mathscr S_2$ are respectively
$L_0t+L_1$, $L_0t^2/2+L_1t+L_2/2$, and
$L_0t^3/3+L_1t^2+L_2t+L_3/3$; their lower-degree terms involve only fixed
Laurent coefficients.

Write $u=L_1/(L_0\ell)$, $u_2=L_2/(L_0\ell^2)$, and $u_3=L_3/(L_0\ell^3)$. Apply positivity at
$x=D_E^{1/2+\delta}$, with $\delta>0$ fixed and arbitrarily small.
The inequalities $\mathscr S_0\geq0$, $0\leq \mathscr S_1\leq t\mathscr S_0$, and
$\mathscr S_2\leq t\mathscr S_1$ imply, in the limit $\delta\to0$,
\begin{equation}\label{eq9.nv.derivatives}
\begin{cases}
u\geq-1/2+o_F(1),\\
-1/4-u+o_F(1+|u|)\leq u_2\leq 1/4+o_F(1+|u|),\\
u_3\leq 1/16-3u_2/4+o_F(1+|u|).
\end{cases}
\end{equation}
Siegel's bound absorbs the contour errors. One may keep $\delta$ fixed
throughout: the strict numerical margin at the end permits a sufficiently
small positive choice.

After summing the four terms in \eqref{eq9.mollified.main},
we compute the constant term along
$(s_1,s_2)=(t/\ell,2t/\ell)$.
This choice ensures that $\epsilon_1+2\epsilon_2\neq0$
for every sign pair. For fixed $E/F$,
\begin{align*}
\lim_{t\to0}\frac{t^3\Lambda_F(1+2\epsilon_1t/\ell)\Lambda_F(1+4\epsilon_2t/\ell)\Lambda_E(1+(\epsilon_1+2\epsilon_2)t/\ell)}{\Lambda_F(2+2(\epsilon_1+2\epsilon_2)t/\ell)}=\frac{\rho_F^3L_0\ell^3}{8\zeta_F(2)\epsilon_1\epsilon_2(\epsilon_1+2\epsilon_2)}.
\end{align*}
We have used $\Lambda_{E_\infty}(1)=\pi^{-d}$ and
$\Lambda_F(2)=\pi^{-d}\zeta_F(2)$.
Expand $L(1+(\epsilon_1+2\epsilon_2)t/\ell,\eta_{E/F})/L_0$
through degree $3$. The other regular factors contribute only to the
error, so
\begin{multline}\label{eq9.mollified.coefficient}
\mathscr L_2=\frac{2^{d-5}\rho_F^3L_0S(Y)}{\zeta_F(2)\ell}
[t^3]\sum_{\epsilon_1,\epsilon_2=\pm1}
\frac{e^{-(3-\epsilon_1-2\epsilon_2)t/2}
Q(\epsilon_1t,2\epsilon_2t)}
{\epsilon_1\epsilon_2(\epsilon_1+2\epsilon_2)}
\\\times\Big(1+u(\epsilon_1+2\epsilon_2)t
+\frac{u_2}2(\epsilon_1+2\epsilon_2)^2t^2
+\frac{u_3}6(\epsilon_1+2\epsilon_2)^3t^3\Big)
+o_F\left(\frac{L_0S(Y)}\ell(1+|u|)\right),
\end{multline}
where $[t^3]$ denotes the coefficient of $t^3$.
Expanding the coefficient in \eqref{eq9.mollified.coefficient} and
combining terms by the symmetry $Q(x,y)=Q(y,x)$, we obtain
\begin{multline*}
\mathscr L_2=\frac{2^{d-2}\rho_F^3}{3\zeta_F(2)}\cdot \frac{L_0S(Y)}{\ell}\cdot  
\bigg\{\frac18+\frac34u+\frac32u_2+u_3
+3\left(\frac14+u+u_2\right)Q_{10}\\
+3\left(\frac12+u\right)Q_{11}
-\frac12Q_{30}+\frac32Q_{21}+o_F(1+|u|)\bigg\}.
\end{multline*}

By \eqref{eq9.nv.derivatives}, the coefficients of $Q_{10},Q_{11}$
are nonnegative up to
$o_F(1+|u|)$. Applying Lemma~\ref{lem9.mollifier.derivatives} to
the expression for $\mathscr L_2$ proves \eqref{eq9.nv.decorated}.
\end{proof}

\subsection{The Nonvanishing Proportion}
We obtain the absolute constant
\begin{equation}\label{eq9.kappaF}
\kappa_0=4.9\times10^{-4}.
\end{equation}

\begin{prop}\label{lem9.1}\label{prop.explicit.basefield}
For a suitable fixed $\beta<1/28$, and for sufficiently large $D_E$, 
\begin{equation}\label{eq9.nv.ratio}
\frac{\left|h_{E/F}^{-1}\sum_\chi
\mathcal A^\times(\chi)\Lambda(1/2,\chi)\right|^2}
{h_{E/F}^{-1}\sum_\chi|\mathcal A^\times(\chi)\Lambda(1/2,\chi)|^2}
\geq\kappa_0.
\end{equation}
Here the sums are over $\widehat{\mathrm{Cl}}_{E/F}$.
\end{prop}

\begin{proof}
By the definitions of $\mathscr L_1$ and $\mathscr L_2$, the ratio in
\eqref{eq9.nv.ratio} is $g_S|\mathscr L_1|^2/\mathscr L_2$.
For fixed $\beta>0$, we have $\sigma+u\gg_\beta1$ by
\eqref{eq9.nv.derivatives}. Lemma~\ref{lem9.mollified.second} therefore implies
\begin{align*}
\mathscr L_2\leq\left(\frac{2^{d-2}\rho_F^3}{3\zeta_F(2)}P(\beta)
+o_F(1)\right)\frac{L(1,\eta_{E/F})S(Y)}\ell(\sigma+u).
\end{align*}
It remains to bound $S(Y)$ from below. The Dirichlet series of $b$ is
\begin{equation}\label{eq9.nv.euler}
\sum_{\mathfrak a}\frac{b(\mathfrak a)}{N_F(\mathfrak a)^s}
=\zeta_F(s)L(s,\eta_{E/F})
\prod_{\mathfrak p\mid\mathfrak D_{E/F}}
(1-N_F(\mathfrak p)^{-s})\,\mathcal H_{E/F}(s).
\end{equation}
For $\mathfrak p\nmid\mathfrak D_{E/F}$, the local factors of
$\mathcal H_{E/F}(s)$ are
\begin{align*}
\mathcal H_{E/F,\mathfrak p}(s)=
\begin{cases}
\displaystyle
\left(1+\frac{2N_F(\mathfrak p)^{1-s}}
{N_F(\mathfrak p)-1}\right)
\left(1-N_F(\mathfrak p)^{-s}\right)^2,
& \eta_{E/F}(\mathfrak p)=1,\\[6pt]
1-N_F(\mathfrak p)^{-2s},
& \eta_{E/F}(\mathfrak p)=-1.
\end{cases}
\end{align*}
Hence $\mathcal H_{E/F}$ is holomorphic and uniformly bounded on
$\Re(s)\geq1/2+\varepsilon$, with
$\mathcal H_{E/F}(1)\asymp_F1$ and
$\mathcal H'_{E/F}(1)/\mathcal H_{E/F}(1)=O_F(1)$.
Both displayed factors equal $1-N_F(\mathfrak p)^{-2}$ at $s=1$;
ramified primes contribute $1$.

Apply Mellin inversion to \eqref{eq9.nv.euler} with the same contour
shift as above. Since $0\leq\mathcal U\leq1$,
\begin{align*}
S(D_E^\sigma)
&\geq\sum_{\mathfrak a}\frac{b(\mathfrak a)}{N_F(\mathfrak a)}
\mathcal U\left(\frac{N_F(\mathfrak a)}{D_E^\sigma}\right)\\
&\geq(1+o_F(1))\rho_Fc_{E/F}\mathcal H_{E/F}(1)L(1,\eta_{E/F})\ell(\sigma+u).
\end{align*}
The ramified logarithmic derivative is nonnegative; the remaining
fixed terms and contour error are negligible. By Lemma~\ref{lem9.unit.cutoff},
\begin{equation}\label{eq9.explicit.local}
c_{E/F}g_S\mathcal H_{E/F}(1)
=\frac1{\zeta_F(2)}
\prod_{\substack{\mathfrak p\mid\mathfrak D_{E/F}\\q_{\mathfrak p}>z}}
(1+q_{\mathfrak p}^{-1})^{-1}
=\frac{1+o_F(1)}{\zeta_F(2)}.
\end{equation}
Combining this bound and the preceding estimate for $S(D_E^\sigma)$
with \eqref{eq9.shortlong}, \eqref{eq9.explicit.first}, and
\eqref{eq9.nv.decorated}, we obtain
\begin{equation}\label{eq9.explicit.ratio}
g_S\frac{|\mathscr L_1|^2}{\mathscr L_2}
\geq\frac{3e^{-2\gamma}\beta^2}{8\sigma(\beta)^2P(\beta)}+o_F(1).
\end{equation}
For $\beta=35714/10^6<1/28$, the numerical factor is
\begin{align*}
\frac{3e^{-2\gamma}\beta^2}{8\sigma(\beta)^2P(\beta)}
=4.906024918\ldots\times10^{-4}>\kappa_0.
\end{align*}
The strict margin absorbs the arbitrarily small fixed loss in
Lemma~\ref{lem9.mollified.second} and all $o_F(1)$ terms, proving \eqref{eq9.nv.ratio} with \eqref{eq9.kappaF}.
\end{proof}

\begin{proof}[Proof of Theorem~\ref{thmC}]
Cauchy--Schwarz bounds the nonvanishing proportion from below by the
left-hand side of \eqref{eq9.nv.ratio}. Proposition~\ref{lem9.1}
therefore proves the theorem with the explicit choice \eqref{eq9.kappaF}.
\end{proof}

\begin{remark}
The threshold $D_0(F)$ is ineffective in this proof: absorbing the errors
in \eqref{eq9.mollified.second} and the Mellin contour shifts uses
Siegel's lower bound for $L(1,\eta_{E/F})$. Thus the explicit proportion
$\kappa_0$ does not provide a computable discriminant threshold.
For $F=\mathbb Q$, Blomer estimated a proportion of order $10^{-6}$
without computing it explicitly \cite[p.~832]{Blo04}.
\end{remark}

\begin{remark}
The summed second-moment error permits $Y=D_E^\beta$ for every
$\beta<1/28$, compared with $\beta<1/364$ in \cite[(4.9)]{Blo04}.
The factor $R_S^\varepsilon$ does not change this range.
\end{remark}

\section{Conditional Subconvexity}\label{sec10}
The twisted second moment also supports a Duke--Friedlander--Iwaniec-type amplified subconvexity argument. As in \cite[\textsection 15--\textsection 16]{DFI95}, the arithmetic input is a supply of small split primes from which to build the amplifier.

For $\beta>0$ we define
\begin{align*}
\mathcal{S}_{\beta}:=\big\{\text{prime $\mathfrak{p}\subset\mathcal{O}_F$ :\ split in $E$},\ \mathfrak{p}\nmid 2\mathfrak{D}_F,\  N_F(\mathfrak{p})\leq D_E^{\beta}\big\}.
\end{align*}

The expected uniform Chebotarev asymptotic in the discriminant aspect predicts $\#\mathcal{S}_{\beta}=D_E^{\beta-o(1)}$. The polynomial lower bound required below remains open in this generality; this is the number-field analogue of the small-split-prime obstruction in \cite[\textsection 15]{DFI95}.

\begin{thm}\label{thm10.subconvexity}
Let $0<\delta<1$, put $\beta=\frac{1}{4(11-3\delta)}$, and suppose that, uniformly in a family of quadratic extensions $E/F$ as $D_E\to\infty$,
\begin{align*}
\#\mathcal{S}_{\beta}\gg_{F,\delta}D_E^{(1-\delta)\beta}.
\end{align*}
Then, for every sufficiently large $E/F$ in that family, every $\chi\in \widehat{\mathcal{C}}_{E/F}^{\text{un}}$, and every $\varepsilon>0$,
\begin{align*}
|L(1/2,\chi)|\ll D_E^{\frac{1}{4}-\frac{1-\delta}{8(11-3\delta)}+\varepsilon},
\end{align*}
where the implied constant depends on $F$, $\delta$, $\varepsilon$, the Archimedean component $\chi_\infty$, and the implied constant in the assumed lower bound for $\#\mathcal{S}_{\beta}$.
\end{thm}

\begin{proof}
Put $X=D_E^\beta$ and fix $\xi\in\widehat{\mathcal C}_{E/F}^{\mathrm{un}}$. For each $\mathfrak{p}\in \mathcal{S}_{\beta}$, there exists $\ell_{\mathfrak{p}}\in \{1,2\}$, depending on $\xi$, such that
\begin{equation}\label{9.3}
\big|\lambda_{\xi}^*(\mathfrak{p}^{\ell_{\mathfrak{p}}})\big|\gg 1. 
\end{equation}
Write $\lambda_\xi^*(\mathfrak p)=\cos\theta_{\mathfrak p}$
and use $\max(|\cos\theta_{\mathfrak p}|,|\cos2\theta_{\mathfrak p}|)\geq1/2$.

Set $c_{\xi}(\mathfrak{p}):=\overline{\lambda_{\xi}^*(\mathfrak{p}^{\ell_{\mathfrak{p}}})}=\lambda_{\xi}^*(\mathfrak{p}^{\ell_{\mathfrak{p}}})$. Consider 
\begin{align*}
\mathcal{J}_{\xi}(X):=\sum_{\chi\in \widehat{\mathcal{C}}_{E/F}^{\text{un}}}\frac{|\Lambda(1/2,\chi)|^2}{L(1,\eta_{E/F})}\cdot\bigg|\sum_{\mathfrak{p}\in \mathcal{S}_{\beta}}c_{\xi}(\mathfrak{p})\lambda_{\chi}^*(\mathfrak{p}^{\ell_{\mathfrak{p}}})\bigg|^2.
\end{align*}

The term $\chi=\xi$ alone implies
\begin{equation}\label{9.5}
\mathcal{J}_{\xi}(X)\gg |\mathcal{S}_{\beta}|^2\cdot \frac{|\Lambda(1/2,\xi)|^2}{L(1,\eta_{E/F})}.	
\end{equation}

Expanding the square and using the Hecke relations
\begin{align*}
\begin{cases}
\lambda_{\chi}^*(\mathfrak{p}^{\ell_{\mathfrak{p}}})^2=2^{-1}\lambda_{\chi}^*(\mathfrak{p}^{2\ell_{\mathfrak{p}}})+2^{-1},\\
\lambda_{\chi}^*(\mathfrak{p}^{\ell_{\mathfrak{p}}})\lambda_{\chi}^*(\mathfrak{q}^{\ell_{\mathfrak{q}}})=\lambda_{\chi}^*(\mathfrak{p}^{\ell_{\mathfrak{p}}}\mathfrak{q}^{\ell_{\mathfrak{q}}}),\quad \mathfrak{p}\neq \mathfrak{q}. 
\end{cases}	
\end{align*}
we obtain 
\begin{align*}
\mathcal{J}_{\xi}(X)=\mathcal{J}_{\xi}^{\text{diag}}(X)+\mathcal{J}_{\xi}^{\text{off-diag}}(X),
\end{align*}
where
\begin{align*}
&\mathcal{J}_{\xi}^{\text{diag}}(X):=2^{-1}\sum_{\mathfrak{p}\in \mathcal{S}_{\beta}}|c_{\xi}(\mathfrak{p})|^2\sum_{\chi\in \widehat{\mathcal{C}}_{E/F}^{\text{un}}}\frac{|\Lambda(1/2,\chi)|^2}{L(1,\eta_{E/F})}\big(\lambda_{\chi}^*(\mathfrak{p}^{2\ell_{\mathfrak{p}}})+1\big),\\
&\mathcal{J}_{\xi}^{\text{off-diag}}(X):=\sum_{\substack{\mathfrak{p},\mathfrak{q}\in \mathcal{S}_{\beta}\\
\mathfrak{p}\neq \mathfrak{q}}}c_{\xi}(\mathfrak{p})c_{\xi}(\mathfrak{q})\sum_{\chi\in \widehat{\mathcal{C}}_{E/F}^{\text{un}}}\frac{|\Lambda(1/2,\chi)|^2}{L(1,\eta_{E/F})}\lambda_{\chi}^*(\mathfrak{p}^{\ell_{\mathfrak{p}}}\mathfrak{q}^{\ell_{\mathfrak{q}}}).
\end{align*}

For sufficiently large $D_E$, write $S=|\mathcal S_\beta|\geq1$.
The twisting ideals above satisfy $\omega_F(\mathfrak n)\leq2$ and
$e_v(\mathfrak n)\leq4$, so Lemma~\ref{lem7.main.bound} applies.
Since $|c_\xi(\mathfrak p)|\leq1$ and $\ell_{\mathfrak p}\leq2$,
Corollary~\ref{cor7.2} bounds the diagonal contribution by
\begin{equation}\label{eq9.6}
\mathcal{J}_{\xi}^{\text{diag}}(X)\ll S D_E^{1/2+\varepsilon}+S D_E^{5/12+\varepsilon}X^{8/3+\varepsilon}.
\end{equation}
For the off-diagonal main term, apply Cauchy--Schwarz:
\begin{align*}
\Big(\sum_{\mathfrak p\in\mathcal S_\beta}N_F(\mathfrak p)^{-\ell_{\mathfrak p}/2+\varepsilon}\Big)^2\leq S\sum_{N_F(\mathfrak p)\leq X}N_F(\mathfrak p)^{-1+2\varepsilon}\ll_{F,\varepsilon}SX^{2\varepsilon}.
\end{align*}
The ideals $\mathfrak p^{\ell_{\mathfrak p}}\mathfrak q^{\ell_{\mathfrak q}}$
have norm at most $X^4$, and there are at most $S^2$ pairs.
By \eqref{eq7.main.bound} and Corollary~\ref{cor7.2}, after renaming
$\varepsilon$,
\begin{equation}\label{eq9.7}
|\mathcal{J}_{\xi}^{\text{off-diag}}(X)|\ll S X^{\varepsilon}D_E^{1/2+\varepsilon}+S^2D_E^{5/12+\varepsilon}X^{8/3+\varepsilon}.
\end{equation}

Combining \eqref{9.5}, \eqref{eq9.6}, and \eqref{eq9.7}, and using
$L(1,\eta_{E/F})\ll_{F,\varepsilon}D_E^\varepsilon$ and $S\leq S^2$, we obtain
\begin{equation}\label{9.8}
|\Lambda(1/2,\xi)|^2 \ll_{F,\varepsilon} |\mathcal{S}_{\beta}|^{-1}\cdot X^{\varepsilon}D_E^{1/2+\varepsilon}+D_E^{5/12+\varepsilon}X^{8/3+\varepsilon}.	
\end{equation}
Since $\Lambda(1/2,\xi)=L_\infty(1/2,\xi)L(1/2,\xi)$ and the Archimedean
factor is nonzero, the same estimate holds with $L(1/2,\xi)$ in place of
$\Lambda(1/2,\xi)$, with an implied constant depending additionally on
$\xi_\infty$.

The split-prime hypothesis makes the two exponents in \eqref{9.8}
$1/2-(1-\delta)\beta$ and $5/12+\frac83\beta$, up to $\varepsilon$.
They agree for our choice of $\beta$. Taking square roots, with the
internal $\varepsilon$ sufficiently small, proves the theorem.
\end{proof}

\begin{remark}
With the expected supply $\#\mathcal S_\beta=D_E^{\beta-o(1)}$, the two
terms in \eqref{9.8} are balanced at $\beta=1/44$, giving a saving
approaching $1/88$. This is the limiting saving from the moment estimate
and amplifier used here.
\end{remark}

\begin{remark}
Duke--Friedlander--Iwaniec \cite[Theorems~4--6]{DFI95} amplify a second
moment over the class group characters of an imaginary quadratic field,
a dihedral family of size $D_E^{1/2+o(1)}$. Their Theorem~6 gives the
explicit saving $1/1156$ under conditions ensuring enough small split
primes. For a cuspidal automorphic induction $\pi$ in this family, both
the level and the conductor of its quadratic central character
$\eta_{E/\mathbb Q}$ equal $D_E$, and
$\pi\otimes\eta_{E/\mathbb Q}\simeq\pi$.
Thus this is already a subconvexity problem in which the central character
has conductor as large as the level.

Their later work \cite{DFI02} enlarges the family to Maass forms of level
$D$ with a fixed primitive central character and uses an amplified fourth
moment; see \cite[Theorem~2.2 and (2.5)]{DFI02}. It gives the unconditional
saving $1/23041$ for these forms and for class group $L$-functions of both
real and imaginary quadratic fields \cite[Theorems~2.4 and~2.6]{DFI02}.
The symmetric spectral reciprocity of \cite{Yan26a} develops this
fourth-moment approach over arbitrary number fields and at general levels.
For quadratic central characters, \cite[Theorem~1.8]{Yan26a} gives
\begin{align*}
|L(1/2,\chi)|\ll_{F,\chi_\infty,\varepsilon}
D_E^{\frac14-\frac1{80}+\varepsilon}
\end{align*}
when $\chi\in\widehat{\mathcal C}_{E/F}^{\mathrm{un}}$ has cuspidal
automorphic induction. This bound requires no split-prime hypothesis.
The present Fourier spectral reciprocity instead generalizes the second-moment
approach of \cite{DFI95} and improves its conditional saving.
\end{remark}

\bibliographystyle{alpha}
\bibliography{LY}

\end{document}